\documentclass[12pt]{amsart}

\usepackage{amssymb,amsxtra,stmaryrd,xspace}
\usepackage{enumerate}
\usepackage{multicol}
\usepackage{times}  

\newcommand\xref[1]{\csname#1\endcsname}

\newcommand\tdotsc{\ldots}
\newcommand\loccit{\emph{loc.{} cit.}\xspace}

\newcommand\mysigma{\eta}
\newcommand\mytau{\eta_0}

\let\Ind\relax
\DeclareMathOperator{\Ind}{Ind}
\DeclareMathOperator{\depth}{d}
\DeclareMathOperator{\ad}{ad}
\DeclareMathOperator{\Ad}{Ad}
\DeclareMathOperator{\supp}{supp}
\DeclareMathOperator{\stab}{stab}
\DeclareMathOperator{\Gal}{Gal}
\DeclareMathOperator{\Lie}{Lie}
\DeclareMathOperator{\Int}{Int}
\DeclareMathOperator{\meas}{meas}
\DeclareMathOperator{\ord}{ord}
\DeclareMathOperator{\GL}{GL}
\DeclareMathOperator\PGL{PGL}
\DeclareMathOperator\SL{SL}
\DeclareMathOperator{\SP}{Sp} 
\DeclareMathOperator{\sgn}{sgn}
\DeclareMathOperator{\Res}{Res}
\DeclareMathOperator{\tr}{tr}
\DeclareMathOperator\chr{char}

\newcommand{\me}[1]{{\protect\ensuremath{#1}}\xspace} 

\newcommand{\mc}[1]{\me{\mathcal #1}}
\gdef\AA(#1,#2){\fAA(#1)} 
\newcommand\fAA{\mc A} 
\newcommand{\BB}{\mc B}

\newcommand{\fQ}{\mc Q}
\newcommand{\fB}{\mc B}

\newcommand{\mf}[1]{\me{\protect\mathfrak{#1}}}
\newcommand{\ff}{\mf f}
\newcommand{\fg}{\mf g}

\newcommand{\mo}[1]{\me{\protect\mathbf #1}}
\newcommand{\bG}{\mo G}
\newcommand{\bH}{\mo H}
\newcommand{\bL}{\mo L}
\newcommand{\bS}{\mo S}
\newcommand{\bT}{\mo T}
\newcommand{\bM}{\mo M}
\newcommand{\bU}{\mo U}
\newcommand{\bX}{\mo X}
\newcommand{\wtilde}[1]{\widetilde{#1}}
\newcommand{\vbG}{\me{\vec\bG}}
\newcommand{\vG}{\me{\vec G}}

\newcommand{\mb}[1]{\me{\mathbb #1}}
\newcommand{\F}{\mb F}
\newcommand{\R}{\mb R}
\newcommand{\C}{\mb C}
\newcommand{\Z}{\mb Z}

\newcommand{\tR}{\me{\wtilde\R}}

\newcommand{\inv}{^{-1}}
\newcommand{\phplus}{{\phantom{+}}}
\newcommand{\cross}{^\times}
\newcommand{\conn}{^\circ}
\newcommand{\textsup}[1]{^{\protect\mathrm{#1}}}
\newcommand{\red}{\textsup{red}}
\newcommand{\sss}{\textsup{ss}}
\newcommand{\unram}{\textsup{un}}
\newcommand{\sep}{\textsup{sep}}
\newcommand{\rBB}{\BB\red}
\newcommand{\lsup}[1]{{}^{#1}}
\newcommand{\lsub}[1]{{}_{#1}}
\newcommand\card[1]{\left|#1\right|}
\newcommand\smcard[1]{\card{\smash{#1}}}

\newcommand\abs[1]{\left|#1\right|}
\newcommand\smabs[1]{\abs{\smash{#1}}}
\newcommand{\simm}[1]{\overset{#1}{\sim}}

\newcommand{\CC}[2]{{C_{#1}^{(#2)}}}
\newcommand{\ZZ}[2]{{Z_{#1}^{(#2)}}}

\newcommand{\mexp}{\me{\mathsf e}}

\newcommand{\ol}[1]{\me{\protect\overline{#1}}}
\newcommand{\ox}{{\ol x}\xspace}
\newcommand{\ul}[1]{\protect\underline{#1}}
\newcommand{\ugamma}{\me{\ul{\gamma}}}
\newcommand{\dpsi}{\me{\dot\psi}}

\newcommand\trans[1]{#1\textsup t}

\newcommand{\odc}[1]{{\llbracket #1\rrbracket}}
\newcommand{\dc}{\odc{\gamma; x, r}}

\renewcommand{\to}{\longrightarrow}
\newcommand{\ap}[1]{\me{#1^\perp}}
\newcommand{\gperp}{\ap g}
\newcommand{\iperp}{\ap i}
\newcommand{\jperp}{\ap j}
\newcommand{\tperp}{\ap t}
\newcommand\cpt{\mc K}
\newcommand\spindx[4]{#1#3 : #4#2}
\newcommand\indx[2]{\spindx[]{#1}{#2}}
\newcommand\Bigindx[2]{\spindx{\Bigl[}{\Bigr]}{#1}{#2}}
\newcommand\dotm{\nolinebreak\cdot\nolinebreak} 
\newcommand\dota{\nolinebreak\cdot\nolinebreak} 

\newcounter{incenumi} 

\newcounter{tagenumi} 

\newenvironment{inc_enumerate}{\begin{list}{(\arabic{incenumi})}{\usecounter{incenumi}}}{\end{list}} 

\newcommand{\set}[2]{ 
	{\left\{\left.
	#1\vphantom{#2\bigl(\bigr)}\,\right|
	\,#2\right\}}}

\makeatletter
\newtoks\sset@tok
\newcommand{\sset}[1]{\sset@tok={#1}\futurelet\sset@temp\sset@action}
\def\sset@witharg[#1]{\ensuremath{\left\{\the\sset@tok\right\}_{#1}}}
\def\sset@withoutarg{\ensuremath{\left\{\the\sset@tok\right\}}}
\def\sset@action{\ifx\sset@temp[
\let\sset@next=\sset@witharg\else\let\sset@next=\sset@withoutarg\fi\sset@next}
\makeatother

\newcommand{\sett}[2]{\set{#1}{\text{#2}}}

\newtheorem{thm}{Theorem}[section]

\newtheorem{prop}[thm]{Proposition}
\newtheorem{lemma}[thm]{Lemma}

\newtheorem{cor}[thm]{Corollary}%

{%
\theoremstyle{definition}
\newtheorem{defn}[thm]{Definition}
\newtheorem{notn}[thm]{Notation}
\newtheorem{example}[thm]{Example}
\newtheorem{hyp}[thm]{Hypothesis}
\newtheorem{nhyp}[thm]{Hypothesis}

}

{%
\theoremstyle{remark}
\newtheorem{rem}[thm]{Remark}

}

\newenvironment{pn}{\begin{prop}}{\end{prop}}
\newenvironment{lm}{\begin{lemma}}{\end{lemma}}
\newenvironment{dn}{\begin{defn}}{\end{defn}}
\newenvironment{rk}{\begin{rem}}{\end{rem}}

\usepackage{amsrefs}

\newcommand{\xrhardcode}[2]{\expandafter\def\csname#1\endcsname{#2}}
\xrhardcode{J-prop:tJdZ=cpct-mod-Z}{2.41} 
\xrhardcode{exp-cor:aniso-Levi}{9.16} 
\xrhardcode{exp-cor:compare-centralizers}{6.14} 
\xrhardcode{exp-cor:master-comm}{5.18}
\xrhardcode{exp-cor:stab-norm}{5.21}
\xrhardcode{exp-defn:Brgamma}{9.5}
\xrhardcode{exp-defn:admissible}{5.8}
\xrhardcode{exp-defn:bracket}{6.6} 
\xrhardcode{exp-defn:compatibly-filtered}{4.3}
\xrhardcode{exp-defn:concave}{5.7}
\xrhardcode{exp-defn:fancy-centralizer-no-underline}{9.3}
\xrhardcode{exp-defn:funny-centralizer}{6.4}
\xrhardcode{exp-defn:r-approx}{6.8} 
\xrhardcode{exp-defn:tame-reductive-sqnc}{5.1}
\xrhardcode{exp-defn:vGvr-split}{5.13}
\xrhardcode{exp-defn:vGvr}{5.14}
\xrhardcode{exp-hyp:conn-cent}B
\xrhardcode{exp-hyp:good-weight-lattice}C
\xrhardcode{exp-hyp:reduced}A
\xrhardcode{exp-hyp:torus-H1-triv}D
\xrhardcode{exp-lem:Brgamma-descends}\undefined
\xrhardcode{exp-lem:Brgamma}{9.6}
\xrhardcode{exp-lem:G-approx-is-G'-approx}{8.2} 
\xrhardcode{exp-lem:aniso-Brgamma}{9.13} 
\xrhardcode{exp-lem:bracket}{9.8} 
\xrhardcode{exp-lem:center-comm}{5.30} 
\xrhardcode{exp-lem:compare-centralizers}{6.13} 
\xrhardcode{exp-lem:more-vGvr-facts}{5.29} 
\xrhardcode{exp-lem:normal-conj}{9.2}
\xrhardcode{exp-lem:perp-commute}{9.1}
\xrhardcode{exp-lem:rigidity}{9.10} 
\xrhardcode{exp-lem:shallow-comm}{5.32} 
\xrhardcode{exp-lem:simult-approx}{8.1} 
\xrhardcode{exp-lem:tame-descent}{5.33} 
\xrhardcode{exp-lem:torus-what-torus}{5.20}
\xrhardcode{exp-prop:H1-iso}{5.39} 
\xrhardcode{exp-prop:aniso-Levi}{9.14} 
\xrhardcode{exp-prop:compatibly-filtered-tame-rank}{4.6}
\xrhardcode{exp-prop:heres-a-gp}{5.40} 
\xrhardcode{exp-prop:unique-approx}{8.4}
\xrhardcode{exp-rem:0+-approx-is-tJd}{6.9}
\xrhardcode{exp-rem:actually-product}{3.4}
\xrhardcode{exp-rem:approx-facts-commutator}{6.10(3)} 
\xrhardcode{exp-rem:approx-facts-in-center}{6.10(1)} 
\xrhardcode{exp-rem:approx-facts-in-stab}{6.10(2)} 
\xrhardcode{exp-rem:approx-facts}{6.10} 
\xrhardcode{exp-rem:bracket-facts-containment}{6.7(1)} 
\xrhardcode{exp-rem:bracket-facts-decomp}{6.7(4)} 
\xrhardcode{exp-rem:bracket-facts}{6.7} 
\xrhardcode{exp-rem:rigidity}\undefined
\xrhardcode{exp-rem:when-hyps-hold}{2.2}
\xrhardcode{exp-sec:buildings}{3.2}
\xrhardcode{exp-sec:good-facts}7
\xrhardcode{exp-sec:notation-generalities}{3.1}

\xrhardcode{exp-defn:good}{6.1}

\newcommand{\indexme}[2]{\label{index:#1}}
\newcommand{\indexmem}[2]{\label{index:#1}}

\numberwithin{equation}{section}

\author[Adler]{Jeffrey D. Adler}
\address{American University \\ Washington, DC \ \ 20016-8050}
\email{jadler@american.edu}
\author[Spice]{Loren Spice}
\address{Texas Christian University \\ Fort Worth, TX \ \ 76109}
\email{l.spice@tcu.edu}
\thanks{The first-named author was partially supported by
the National Security Agency
(H98230-05-1-0251, H98230-07-1-002, and H98230-08-1-0068)
and by a University of Akron Faculty Research Grant (\#1604).
The second-named author was supported by
National Science Foundation Postdoctoral Fellowship award
DMS-0503107.}

\title[Supercuspidal characters]%
{Supercuspidal characters of reductive $p$-adic groups}
\subjclass[2000]{Primary 22E50, 22E35}
\keywords{%
$p$-adic group, supercuspidal representation, character}
\date{4 November, 2009}

\begin{document}
\begin{abstract}
We compute the characters of many supercuspidal representations
of reductive $p$-adic groups.
Specifically, we deal with representations that arise
via Yu's construction from data satisfying a certain
compactness condition.
Each character is expressed in terms of a depth-zero character
of a smaller group,
the (linear) characters appearing in Yu's construction,
Fourier transforms of orbital integrals,
and certain signs and cardinalities
that are described explicitly
in terms of the datum associated to the representation
and of the element at which the character is evaluated.
\end{abstract}

\maketitle
\setcounter{tocdepth}2
\tableofcontents

\pagebreak 

\addtocounter{section}{-1}
\section{Introduction}
\label{sec:intro}

\subsection{History}
\label{ssec:history}
Suppose $F$ is a non-Archimedean local field,
$\bG$ is a connected reductive $F$-group,
and $G = \bG(F)$.
For simplicity of the present discussion, assume
that $F$ has characteristic zero.
For $\pi$ an irreducible, admissible representation
of $G$, let $\Theta_\pi$ denote the distribution character of $\pi$,
a linear functional on the
space $C_c^\infty(G)$ of locally constant, compactly supported
functions on $G$.
Howe \cite{howe:qualitative} and Harish-Chandra \cite{hc:submersion}
showed that $\Theta_\pi$ can be represented by
a locally constant function on the set of regular,
semisimple elements of $G$.
We will also denote the representing function by $\Theta_\pi$.

A great deal is known about the asymptotic behavior of
characters (as functions) near singular points.
For example, the blow-up of $\Theta_\pi$ is controlled by
the fact, due to Harish-Chandra \cite{hc:queens},
that
$\abs{D_G}^{1/2}\Theta_\pi$ is locally integrable on $G$,
where $D_G$ is a certain polynomial function on $G$,
the \emph{discriminant} of $G$.
From Howe \cite{howe:fourier} and Harish-Chandra \cite{hc:queens},
we know that, near a singular point, $\Theta_\pi$
(composed with a suitable logarithmic map)
has an expansion
in terms of Fourier transforms of nilpotent orbital integrals.
More recent work has made precise where these expansions hold
(see \cites{hales:characters,moy-prasad:k-types}
for a conjecture,
\cite{debacker:homogeneity} for the main result,
and \cite{adler-korman:loc-char-exp} for a generalization);
presented other expansions, where the collection
of orbital integrals involved is smaller and depends on $\pi$
(see
\cites{%
murnaghan:chars-sln,
murnaghan:chars-u3,
murnaghan:chars-classical,
murnaghan:chars-gln
});
or done both
(see
\cites{%
cunningham:ell-exp,
adler-debacker:mk-theory,
debacker-kazhdan:mk-zero,
jkim-murnaghan:charexp,
jkim-murnaghan:gamma-asymptotic
}).

Despite the work mentioned above,
in most cases, we do not have explicit character formulas,
even in a limited domain,
because neither the orbital integrals nor their coefficients
are understood explicitly
(though see
\cites{assem:sll,debacker-sally:germs,sally-shalika:orbital-integrals}
for exceptions).
In practice, such formulas usually arise from explicit information about
the construction of representations.
However, the construction methods can be quite complicated.

Let us restrict our attention to supercuspidal representations.
Suppose that
the residual characteristic $p$ of $F$ is odd.
Then earlier work has yielded character formulas for all of the
supercuspidal representations of
$\SL_2$ \cite{sally-shalika:characters}
(using the construction in \cite{sally:unimodular}, which is known to be
exhaustive by \cite{sally-shalika:plancherel});
$\PGL_2$ \cite{silberger:pgl2};
$\GL_2$ \cite{shimizu:gl2};
$\GL_\ell$ \cites{corwin-moy-sally:gll,debacker:thesis};
$\SL_\ell$ \cite{spice:thesis};
and
division algebras of degree $\ell$
\cites{corwin-moy-sally:degrees,corwin-moy-sally:gll}.
In the latter cases, $\ell$ is a prime that is sufficiently
small with respect to $p$.
In addition, one knows the characters of many
depth-zero representations of unramified groups
\cite{debacker-reeder:depth-zero-sc}
(namely, those induced from inflations of Deligne--Lusztig
representations of associated finite groups of Lie type),
certain depth-zero character values for $\SP_4$ \cite{boller:thesis},
and inductive formulas for characters of division algebras
\cite{corwin-howe:div-alg}.

An earlier announcement
\cite{adler-corwin-sally:division-formulas}
contains formulas for the characters of (necessarily
supercuspidal) representations of the multiplicative group
of a central division algebra over a $p$-adic field.
The present paper generalizes these results to the setting
of general tame reductive groups over a $p$-adic field
of odd residual characteristic,
where the construction of J.-K.{} Yu (see
\cite{yu:supercuspidal} and our \S\ref{sec:JK}) can be used
to replace that of Corwin, Howe, and Moy
(see, for example, \cites{corwin:division-alg-tame,howe:gln,moy:thesis}).
If $p$ is large enough,
then all supercuspidals of $G$ arise via this construction
(see \cite{jkim:exhaustion}).
If \bG is $\GL_n$ or the multiplicative group of a central division algebra over $F$ of index $n$,
then the Corwin--Howe--Moy construction is known to be
exhaustive even if we assume only that $p$ does not divide $n$
(see \cite{moy:thesis}).
In \cite{hakim-murnaghan:distinguished}*{\S 3.5}, Hakim and
Murnaghan discuss the relationship between this construction
and Yu's.

There are other constructions of supercuspidal representations that make
no tameness assumptions.  These start with
\cites{kutzko:supercuspidal-gl2-1,kutzko:supercuspidal-gl2-2}
and presently culminate in
\cites{%
bushnell-kutzko:gln-book,
bushnell-kutzko:sln-1,
bushnell-kutzko:sln-2,
secherre:glnd-3,
stevens:classical-sc}.
However, an attempt to use these to compute explicit character formulas
would require a different approach in order to overcome many serious technical difficulties.
For example, among many other things,
we make use of Bruhat--Tits theory and
Moy--Prasad filtrations, both of which behave poorly
under wild Galois descent.

\subsection{Outline of this paper}
In order to evaluate the character $\Theta_\pi$ of a representation
$\pi$ at a regular, semisimple element $\gamma$ in $G$, we require
first of all that $\gamma$ lie near a tame $F$-torus.
If $p$ is larger than a constant determined by the root system
of $\bG$, then all semisimple elements of $G$ have this property.
Second, we require that $\gamma$ be well approximated by a
product of good elements.
Such approximations,
called ``normal $r$-approximations'',
are analogous to truncations of expressions of
elements of $F\cross$ in the form
$$
\varepsilon^{m_0} \varpi^d \dotm
\prod_{i=1}^\infty (1 + \varepsilon^{m_i} \varpi^i),
$$
where $\varpi$ is a uniformizer of $F$ and $\varepsilon$
is a root of unity in $F$ of order coprime to $p$.
From Lemma \xref{exp-lem:simult-approx} of \cite{adler-spice:good-expansions},
we see that many
tame elements of $G$ have such an expansion.
Under mild hypotheses,
which are always satisfied when $\bG$ is an inner form of $\GL_n$,
all tame elements of $\bG$ have such an expansion.
The expansions we require,
together with their basic properties,
are discussed in
\cite{adler-spice:good-expansions}.  The reader may find it
particularly convenient to have at hand the statements of
Lemmata \xref{exp-lem:more-vGvr-facts} and
\xref{exp-lem:shallow-comm},
Proposition \xref{exp-prop:heres-a-gp},
and
Remarks \xref{exp-rem:bracket-facts}
and
\xref{exp-rem:approx-facts} of \loccit
(Analogous approximations, with analogous properties,
exist for elements of the Lie algebra of $G$.
The proofs are similar to, but easier than,
those in \loccit)

After presenting our basic notation in \S\ref{sec:notation},
we outline (in \S\ref{sec:JK})
Yu's construction of supercuspidal representations
(see \cite{yu:supercuspidal}).
Briefly, Yu starts with a sequence
$(\bG^0 \subsetneq \cdots \subsetneq \bG^d = \bG)$
of reductive groups,
together with (an inducing datum for) a depth-zero, supercuspidal
representation $\pi'_{0}$ of $G^0$,
a character $\phi_i$ of each $G^i$,
and
a point $\ox$ in the reduced
building of $\bG^0$ over $F$,
all satisfying certain properties.
He then constructs inductively, for each $i=0,\ldots,d$,
a smooth
representation $\rho'_i$ of an open compact modulo center
subgroup of $G^i$ such that
the representation $\pi_i$ of $G^i$ induced from
$\rho'_i \otimes \phi_i$ is irreducible and supercuspidal.

We will assume for the
remainder of this subsection that
$d > 0$ and $\bG^{d - 1}/Z(\bG)$ is
$F$-anisotropic.
(Notice that the latter hypothesis follows from the former if
\bG is $F$-anisotropic, or if
\bG is $\GL_\ell$ or $\SL_\ell$ with $\ell$ a prime.)
The reason for this assumption is that we require very precise control
over the behavior, with respect to a given Moy--Prasad
filtration, of certain commutators (see, for example,
Propositions \ref{prop:step1-support} and
\ref{prop:step1-formula1}).
Since the computation of the character of $\pi$
requires the evaluation of an integral
formula (see \eqref{eq:char-pi-1}) involving arbitrary
conjugates of the element in which we are interested, we
cannot guarantee this good behavior for arbitrary groups
$\bG^{d - 1}$;
but it \emph{does} occur when our compactness condition is
satisfied (see Corollary \ref{cor:wk-step1-support}).
Even without the compactness condition,
we can still compute the character values at many points
of a representation $\tau = \tau_d$ induced from
$\rho'_d \otimes \phi_d$ to a large open compact modulo
center subgroup of $G$ (see \S\ref{sec:JK}).
(In our situation,
the representation $\tau_i$ of \S\ref{sec:JK} is equal to $\pi_i$
for $0 \le i < d$.)

Since Weil representations over finite fields play an
essential role in Yu's construction of supercuspidal
representations, in \S\ref{sec:weil} we compute some
of their characters at certain elements,
following G\'erardin (see \cite{gerardin:weil}).

After the Weil representation computations,
our character computations broadly follow the strategies
pursued in \cite{corwin-moy-sally:gll} and
\cite{debacker:thesis}, both of which rely on
vanishing results to cut down the support of
the relevant characters.
In \cite{corwin-moy-sally:gll}, these vanishing results are
approached by computing first not the full induced character, but
rather the character of a representation induced to a
smaller open and compact modulo center subgroup.
For us, this is the representation $\sigma_i$ defined in
\S\ref{sec:JK}.
The desired vanishing results are discussed in
\S\ref{sec:vanishing}, where we use the fact that the
character of $\sigma_i$ transforms by a linear character
near the identity (see Corollary
\ref{cor:trho-isotyp}) to prove Proposition
\ref{prop:step1-support}

In \S\ref{sec:induction1}, we compute the character of
$\sigma_i$, using the results of \S\ref{sec:vanishing} to
cut down the class of elements we must consider.
Although Proposition \ref{prop:induction1} is the result
that is used most often in the sequel, the heart of this
section is really Proposition \ref{prop:step1-formula1}.
The proof of this result involves fairly intricate
manipulations of the Frobenius formula (see
\cite{sally:remarks}), based on our detailed
understanding of the behavior of taking commutators with an
element $\gamma$ (see
\cite{adler-spice:good-expansions}*{\S\xref{exp-sec:good-facts}}).
Historically, supercuspidal character formulas
(specifically, Theorem 4.2(c)
of \cite{corwin-moy-sally:gll},
Theorem 5.3.2 of \cite{debacker:thesis},
and Theorem 4(2) of
\cite{adler-corwin-sally:division-formulas})
have involved Gauss sums in some form.
These sums also appear in the present setting, but
in disguise.
We devote \S\ref{ssec:gauss} to recognising and computing them.

With vanishing results in place for this partially induced
representation, it becomes easier to describe the
character of the full induced representation.
After proving a few results (Lemmata
\ref{lem:finite-double-coset}--\ref{lem:char-pi-cpt-supp})
to show that certain naturally
arising integrals converge, we compute the full character in
Theorem \ref{thm:char-tau|pi-1} using Harish-Chandra's
integral formula.

The construction of \cite{yu:supercuspidal} is
inductive, in the sense that an inducing datum for $\pi_i$
is constructed from an inducing datum for $\pi_{i - 1}$ and
some additional data.
This means that, in order to compute the character of
$\pi = \pi_d$,
we focus on explicating the relationship
between the characters of $\pi_{d - 1}$ and $\pi_d$ (or, in
the notation of \S\ref{sec:JK}, $\tau_{d - 1}$ and $\tau_d$).
Theorem \ref{thm:char-tau|pi-1} is actually a statement
about this relationship.
Accordingly, the groups $\bG^i$ and representations $\pi_i$
for $0 \le i < d - 1$ play no explicit role in our
calculations until \S\ref{sec:full}, where we unroll the
inductive computations of
\S\S\ref{sec:weil}--\ref{sec:induction2} to obtain an
explicit formula for $\pi_d$ in terms of the original
datum.

The result of this unrolling is contained in Theorem \ref{thm:full-char}.
Applying the inductive formulas
of the preceding sections, we obtain there a formula
for $\Theta_\pi = \Theta_{\pi_d}$ in terms of the character
of $\pi'_0$, the (linear) characters $\phi_i$, and Fourier
transforms of certain orbital integrals.
(If $\bG^{d - 1}/Z(\bG)$ is not $F$-anisotropic, then we
compute instead the character of $\tau = \tau_d$ in terms of
essentially the same data, but with $\rho'_0$ in place of
$\pi'_0$.)
Also appearing in the character formulas are some
explicitly defined
positive constants (the numbers $c(\vec\phi, \gamma'_{< r})$ of
Theorem \ref{thm:full-char}) --- essentially the
square roots of cardinalities of certain quotients of
filtration subgroups of $G$ ---
and signs $\varepsilon(\phi, \gamma)$
and $\mf G(\phi, \gamma)$ ---
computed in Propositions \ref{prop:theta-tilde-phi} and
\ref{prop:gauss-sum} in terms of the root system of \bG
and various fields associated to the representation $\pi$
and to the element $\gamma$.

Thus, we obtain formulas for evaluating, at many points,
many supercuspidal characters of many groups.

\subsection{Future goals}
Our hypotheses are weak enough that,
in the case of ``tame'' division algebras, i.e., those of
index coprime to $p$,
we can evaluate all characters at all points.
In this case, the presence of considerable additional
structural information (and fine control over conjugacy,
thanks to the Skolem--Noether theorem) allows us to make the
formula of Theorem \ref{thm:full-char} more explicit.
In the process, we will correct an error in Theorem 4 of
\cite{adler-corwin-sally:division-formulas} (some of whose
corollaries remain valid).
This will be carried out in future work.

Work of Henniart (see \cite{henniart:jl1}) has suggested that
it is valuable to be able to recognise a representation
given only the values of its character in a certain domain.
Theorem \ref{thm:full-char} may be sufficiently explicit to
allow us to describe a domain for which this can be done
(at least, if we restrict ourselves to appropriate
supercuspidal representations).
This would nicely complement
\cite{hakim-murnaghan:distinguished}*{Chapter 6}, which describes
another way of identifying supercuspidal
representations.

The stability calculations of
\cite{debacker-reeder:depth-zero-sc} proceed from 
Proposition 10.1.1 of \loccit, a ``reduction formula''.
We have modelled our Theorem \ref{thm:char-tau|pi-1} after
this reduction formula, and believe that the similarity of
statements should provide a guide to stability calculations
for positive-depth supercuspidal representations.

\subsection{Acknowledgements}
This work could not have been completed without the notes of
the late
Lawrence Corwin, written in collaboration with Paul Sally,
on their computation of characters of
division algebras.
The work was actually begun by Allen Moy and Paul Sally,
who computed, in \cite{corwin-moy-sally:degrees},
the formal degrees of representations of division algebras
and general linear groups.
Our \S\ref{ssec:frobenius} is translated from
these notes, with the notation and techniques adapted to our
present setting (in particular, using the tools of
\cite{adler-spice:good-expansions}).
The general strategy of our work was also suggested by these
notes.

This work
has also benefited from our conversations with Paul Sally,
Gopal Prasad, Stephen DeBacker, Brian Conrad, and J.-K.\ Yu,
and from the referee's comments.
It is a pleasure to thank all of these people.

\numberwithin{thm}{subsection}
\numberwithin{equation}{subsection}
\section{Notation and preliminaries}
\label{sec:notation}

\subsection{Generalities on fields and linear reductive groups}
\label{sec:generalities}

Let $\tR = \R \sqcup \set{ r{+} }{r \in \R} \sqcup \sset\infty$,
\indexmem{tR}{\tR}
and extend the order and additive structures on $\R$ to ones on $\tR$
in the usual way
(see, for example,
\cite{adler-spice:good-expansions}*{\S\xref{exp-sec:notation-generalities}}).
For \F a finite field,
let $\sgn_\F$
\indexmem{sgnF}{\sgn_\F}
denote the
character of $\F\cross$ with
kernel precisely $(\F\cross)^2$.

If $F$ is a valued field with valuation $\ord$,
\indexmem{ord}{\ord}
and $r \in \tR_{\ge 0}$, then let
$F_r$ denote $\set{t \in F}{\ord(t) \ge r}$.
Then the residue field
\indexmem{ffF}{\ff_F}
$\ff=\ff_F$
of $F$ is the quotient
$F_0/F_{0{+}}$ of $F_0$ by $F_{0{+}}$.
We will identify functions on
$F_0$ that are trivial on $F_{0{+}}$
with functions on $\ff$.
In particular, if $\ff$ is finite, then we have the
function $\sgn_{\ff}$ on $F_0$\,.

\indexmem{ff}\ff
From now on,
assume that $F$ is locally compact,
and that the characteristic $p$
\indexmem{p}{p}
of its residue field $\ff = \ff_F$
is not $2$.
Fix an algebraic closure $\ol F$ of $F$,
and let
\indexmem{Fun}{F\unram}
$F\unram$
and
\indexmem{Fsep}{F\sep}
$F\sep$
denote the maximal
unramified and separable extensions of $F$ in $\ol F$.
Since $F$ is Henselian, there is a unique extension of
$\ord$ to each algebraic extension $E/F$ (in particular, to
$\ol F/F$), which we will denote again by $\ord$.

Fix a square root $\sqrt{-1}$ of $-1$ in \C, and
an additive character $\Lambda$
\indexmem{Lambda}{\Lambda}
of \ol F that is trivial
on $\ol F_{0+}$
and induces on $\ff = F_0/F_{0{+}}$ the character
$t \mapsto \exp\bigl(2\pi\sqrt{-1}\tr_{\ff/\F_p}(t)/p\bigr)$,
where $\F_p$ is the finite field with $p$ elements.
We may, and do, write again $\Lambda$ for the resulting
character of $\ff_E$, for any discretely valued algebraic
extension $E/F$.
Except in \S\ref{sec:induction1}, we will be concerned only with
the restriction to $F$ of $\Lambda$.
All Fourier transforms will be taken with respect to
$\Lambda$.
The particular choice of square root will be of interest
only in the statement of Proposition \ref{prop:gauss-sum}.

If we denote an algebraic $F$-group by a bold, capital, Latin
letter, such as $\bH$, then we will sometimes denote its
Lie algebra by the corresponding bold, small Gothic letter,
such as $\pmb{\mf h}$.
We will denote sets of rational points by the corresponding
non-bold letters, such as $H$ and $\mf h$.

For any set $S \subseteq X$, we denote by $[S]$ the
characteristic function of $S$.
(The ``universe'' $X$ will be understood from the context.)
If $S$ is finite, then we denote by $\card S$ its
cardinality.
If $H' \subseteq H$ are groups and $f$ is a function on
$H'\backslash H$, then
$\sum_{g \in H'\backslash H} f(g)$
will be shorthand for
$\sum_{H'g \in H'\backslash H} f(H'g)$.
Similar notation will be used for sums over double coset
spaces.

Let \bG denote a reductive $F$-group.
For the moment, we do not assume that $\bG$ is connected.
Let $\bG\conn$ denote the identity component of $\bG$.
Write $\fg^*$ for the dual Lie algebra of $G$, i.e., the
vector-space dual of \fg.

Suppose that $X^* \in \fg^*$ is semisimple, in the sense
that it is fixed by the coadjoint action of some maximal
torus in \bG.
Any $G$-equivariant identification of $\fg^*$ with \fg carries
$X^*$ to a semisimple element of \fg (in the usual sense),
so one knows that $G/C_G(X^*)$ carries an invariant
measure, say $d\dot g$, and that the integral
$$
\int_{G/C_G(X^*)} f(\lsup g X^*)d\dot g
$$
converges for $f \in C_c^\infty(\fg^*)$.
Thus we may define a distribution
$\hat\mu_{X^*}$
\indexmem{hat-mu}{\hat\mu_{X^*}^G}
on $\fg$ by
$$
\hat\mu_{X^*}(f)
= \int_{G/C_G(X^*)} \hat f(\lsup g X^*)d\dot g
\quad\text{for $f \in C_c^\infty(\fg)$}.
$$
By Theorem A.1.2 of \cite{adler-debacker:mk-theory},
$\hat\mu_{X^*}$ is representable by a locally constant function on the regular semisimple
set in \fg.
We may, and do, sometimes regard $\hat\mu_{X^*}$ as
defined by an integral over $G/Z$, where $Z$ is any closed
cocompact unimodular subgroup of $C_G(X^*)$.
By abuse of notation, we will denote again by
$\hat\mu_{X^*}$ (or $\hat\mu_{X^*}^G$, if we wish to
emphasize the ambient group $G$)
the representing function.
Notice that this function depends on the measure chosen.

If \bM is a Levi (not necessarily $F$-Levi) subgroup of \bG, then,
as in \cite{yu:supercuspidal}*{\S 8}, we identify the dual
Lie algebras of $Z(\bM)$
and \bM with the fixed points in the
dual Lie algebra of \bG for the coadjoint actions of \bM and
$Z(\bM)$, respectively.

\subsection{Hypotheses}
\label{ssec:hypotheses}

Assume now, and for the remainder of the paper,
that \bG is connected,
splits over some tame extension of $F$,
and satisfies
Hypotheses (\textbf{\xref{exp-hyp:conn-cent}}) and
(\textbf{\xref{exp-hyp:good-weight-lattice}}) of
\cite{adler-spice:good-expansions}.
By Remark \xref{exp-rem:when-hyps-hold} of \loccit,
Hypotheses (\textbf{\xref{exp-hyp:reduced}}) and
(\textbf{\xref{exp-hyp:torus-H1-triv}})
follow from the tameness of \bG.
Thus, we may apply all the results of
\cite{adler-spice:good-expansions}.
In some places, we also assume Hypothesis
\ref{hyp:X*-central}.

\subsection{Buildings and filtrations}
\label{sec:buildings}

For any algebraic extension $E/F$ of finite ramification degree,
let $\rBB(\bG,E)$ and $\BB(\bG,E)$ denote the reduced and
enlarged Bruhat--Tits buildings of $\bG(E)$,
respectively.
Then $\BB(\bG,E)$ is the product of $\rBB(\bG,E)$ and an affine
space.  For a point $x\in \BB(\bG,E)$,
let $\ox$ denote the image of $x$ under the natural projection
to $\rBB(\bG,E)$.

If $H$ is a closed subgroup of $G$ and $x \in \BB(\bG, F)$,
then we will abbreviate $H \cap \stab_G(\ox)$ to $\stab_H(\ox)$.
\indexmem{stabHx}{\stab_H(\ox)}
Note that, in this notation, \ox is an
element of $\rBB(\bG, F)$,
not $\rBB(\bH, F)$,
even if $H = \bH(F)$
with \bH a compatibly filtered $F$-subgroup of \bG
(as in Definition \xref{exp-defn:compatibly-filtered} of
\cite{adler-spice:good-expansions})
and $x \in \BB(\bH, F)$.
Of course, if further $Z(\bH)/Z(\bG)$ is $F$-anisotropic, then
actually there is no ambiguity, since we may regard
$\rBB(\bH, F)$ as a subcomplex of $\rBB(\bG, F)$.

Suppose \bT is a maximal $F$-torus in \bG.
Then we let
\indexmem{PhiGT}{\Phi(\bG,\bT)}
$\Phi(\bG,\bT)$
denote the set of roots of $\bT$ in $\bG$,
and put
\indexmem{tildePhiGT}{\wtilde\Phi(\bG,\bT)}
$\wtilde\Phi(\bG,\bT) =\Phi(\bG,\bT) \cup \{0\}$.
For each root $\alpha\in \Phi(\bG,\bT)$, let
\indexmem{LieGalpha}{\Lie(\bG)_\alpha}%
$\Lie(\bG)_\alpha$
and
$\bU_\alpha$
\indexmem{Ualpha}{\bU_\alpha, \quad \text{$\alpha$ a root}}%
denote the corresponding root space and root group,
respectively.
If $\alpha = 0$, then put
$\Lie(\bG)_\alpha = \Lie(\bT)$
and $\bU_\alpha = \bT$.
If \bT is $F$-split, then there are associated to \bT an
affine space $\AA(\bT, F)$ under
$\bX_*(\bT) \otimes_\Z \R$, the lattice of cocharacters of
\bT (tensored with \R),
and an embedding of $\AA(\bT, F)$ in $\BB(\bT, F)$.
\indexmem{psi+}{\psi{+}}
For us, an \emph{affine root} will be either an affine function $\psi$
on $\AA(\bT, F)$ whose gradient $\dot\psi$
belongs to $\wtilde\Phi(\bG,\bT)$,
or a function of the form
$\psi{+} \colon x \mapsto \psi(x){+}$
with $\psi$ as above.
For each affine root $\psi$, we have
a compact subgroup
\indexmem{EUpsi}{\lsub E U_\psi,\quad\text{$\psi$ an affine root}}
$\lsub F U_\psi$ of $U_\dpsi$
and a lattice
\indexmem{Eupsi}{\lsub E \mf u_\psi}
$\lsub F \mf u_\psi$ in $\fg_\dpsi$.
Note that other authors reserve the term ``affine root''
for an affine function $\psi$ such that $\dot\psi \in \Phi(\bG,\bT)$
and
$\lsub F U_\psi \ne \lsub F U_{\psi{+}}$.

In \cite{moy-prasad:k-types}*{\S\S 2.6, 3.2}
and \cite{moy-prasad:jacquet}*{\S\S 3.2--3},
Moy and Prasad have defined, for each $x \in \BB(\bG, F)$,
filtrations
$(G_{x, r})_{r \in \R_{\ge 0}}$,
$(\fg_{x, r})_{r \in \R}$,
and
$(\fg^*_{x, r})_{r \in \R}$ of
$G$ by compact open subgroups,
\fg by lattices,
and
$\fg^*$ by lattices, respectively.
We extend these filtrations in the usual fashion to be
defined for all $r \in \tR$ (or $r \in \tR_{\ge 0}$, in the
case of the filtration on $G$).
If $x \in \BB(\bG, F)$ and $g \in G_{x, 0}$\,,
then we let $\depth_x(g)$
\indexmem{depthx}{\depth_x}
be the greatest index $t$ such
that $g \in G_{x, t}$\,.
We define similar functions on \fg and $\fg^*$ (not just
$\fg_{x, 0}$ and $\fg^*_{x, 0}$), and denote them also by
$\depth_x$.

\indexmem{filt-quot}{\mc G_{i:j}}
If a group \mc G has a filtration $(\mc G_i)_{i \in I}$, then
we shall frequently write $\mc G_{i:j}$ in place of
$\mc G_i/\mc G_j$ when $\mc G_j \subseteq \mc G_i$
(even if the quotient is not a group).  For example, we put
$F_{r:t} = F_r/F_t$,
$U_{\psi_1:\psi_2} = U_{\psi_1}/U_{\psi_2}$, and
$G_{x, r:t} = G_{x, r}/G_{x, t}$
for $r \le t$ (and $r \ge 0$, in the last case) and for
affine roots $\psi_1$ and $\psi_2$ such that
$\dpsi_1 = \dpsi_2$ and $\psi_1 \le \psi_2$.

By Proposition \ref{prop:strong-mock-exp-holds}, 
for each finite, tamely ramified extension $E/F$,
tamely ramified maximal $F$-torus \bT,
and point $x \in \BB(\bG, E)$ (respectively,
$x \in \BB(\bG, F)$), we have maps
$\mexp^E_{x, t:u}$ and $\mexp_{\bT, x}$ satisfying
Hypotheses \ref{hyp:mock-exp} and \ref{hyp:strong-mock-exp}.
We write $\mexp_{x, t:u}$
\indexmem{exp-xtu}{\mexp_{x, t:u}}
for $\mexp^F_{x, t:u}$\,.
\indexmem{exp-x}{\mexp_x}
If the choice of \bT is unimportant, then we will sometimes
write $\mexp_x$ for $\mexp_{\bT, x}$.

In Definition \xref{exp-defn:vGvr} of
\cite{adler-spice:good-expansions},
following
\cite{yu:supercuspidal}*{\S\S 1--2},
we defined, for \bT a tame maximal $F$-torus,
filtration subgroups
$\lsub\bT G_{x, f}$
\indexmem{TGxf}{\lsub\bT G_{x, f}}
of $G$ associated to
a pair $(x, f)$ consisting of a point $x \in \BB(\bT, F)$
and a $\Gal(F\sep/F)$-invariant concave function
$f$ on the root system of \bT in \bG
(see Definition \xref{exp-defn:concave} of
\cite{adler-spice:good-expansions}).
It will be convenient here to define filtration lattices
$\lsub\bT\Lie(G)_{x, f}$
\indexmem{TLieGxf}{\lsub\bT\Lie(G)_{x, f}}
in \fg in the analogous fashion, by
$$
\lsub\bT\Lie(G)_{x, f}
:= \sum
\lsub E\mf u_\psi
	\cap
\fg\,,
$$
where $E/F$ is a tame splitting field for \bT,
the sum is taken over those affine roots $\psi$ of \bT in \bG
with $\psi(x) \ge f(\dpsi)$.
(In Definition \xref{exp-defn:vGvr} of
\cite{adler-spice:good-expansions},
we had to take considerable
care --- for example, intersecting with
some $G_{x', 0}$ instead of just with $G$, since parahorics
tend to behave badly under ramified descent; but, since
Lie algebra filtrations are considerably better behaved
(see, for example, Proposition 1.4.1 of
\cite{adler:thesis}), such care is not necessary here.)

If $(\bT, \vbG)$ is a tame reductive $F$-sequence in \bG,
in the sense of Definition \xref{exp-defn:tame-reductive-sqnc} of \cite{adler-spice:good-expansions},
and
$\vec r$ is an admissible sequence, in the sense of
Definition \xref{exp-defn:admissible} of \loccit (with associated concave
function $f_{\vbG, \vec r}$), then,
by analogy with the definition
\indexmem{vGxr}{\vG_{x, \vec r}}%
$\vG_{x, \vec r} = \lsub\bT G_{x, f_{\vbG, \vec r}}$
of Definition \xref{exp-defn:vGvr} of \loccit,
we put
\indexmem{vLieGxr}{\Lie(\vG)_{x, \vec r}}%
$\Lie(\vG)_{x, \vec r}
= \lsub\bT\Lie(G)_{x, f_{\vbG, \vec r}}$\,.
It is shown in Lemma \xref{exp-lem:torus-what-torus} of
\loccit that
$\vG_{x, \vec r}$ is independent of the choice of torus \bT.
The proof of the analogous result for $\Lie(\vG)_{x, \vec r}$
is, except for minor changes, the same.
For convenience, by abuse of notation, we will often write
\indexmem{LieTGxf}{\Lie(\lsub\bT G_{x, f})}%
$\Lie(\lsub\bT G_{x, f})$ in place of
$\lsub\bT\Lie(G)_{x, f}$ and $\Lie(\vG_{x, \vec r})$
\indexmem{LievGxr}{\Lie(\vG_{x, \vec r})}%
in
place of $\Lie(\vG)_{x, \vec r}$\,.

\subsection{Normal approximations}
\label{sec:normal}

We now define some basic concepts that will be needed in
what follows.  Since the definitions do not necessarily give
the full flavor of what is going on, we give a ``pictorial''
example in Example \ref{example:double-bracket} and describe
a detailed computation in \S\ref{ssec:dc-compute}.

If $t \in \tR$ and
$\ugamma = (\gamma_i)_{0 \le i < t}$ is a good sequence in $G$
(in the sense of Definition \xref{exp-defn:funny-centralizer} of \cite{adler-spice:good-expansions}),
then put
\begin{align*}
\CC\bG t(\ugamma)
& = \Bigl(\bigcap_{0 \le i < t} C_\bG(\gamma_i)\Bigr)\conn, \\
\intertext{and}
\CC G t(\ugamma) & = \CC\bG t(\ugamma)(F).
\end{align*}
In particular, $\CC\bG t(\ugamma) = \bG$ if $t \le 0$.
Note that the intersection defining $\CC\bG t(\ugamma)$ is
really a \emph{finite} intersection if $t < \infty$
(and, if $t = \infty$, then we have that
$\CC\bG\infty(\ugamma) = \CC\bG{t'}(\ugamma)$ for $t' \in \R$
sufficiently large).
\indexme{r-approx}{normal $r$-approximation}
We say (as in Definition \xref{exp-defn:r-approx} of
\loccit) that
\ugamma is a \emph{normal $t$-approximation} to an element
$\gamma \in G$
if there is an element $x \in \BB(\CC\bG t(\ugamma), F)$
such that
$\gamma
\in \bigl(\prod_{0 \le i < t} \gamma_i\bigr)\CC G t(\ugamma)_{x, t}$\,.
Sometimes, we will say for emphasis
that $(\ugamma, x)$ is a normal $t$-approximation.
In this case, we put
\indexmem{CbGrgamma}{\protect\CC{\protect\bG} r(\gamma)}%
\indexmem{CGrgamma}{\protect\CC{G}{r}(\gamma)}%
\indexmem{ZbGrgamma}{\protect\ZZ{\protect\bG} r(\gamma)}%
\indexmem{ZGrgamma}{\protect\ZZ G r(\gamma)}%
\begin{align*}
\CC\bG t(\gamma) & = \CC\bG t(\ugamma) \\
\CC G t(\gamma) & = \CC G t(\ugamma), \\
\ZZ\bG t(\gamma) & = Z(\CC\bG t(\gamma)), \\
\intertext{and}
\ZZ G t(\gamma) & = \ZZ\bG t(\gamma)(F).
\end{align*}
By Proposition \xref{exp-prop:unique-approx} of \loccit,
these groups are all independent of the choice of normal
$t$-approximation to $\gamma$.
Note that, if $t' \in \tR$ and $t' \le t$, then \ugamma is also a normal
$t'$-approximation to $\gamma$, so the notations
$\CC\bG{t'}(\gamma)$ and $\ZZ\bG{t'}(\gamma)$ are defined;
and we have that
$\ZZ\bG{t'}(\gamma) \subseteq \ZZ\bG t(\gamma)
\subseteq \CC\bG t(\gamma) \subseteq \CC\bG{t'}(\gamma)$.

We will also write
\indexmem{gamma<r}{\gamma_{< r}}%
\indexmem{gamma-ge-r}{\gamma_{\ge r}}%
$\gamma_{< t} = \prod_{0 \le i < t} \gamma_i$
and
$\gamma_{\ge t} = \gamma_{< t}\inv\gamma$
(so that, with the point $x \in \BB(\CC\bG t(\gamma), F)$ as
above, we have
$\gamma_{\ge t} \in \CC G r(\gamma)_{x, t}$).
These elements should be thought of as the ``head'' and
``tail'' of $\gamma$, respectively.
By Corollary \xref{exp-cor:compare-centralizers} of \loccit,
$\CC\bG t(\gamma) = C_\bG(\gamma_{< t})\conn$.
Although the head and tail
are not independent of the choice of normal
$t$-approximation to $\gamma$, they are usually
``well determined enough'' (as described precisely in
Proposition \xref{exp-prop:unique-approx} of \loccit)
that we need not specify the choice.

If $t > 0$, then we put
\indexmem{Brgamma}{\BB_r(\gamma)}
$\BB_t(\gamma)
= \set{x \in \BB(\CC\bG t(\gamma), F)}
{\depth_x(\gamma_{\ge t}) \ge t}$.
By Lemma \xref{exp-lem:Brgamma} of \loccit, this is uniquely determined,
even though $\gamma_{\ge t}$ is not.
(An analogous set can also be defined when $t = 0$, as in
Definition \xref{exp-defn:Brgamma} of \loccit; but we do not
need this.)
Since $\depth_x(\gamma_i) \ge i$ for
$0 \le i < t$
and $x \in \BB(\CC\bG r(\gamma), F)$ (in particular,
for $x \in \BB_t(\gamma)$),
we have that $\BB_t(\gamma) \subseteq \BB_{t'}(\gamma)$
whenever $t' \in \tR_{> 0}$ and $t' \le t$.

If $t \in \tR_{\ge 0}$ and
$\gamma \in G$ has a normal
$t$-approximation, then,
in the notation of Definition \xref{exp-defn:vGvr} of
\cite{adler-spice:good-expansions},
we put
$\vbG = (\CC\bG{t - i}(\gamma))_{0 < i \le t}$
and
$\vec s = (i/2)_{0 < i \le t}$\,,
and write
\indexmem{dc}\dc%
$\odc{\gamma; x, t} = \vG_{x, \vec s}$\,.

We will also need various ``truncations''
\indexmem{dcj}{\dc^{(j)}}%
$\odc{\gamma; x, t}^{(j)}$ of $\odc{\gamma; x, t}$, as in Definition
\xref{exp-defn:fancy-centralizer-no-underline} of \loccit
These arise by taking only those terms in \vbG and $\vec s$
above with $0 < i < 2j$.
We will append a subscript $G'$ (writing instead
$\odc{\gamma; x, t}_{G'}$ or
$\odc{\gamma; x, t}_{G'}^{(j)}$)
\indexmem{dcG'}{\dc_{G'}}%
\indexmem{dcG'j}{\dc_{G'}^{(j)}}%
to indicate that we are constructing the
analogous object, but inside the ambient group $G'$, rather than
$G$.

\begin{dn}
\label{defn:trunc}
If $\vbG = (\bG^0, \dotsc, \bG^d = \bG)$ is a tame reductive sequence
in \bG, in the sense of Definition \xref{exp-defn:tame-reductive-sqnc}
of \cite{adler-spice:good-expansions},
and $\vec r = (r_0, \dotsc, r_d) \in \tR_{\ge 0}^{d + 1}$,
then we write $\mc T(\vbG, \vec r)$
\indexmem{TvGvr}{\mc T(\vbG, \vec r)}
for the set
$$
\sett{\delta \in G}
	{$\delta$ has a normal $r_{d - 1}$-approximation and
	$\delta_{< r_i} \in G^i$ for $0 \le i < d$}.
$$
(Note that $r_d$ is a ``dummy number'' that has no effect on
the resulting set \mc T.)
\end{dn}

\begin{example}
\label{example:double-bracket}
Suppose that $\gamma$ has a normal $\infty$-expansion
$(\gamma_i)_{i \ge 0}$ with the property that
$\gamma_i = 1$ when $i \not\in \Z$ or $i > 3$.
Thus, $\gamma = \gamma_0\gamma_1\gamma_2\gamma_3$.
Then the group
$\odc{\gamma; x,7}$ is a product of various filtration
subgroups of centralizer subgroups of $G$.
(See Figure~\ref{fig:sears-tower}.)
The larger the centralizer subgroup that is involved,
the deeper is the filtration subgroup that appears.
The group 
$\odc{\gamma; x,7}^{(3)}$ corresponds to the region between
the vertical dotted lines in Figure~\ref{fig:sears-tower}.

\begin{figure}
\newcommand{\StringA}{
	$\CC G{3+}(\gamma)_{x,0+} = \CC G{\infty}(\gamma)_{x,0+}$}
\newcommand{\StringB}{
	$\CC G{2+}(\gamma)_{x,2} = \CC G{3}(\gamma)_{x,2}$}
\newcommand{\StringC}{
	$\CC G{1+}(\gamma)_{x,\frac52} = \CC G{2}(\gamma)_{x,\frac52}$}
\newcommand{\StringD}{
	$\CC G{0+}(\gamma)_{x,3} = \CC G{1}(\gamma)_{x,3}$}
\newcommand{\StringE}{
	$G_{x,\frac72} = \CC G{0}(\gamma)_{x,\frac72}$}


\setlength{\unitlength}{3000sp}%
\begin{picture}(6624,5751)(1489,-6973)
\put(6901,-3886){\makebox(0,0)[lb]{\StringE}}
{\multiput(3301,-4561)(0.00000,-117.07317){21}{\line( 0,-1){ 58.537}}
}%
{\multiput(6301,-4561)(0.00000,-117.07317){21}{\line( 0,-1){ 58.537}}
}%
{\put(6601,-3361){\vector( 0,-1){1500}}
}%
{\put(7501,-3961){\vector( 0,-1){1425}}
}%
{\put(5242,-1579){\vector(-1,-2){465}}
}%
{\put(5649,-2144){\vector(-1,-4){469}}
}%
{\put(6186,-2770){\vector(-1,-4){422}}
}%
\put(5101,-1486){\makebox(0,0)[lb]{\StringA}}
\put(5476,-2086){\makebox(0,0)[lb]{\StringB}}
\put(5926,-2686){\makebox(0,0)[lb]{\StringC}}
\put(6451,-3286){\makebox(0,0)[lb]{\StringD}}
{\put(1501,-6961){\line( 0, 1){1950}}
\put(1501,-5011){\line( 1, 0){1200}}
\put(2701,-5011){\line( 0, 1){450}}
\put(2701,-4561){\line( 1, 0){600}}
\put(3301,-4561){\line( 0, 1){450}}
\put(3301,-4111){\line( 1, 0){900}}
\put(4201,-4111){\line( 0, 1){450}}
\put(4201,-3661){\line( 1, 0){300}}
\put(4501,-3661){\line( 0, 1){1725}}
\put(4501,-1936){\line( 1, 0){600}}
\put(5101,-1936){\line( 0,-1){1725}}
\put(5101,-3661){\line( 1, 0){300}}
\put(5401,-3661){\line( 0,-1){450}}
\put(5401,-4111){\line( 1, 0){900}}
\put(6301,-4111){\line( 0,-1){450}}
\put(6301,-4561){\line( 1, 0){600}}
\put(6901,-4561){\line( 0,-1){450}}
\put(6901,-5011){\line( 1, 0){1200}}
\put(8101,-5011){\line( 0,-1){1950}}
}%
\end{picture}%

\caption{Illustration of $\odc{\gamma; x,7}$
in Example~\ref{example:double-bracket}}
\label{fig:sears-tower}
\end{figure}

\end{example}

\subsection{Example computation of $\dc$}
\label{ssec:dc-compute}

We give an extended example to illustrate how to compute
normal approximations and groups $\dc$ in practice.
This involves a considerable amount of notation, all of
which should be regarded as being in force for this
subsection only.

Suppose that
$\bG = \GL(V)$ for some
finite-dimensional $F$-vector space $V$.
Let $\gamma$ be an element of $G = \GL_F(V)$.
Suppose for simplicity that $\gamma$ is compact and semisimple.
We describe a recipe for computing the leading term
in a normal approximation to $\gamma$.
This gives an inductive recipe for computing a normal
approximation to $\gamma$, from which falls out an
explicit description of
$\CC\bG r(\gamma)$,
$\BB_r(\gamma)$,
and
$\dc$
for $r \in \tR_{\ge 0}$.
(Although we do not do so here, it is very easy also to
compute
the group
$[\gamma; x, r]$
occurring in Definition \xref{exp-defn:fancy-centralizer-no-underline}
of \cite{adler-spice:good-expansions}
from our description.)
Remember that we write
$\dc_G$ rather than just $\dc$,
and
$\BB^\bG_r(\gamma)$ rather than just $\BB_r(\gamma)$,
when we wish to emphasise the ambient group.

Note that, in general, there is no canonical choice of normal
approximation.
This is reflected in our recipe in the fact that we have to
make some choices (namely, of a field $L$ and a uniformizer
$\varpi_L$ of $L$).
The point is that, as remarked earlier,
the groups $\CC\bG r(\gamma)$ and $\dc$,
and the set $\BB_r(\gamma)$, are nonetheless well defined.

For $\lambda \in (F\sep)\cross$, write $E_\lambda(V)$ for the
minimal $\gamma$-stable $F$-subspace of $V$ such that the action
of $\gamma$ on $V/E_\lambda(V)$ does not have $\lambda$ as an eigenvalue.
Then
$E_\lambda(V) = \sset 0$ unless $\lambda \in (F\sep)\cross_0$;
$E_\lambda(V) = E_{\sigma\lambda}(V)$ for
$\sigma \in \Gal(F\sep/F)$;
and
$V = \bigoplus_{\lambda \in (\dot F\sep)\cross} E_\lambda(V)$,
where $(\dot F\sep)\cross$ is a set of representatives for
the action of $\Gal(F\sep/F)$ on $(F\sep)\cross$.
We have
for $\lambda \in (\dot F\sep)\cross$
that $E_\lambda(V)$ carries the structure of an
$F[\lambda]$-vector space, where $\lambda$ acts by
$\gamma$.

Put $\bT = Z(C_\bG(\gamma))$,
so that $T = \bT(F)$ is the set of all $g \in \GL_F(V)$ that
act on each $E_\lambda(V)$ as scalar multiplication by an
element of $F[\lambda]^\times$.
Note that \bT is maximal if and only if $\gamma$ is regular.
Assume further that $\gamma$ is \emph{tame}, i.e., that
there exists a finite,
tamely ramified, Galois extension $L$ of $F$ that contains
all of the eigenvalues of $\gamma$.
Then \bT is an $L$-split, hence tame, torus.

We now choose, for each $d \in \R_{\ge 0}$,
a set $\Lambda_d$ of coset representatives for
$L\cross_{d:d+}$ as follows.
Let $\varpi_L$ be a uniformizer of $L$ such that
$\varpi_L^{e(L/F)} \in F$.
Write $\Lambda_0$ for the set of absolutely semisimple
elements of $L\cross_0$ (that is, elements whose order is finite
and prime to $p$).
If $d > 0$ and $L_d = L_{d+}$, then put $\Lambda_d = \sset 1$.
If $d > 0$ and $L_d \ne L_{d+}$, then there is some integer
$k$ such that $\varpi_L^k \in L_d \smallsetminus L_{d+}$.
Put
$\Lambda_d = \set{1 + \lambda_0\varpi_d^k}{\lambda_0 \in \Lambda_0}$.
It is easy to verify that, for any $d \in \R_{\ge 0}$ and
$\lambda \in \Lambda_d$,
the stabilizers in $\Gal(L/F)$ of $\lambda$ and
$\lambda L\cross_{d+}$ are the same.

Let $d$ be the least index $i \in \R_{\ge 0} \cup \sset\infty$ such that
$E_\lambda(V) \ne 0$ for some $\lambda \in L\cross_i$.
Then $d = \depth(\gamma)$.
If $d = \infty$, then $\gamma = 1$; so we assume that
$d < \infty$.

Put $\bG_\lambda = R_{F[\lambda]/F}\GL(E_\lambda(V))$.
Then there is a natural isomorphism of
$\prod_{\lambda \in \Lambda_d} \bG_\lambda$ with $C_\bG(\bT)$,
hence a natural injection of it into \bG.
Corresponding to this injection is an injection of
$\BB(\prod_{\lambda \in \Lambda_d} \bG_\lambda, F)$,
hence of
$\prod_{\lambda \in \Lambda_d} \BB(\bG_\lambda, F)$,
into $\BB(\bG, F)$ with certain properties
(see
Proposition 2.1.5 of \cite{landvogt:functorial} or
Proposition
\xref{exp-prop:compatibly-filtered-tame-rank}
of \cite{adler-spice:good-expansions}
for details).
We will regard these injections as inclusions.

Remember that we have chosen a set $\Lambda_d$ of
representatives for $L\cross_{d:d+}$.
For $\lambda \in L\cross_d$, write $s_\lambda$ for the
element of $\Lambda_d \cap \lambda L\cross_{d+}$.
Let $\gamma_d$ be the element of $\GL_F(V)$ that acts on
$E_\lambda(V)$ by scalar multiplication by
$s_\lambda \in F[\lambda]$ for all
$\lambda \in L\cross_d$.
We claim that $\gamma_d \in T_d$ is good,
in the sense of Definition \xref{exp-defn:good} of
\cite{adler-spice:good-expansions}.
Indeed, if
$\alpha \in \Phi(\bG, \bT)$,
then we have that
$\alpha(\gamma_d) = s_\lambda s_{\lambda'}\inv$
for some $\lambda, \lambda' \in L\cross_d$.
Suppose that $\alpha(\gamma_d) \in L\cross_{d+}$.
Since $s_\lambda$ and $s_{\lambda'}$ are elements of a set
of representatives for $L\cross_{d:d+}$, we have that
$s_\lambda = s_{\lambda'}$, hence that
$\alpha(\gamma_d) = 1$, as desired.

Further,
$\gamma \equiv \gamma_d \pmod{T_{d+}}$.
Thus, $(\gamma_d)$ is a normal $(d{+})$-approximation to
$\gamma$.
Put $\gamma_{> d} = \gamma_d\inv\gamma \in \GL_F(V)$.
By abuse of notation, we will also write $\gamma_{> d}$ for
the restriction of this element to any space
$E_\lambda(V)$.

The groups $\CC\bG r(\gamma)$ and $\dc$, and the sets
$\BB_r^\bG(\gamma)$, look different depending on the
relative values of $r$ and $d$.
For ``small values'' of $r$, we have
\begin{align*}
\CC\bG r(\gamma) & = \bG & \text{for $r \le d$,} \\ \intertext{and}
\dc_G & = \GL_F(V)_{x, 0+} & \text{for $r \le d+$
	and $x \in \BB(\CC\bG r(\gamma), F)$.}
\end{align*}
For ``large values'' of $r$, remember that we have
identified
$\prod_{\lambda \in \Lambda_d} \BB(\bG_\lambda, F)$ with a subset of
$\BB(\bG, F)$.
If $x \in \BB(\bG, F)$ lies in this subset
and $\lambda \in \Lambda_d$, then we will write
$x_\lambda$ for the image of $x$ under the natural projection to
$\BB(\bG_\lambda, F)$.
With this notation, we have
\begin{align*}
\CC\bG r(\gamma)
& = \prod_{\lambda \in \Lambda_d}
	\CC{\bG_\lambda}r(\gamma_{> d})
& \text{for $r > d$,} \\
\BB^\bG_r(\gamma)
& = \prod_{\lambda \in \Lambda_d}
	\BB^{\bG_\lambda}_r(\gamma_{> d})
& \text{for $r \ge d$,}
\end{align*}
and
$$
\dc_G
 = G_{x, (r - d)/2}\dotm\prod_{\lambda \in \Lambda_d}
	\odc{\gamma_{> d}; x_\lambda, r}_{G_\lambda}
\quad
\text{for $r > d+$
	and $x \in \BB(\CC\bG r(\gamma), F)$.}
$$
In particular,
\begin{align*}
\CC\bG{d+}(\gamma) & = \prod_{\lambda \in \Lambda_d}
	\bG_\lambda \\
\intertext{and}
\BB_d^\bG(\gamma) & = \prod_{\lambda \in \Lambda_d}
	\BB(\bG_\lambda, F).
\end{align*}

Although the recipe given always works, it can result in
normal approximations with more non-$1$ terms than necessary.
In practice, it is usually easy to find a shorter normal
approximation to a given element of $G$.
We illustrate this in case $V = F^3$, so that $G = \GL_3(F)$.
There are too many cases to consider them all, so we give
only a few representative examples.

If $\gamma$ is regular elliptic,
then we have that $T$ is isomorphic to the
multiplicative group of a cubic extension of $F$.
If $\gamma \not\in Z(G)T_{d+}$,
then $(\gamma)$ is a normal
$\infty$-approximation to $\gamma$.
If there exists $r \in \R$ with $r > d$ such that
$\gamma \in Z(G)T_r \smallsetminus Z(G)T_{r+}$,
say $\gamma = z t$ with $z \in Z(G)$ and $t \in T_r$,
then $(z, t)$ is a normal
$\infty$-approximation to $\gamma$.

If $\gamma$ is neither split nor regular elliptic, then
we may write $V = V' \oplus V''$, where
$V'$ and $V''$ are $\gamma$-stable $F$-subspaces of $V$ of
dimensions $1$ and $2$ respectively.
Write $\gamma'$ and
$\gamma''$ for the restrictions of $\gamma$ to $V'$ and
$V''$, respectively; $G''$ for $\GL_F(V'')$; and $T''$ for
$Z(C_{G''}(\gamma''))$.
When convenient, we will abuse notation and consider an operator on $V'$
to be an operator on $V$ that acts trivially on $V''$ (and vice versa).
Since $\gamma'$ is an operator on a $1$-dimensional vector
space, it acts as multiplication by a scalar $\lambda'$.
We divide this case into subcases depending on
which of $\gamma'$ and $\gamma''$ has depth $d$.
\begin{itemize}
\item
Suppose that $\gamma'$ and $\gamma''$ both have depth $d$.
If $\gamma'' \not\in Z(G'')T''_{d+}$,
then $(\gamma)$ is a normal $\infty$-approximation to $\gamma$.
On the other hand,
if $\gamma'' \in Z(G'')T''_{d+}$,
then ${\lambda'}\inv\gamma'' \in Z(G'')T''_{d+}$.
Since $\gamma$ is not split, there exists $r \in \R$ with $r>d$ such that
${\lambda'}\inv\gamma'' \in Z(G'')T''_{r} \smallsetminus Z(G'')T''_{r+}$.
Write
${\lambda'}\inv\gamma'' = z''t$ with $z''\in Z(G'')$
and $t \in T''_r$.
If $z''$ has depth $d$, then
$(\lambda'z'', t)$
is a normal $\infty$-approximation to $\gamma$.
If $d < \depth(z) < r$, then
$(\lambda',z'', t)$
is a normal $\infty$-approximation to
$\gamma$.
\item
Now suppose that $\gamma'$ has depth $d$ and $\gamma''$ has depth $s>d$.
Since $\gamma$ is not split, there exists $r \in \R$ with $r\geq s$ such that
$\gamma'' \in Z(G'')T''_{r} \smallsetminus Z(G'')T''_{r+}$.
If $r = s$, then $(\gamma',\gamma'')$ is a normal $\infty$-approximation
for $\gamma$.
If $r>s$, then we may write
$\gamma'' = z''t$ with $z''\in Z(G'')$
and $t \in T''_r \smallsetminus T''_{r+}$,
in which case $(\gamma', z'', t)$ is
a normal $\infty$-approximation
for $\gamma$.
\item
The subcase where $\gamma''$ has depth $d$ and $\gamma'$ has depth greater
than $d$ is straightforward, but involves several sub-subcases, so we
omit it.
\end{itemize}

There remains the case when $\gamma$ is split.
This is
handled similarly, but the plethora of possibilities for
depths and congruences among various eigenvalues makes it
impractical to give a complete list.

\subsection{Representations and characters}
\label{sec:reps}

Recall that, if $\pi$ is a smooth admissible
representation of $G$,
then the character of $\pi$ is a distribution on
$G$.
From work of Harish-Chandra \cite{hc:submersion}
and G.{} Prasad \cite{adler-debacker:mk-theory}*{Appendix B}
(see also Corollary A.11 of
\cite{bushnell-henniart:local-tame-1}),
this distribution is represented on the regular semisimple set in
$G$
by a locally constant function.
As mentioned in \S\ref{ssec:history},
we will denote by $\Theta_\pi$
\indexmem{Theta-pi}{\Theta_\pi}
both the function and
the distribution.
If $H$ is an open subgroup of $G$ and $\rho$ is a
finite-dimensional representation of $H$,
then
we will often denote by $\theta_\rho$
\indexmem{theta-rho}{\theta_\rho}%
the function
$h \mapsto \tr \rho(h)$ on $H$,
and by
$\dot\theta_\rho$
\indexmem{dot-theta-rho}{\dot\theta_\rho}%
the function on $G$
that is equal to $\theta_\rho$ on $H$, and to $0$ on
$G\smallsetminus H$.

\numberwithin{thm}{section}
\numberwithin{equation}{section}
\section{Review of J.-K. Yu's construction}
\label{sec:JK}

We review here the construction of supercuspidal representations
found in \cite{yu:supercuspidal}.
The terminology of this section will remain in force
throughout the remainder of the paper.

If $\phi$ is a character of $G$
and $x \in \BB(\bG, F)$,
then denote by
\indexmem{depth-x-char}{\depth_x(\phi), \quad\text{$\phi$ a character}}
$\depth_x(\phi)$
the smallest index $d \in \R_{\ge 0}$ such that $\phi$ is trivial on
$G_{x, d+}$\,.

\begin{defn}
\label{defn:cusp-dat}
A \emph{cuspidal datum}
\indexme{cusp-dat}{cuspidal datum}
is a quintuple
$\Sigma = (\vbG, \vec{\phi}, \vec{r}, x, \rho_0')$, where
\begin{itemize}
\item
\indexmem{bGi}{\bG^i}
$\vbG=(\bG^0,\ldots,\bG^d=\bG)$ is a tame Levi $F$-sequence,
and $Z(\bG^0)/Z(\bG)$ is $F$-anisotropic;
\item
\indexmem{x}{x}
$x$ lies in $\BB(\bG^0,F)$, and $\ox \in \rBB(\bG^0, F)$
is a vertex;
\item
\indexmem{r_i}{r_i}
$\vec{r} = (r_0, \ldots, r_d)$
is a sequence of real numbers satisfying
$0 \le r_0 < \dotsb < r_{d-1} \leq r_d$
and $r_0 > 0$ if $d > 0$;
\item
\indexmem{phi_i}{\phi_i}
$\vec{\phi}=(\phi_0,\ldots,\phi_d)$,
where, for each $0 \le i < d$,
$\phi_i$ is a
character of $G^i$
such that $\depth_x(\phi_i) = r_i$ and
$\phi_i$ is
\emph{$\bG^{i + 1}$-generic relative to $x$}
(in the sense of \cite{yu:supercuspidal}*{\S9}),
and $\depth_x(\phi_d) = r_d$, or
$\phi_d = 1$ and $r_d = r_{d - 1}$;
\item
$\rho_0'$ is an irreducible representation of $\stab_{G^0}(\ox)$
whose restriction to $G^0_{x, 0}$ contains the inflation of
a cuspidal representation of $G^0_{x, 0:0+}$\,.
\end{itemize}
\end{defn}

For the remainder of this paper, fix a cuspidal datum
$\Sigma$, with associated notation as above.
For $0\le i \le d$, we have the following objects
associated to $\Sigma$:
\begin{itemize}
\item
\indexmem{s_i}{s_i}
non-negative real numbers $s_i = r_i/2$;
\item
\indexmem{K-i}{K^i}
subgroups
$$
K^i_\phplus =
\stab_{G^0}(\ox)
(G^0, \ldots, G^i)_{x, (0{+}, s_0, \dotsc, s_{i - 1})}
$$
and (for $i > 0$)
\indexmem{J-i}{J^i}
\indexmem{J-i+}{J^i_+}
\begin{gather*}
J^i_\phplus = (G^{i - 1}, G^i)_{x, (r_{i - 1}, s_{i - 1})},
\\
\intertext{and}
J^i_+ = (G^{i - 1}, G^i)_{x, (r_{i - 1}, s_{i - 1}{+})};
\end{gather*}
\item
\indexmem{rho_i'}{\rho_i'}
representations $\rho_i'$ of $K^i$
(see Remark \ref{rem:yu-particulars});
and
\item
\indexmem{pi_i}{\pi_i}
irreducible supercuspidal representations
$\pi_i = \Ind_{K^i}^{G^i} (\rho_i' \otimes \phi_i)$ of depth $r_i$,
in the sense of \cite{moy-prasad:jacquet}*{\S 3.4}.
\end{itemize}
In particular,
$\pi_0 = \Ind_{\stab_{G^0}(\ox)}^{G^0} \rho_0' \otimes \phi_0$
is a twist of a depth-zero irreducible supercuspidal representation.
Since our calculations offer no new information about
depth-zero supercuspidal representations, we assume
throughout that $d > 0$.

\begin{rk}
\label{rem:yu-particulars}
In \cite{yu:supercuspidal}*{\S11},
for $0 \le i < d$, Yu constructs a canonical
representation $\tilde\phi_i$
\indexmem{tilde-phi_i}{\tilde\phi_i}
of $\stab_{G^i}(\ox) \ltimes J^{i + 1}$.
(However, by Proposition 3.26 of
\cite{hakim-murnaghan:distinguished}, one actually has
considerable freedom in constructing this representation.)
Then, by \cite{yu:supercuspidal}*{\S 4},
$\rho_{i + 1}'$ is the push-forward of
$\tilde\phi_i
	\otimes ((\rho_i' \otimes \phi_i) \ltimes 1)$
along the map
$K^i \ltimes J^{i + 1} \to K^iJ^{i + 1} = K^{i + 1}$.
\end{rk}

For $0 \le i \le d$, write again $\phi_i$ for the character
of $G^i_{x, s_i{+}:r_i{+}}$ induced by $\phi_i$.
Then there is an element
\indexmem{X_i*}{X_i^*}
$X_i^* \in \fg^{i\,*}_{x, -r_i}$ such that
$$
\phi_i \circ \mexp_{x, s_i{+}:r_i{+}}
	\bigr|_{\fg^i_{x, s_i{+}:r_i{+}}}
= \Lambda \circ X_i^*\bigr|_{\fg^i_{x, s_i{+}:r_i{+}}}.
$$
(Note that the right-hand side makes sense as a map on
$\fg^i_{x, s_i{+}:r_i{+}}$\,, because
$X_i^*(\fg^i_{x, r_i{+}})
\subseteq F_{0{+}} \subseteq \ker \Lambda$.)
By the definition of genericity, we have
$X_i^* \in \mf z(\fg^i)^*_{x, -r_i}
	+ \fg^{i\,*}_{x, (-r_i){+}}$.
Note that $X_i^*$ is determined only up to translation by
$\fg^{i\,*}_{x, -s_i}$\,.

For the results of \S\ref{ssec:gauss}, we require a hypothesis on the
elements $X_i^*$
that is a weaker version of Hypothesis C($\vbG$) of
\cite{hakim-murnaghan:distinguished}*{\S 2.6}.
In particular, by Lemma 2.50 of \loccit, it holds whenever
$\bG = \GL_n$.

\begin{hyp}
\label{hyp:X*-central}
$X_i^* \in \mf z(\fg^i)^* + \fg^{i + 1\,*}_{x, -s_i}$
for $0 \le i < d$.
\end{hyp}

This hypothesis is used only in the proofs of Corollaries \ref{cor:Q-const}
and \ref{cor:Q-and-B} to allow the invocations there of
Lemma \ref{lem:phi-trivial}.
These results, in turn, are necessary only for the computations
of
Propositions \ref{prop:gauss-sum-card} and \ref{prop:gauss-sum}.
If the hypothesis were dropped, then we could still prove a
version of Theorem \ref{thm:full-char}, but it would involve
the undetermined quantity $\wtilde{\mf G}(\phi, \gamma)$ (see
\S\ref{ssec:gauss}), hence be less
explicit.

By Proposition \xref{exp-prop:heres-a-gp} of \cite{adler-spice:good-expansions},
$\stab_{G^i}(\ox)G_{x, s_i{+}}
= \stab_{G^i}(\ox)(G^i, G)_{x, (r_i{+}, s_i{+})}$.
We denote by $\hat\phi_i$
\indexmem{hat-phi_i}{\hat\phi_i}
the character of
$\stab_{G^i}(\ox)G_{x, s_i{+}}$
that agrees with $\phi_i$ on $\stab_{G^i}(\ox)$ and is trivial on
$(G^i, G)_{x, (r_i{+}, s_i{+})}$.
(In particular, $\hat\phi_i$ is trivial on $G_{x, r_i{+}}$\,.)
If we write again $\hat\phi_i$ for the induced character of
$G_{x, s_i{+}:r_i{+}}$\,, then we have that
$$
\hat\phi_i \circ \mexp_{x, s_i{+}:r_i{+}}
= \Lambda \circ X_i^*
$$
as maps on $\fg_{x, s_i{+}:r_i{+}}$\,.

In order to study the various $\pi_i$ via induction in
stages, we put
\indexmem{K-sigma-i}{K_{\sigma_i}}%
$K_{\sigma_i} = \stab_{G^{i - 1}}(\ox)G^i_{x, 0{+}}$
and
\indexmem{sigma_i}{\sigma_i}%
$\sigma_i = \Ind_{K^i}^{K_{\sigma_i}} \rho_i'$
for $0 \le i \le d$.
Since we cannot yet compute the character of the full induced
representations $\pi_i$ in all cases,
we will sometimes consider instead the character of
\indexmem{tau_i}{\tau_i}%
$\tau_i
:= \Ind_{K^i}^{\stab_{G^i}(\ox)}
	\rho_i' \otimes \phi_i
= \Ind_{K_{\sigma_i}}^{\stab_{G^i}(\ox)} \sigma_i \otimes \phi_i$
for $0 \le i \le d$.

Since we will need to use induction on the length $d$ of $\Sigma$,
we abbreviate $\bG' = \bG^{d - 1}$ (and $G' = \bG'(F)$).
Further, we put
\indexmem{G'}{G'}
\indexmem K K\indexmem J J\indexmem{J+}{J_+}\indexmem{rho'}{\rho'}\indexmem{pi}\pi%
\indexmem{K-sigma}{K_\sigma}\indexmem{sigma}\sigma\indexmem{tau}\tau%
\indexmem r r%
\indexmem s s%
\indexmem{phi}\phi%
\indexmem{X*}{X^*}%
\indexmem{hat-phi}{\hat\phi}%
\indexmem{tilde-phi}{\tilde\phi}%
\begin{gather*}
r = r_{d - 1},\ s = s_{d - 1},\ \phi = \phi_{d - 1},\ 
X^* = X^*_{d - 1},\ \hat\phi = \hat\phi_{d - 1},\ 
\tilde\phi = \tilde\phi_{d - 1}, \\
K = K^d = \stab_{G^0}(\ox)\vG_{x, (0{+}, s_0, \dotsc, s_{d - 1})}, \\
K_\sigma = K_{\sigma_d} = \stab_{G'}(\ox)G_{x, 0+}, \\
J = J^d = (G', G)_{x, (r, s)}, \\
J_+ = J^d_+ = (G', G)_{x, (r, s{+})}, \\
\rho' = \rho'_d,\
\sigma = \sigma_d = \Ind_K^{K_\sigma} \rho',\
\tau = \tau_d = \Ind_{K_\sigma}^{\stab_G(\ox)} (\sigma \otimes \phi_d), \\
\intertext{and}
\pi = \pi_d = \Ind_{K_\sigma}^G (\sigma \otimes \phi_d).
\end{gather*}
That is, omitting a sub- or superscript $i$ will be the same
as taking $i = d$, except for $r$, $s$, $\phi$, $X^*$,
$\hat\phi$, and $\tilde\phi$, where it will be the same as
taking $i = d - 1$.
The ``basic ingredient'' in our character formula for $\pi$ will be
\indexmem{pi'_0}{\pi'_0}
$\pi'_0 := \Ind_{K^0}^{G^0} \rho'_0$,
a depth-zero supercuspidal representation.

Finally, put $\tilde\rho = \Ind_K^{\stab_{G'}(\ox)K} \rho'$.
\indexmem{tilde-rho}{\tilde\rho}

\begin{lm}
\label{lem:trho-char}
For $k \in K^{d - 1}$ and $j \in J$,
$$
\theta_{\tilde\rho}(k j)
= \theta_{\tilde\phi}(k \ltimes j)\theta_{\tau_{d - 1}}(k).
$$
\end{lm}

\begin{proof}
By the Frobenius formula,
$$
\theta_{\tilde\rho}(k j)
= \sum_{g \in K\backslash\stab_{G'}(\ox)K} \dot\theta_{\rho'}(\lsup g(k j)).
$$
We may, and do, actually regard the sum as running
over 
$(K \cap \stab_{G'}(\ox))\backslash\stab_{G'}(\ox)$.

Note that, by
Lemmata \xref{exp-lem:tame-descent} and
\xref{exp-lem:more-vGvr-facts} of \cite{adler-spice:good-expansions},
\begin{multline*}
K \cap \stab_{G'}(\ox)
= \stab_{G^0}(\ox)\vG_{x, \vec s} \cap \stab_{G'}(\ox) \\
= \stab_{G^0}(\ox)\vG_{x, \vec s(d - 1)}
= K^{d - 1}
\end{multline*}
(where $\vec s = (0{+}, s_0, \dotsc, s_{d - 2}, s_{d - 1})$
and
$\vec s(d - 1) = (0{+}, s_0, \dotsc, s_{d - 2}, \infty)$.)

Fix $g \in \stab_{G'}(\ox)$.
By Corollary \xref{exp-cor:stab-norm} of \cite{adler-spice:good-expansions},
$\stab_{G'}(\ox)$ normalizes $J$.
In particular, $\lsup g j \in J$.
Since $J \subseteq K$, we have that
$\lsup g(k j) \in K$ if and only if
$\lsup g k \in K$, i.e., if and only if
$\lsup g k \in K \cap \stab_{G'}(\ox) = K^{d - 1}$.
Therefore, either
\begin{enumerate}
\item
$\lsup g(k j) \not\in K$, so $\lsup g k \not\in K^{d - 1}$,
and
$$
\dot\theta_{\rho'}(\lsup g(k j)) = 0
= \theta_{\tilde\phi}(\lsup g k \ltimes \lsup g j)
\dot\theta_{\rho_{d - 1}' \otimes \phi}(\lsup g k);
$$
or
\item
$\lsup g(k j) \in K$, so $\lsup g k \in K^{d - 1}$, and, by
Remark \ref{rem:yu-particulars},
$$
\theta_{\rho'}(\lsup g(k j))
= \theta_{\tilde\phi}(\lsup g k \ltimes \lsup g j)
\theta_{\rho_{d - 1}' \otimes \phi}(\lsup g k).
$$
\end{enumerate}
Since $\tilde\phi$ is a representation of $\stab_{G'}(\ox) \ltimes J$,
its character has the same value at
$\lsup g k \ltimes \lsup g j = \lsup{g \ltimes 1}(k \ltimes j)$
as at $k \ltimes j$ for all $g \in \stab_{G'}(\ox)$.

Thus
\begin{multline*}
\theta_{\tilde\rho}(k j)
= \sum \theta_{\tilde\phi}(\lsup g k \ltimes \lsup g j)
	\dot\theta_{\rho_{d - 1}' \otimes \phi}(\lsup g k) \\
= \theta_{\tilde\phi}(k \ltimes j)
\sum \dot\theta_{\rho_{d - 1}' \otimes \phi}(\lsup g k) \\
= \theta_{\tilde\phi}(k \ltimes j)\theta_{\tau_{d - 1}}(k),
\end{multline*}
where both sums run over those cosets in
$K^{d - 1}\backslash\stab_{G'}(\ox)$
containing an element $g$
such that $\lsup g k \in K^{d - 1}$,
and
the last equality comes from another application of the
Frobenius formula and the fact that
$\tau_{d - 1}
= \Ind_{K^{d - 1}}^{\stab_{G'}(\ox)}
	\rho'_{d - 1} \otimes \phi$.
\end{proof}

\begin{lm}
\label{lem:trho-isotyp}
$\Res^{\stab_{G'}(\ox)K}_{J_+} \tilde\rho$ and
$\Res^{K_\sigma}_{G_{x, r}} \sigma$ are
$\hat\phi$-isotypic,
and $\Res^{\stab_G(\ox)}_{G_{x, r{+}}} \tau$
is $\phi_d$-isotypic.
\end{lm}

\begin{proof}
The statement about the restriction of $\tilde\rho$
follows from our Lemma \ref{lem:trho-char},
and Theorem 11.5 of \cite{yu:supercuspidal}
(reproduced as Theorem \ref{thm:yu:supercuspidal} below).
For the statement about the restriction of $\sigma$,
fix $\gamma\in G_{x,r}$\,,
and remember we have the Frobenius formula
$$
\theta_\sigma(\gamma)
= \sum_{g \in \stab_{G'}(\ox)K\backslash K_\sigma}
	\dot\theta_{\tilde\rho}(\lsup g\gamma).
$$
We may, and do, choose coset representatives $g$ in
the sum belonging to $G_{x, 0+}$\,, so that
$\lsup g\gamma \equiv \gamma \pmod{G_{x, r+}}$.
By the first statement, $G_{x, r+} \subseteq \ker \tilde\rho$.
Therefore,
$\Res^{K_\sigma}_{G_{x, r}} \sigma$ is
$\tilde\rho$-isotypic, so the second statement follows from another
application of the first statement.

Similarly, for $\gamma \in G_{x, r{+}}$, we have that
$\lsup g\gamma \in G_{x, r{+}}$\,,
hence that
$\theta_\sigma(\lsup g\gamma)
= \deg(\sigma)\hat\phi(\lsup g\gamma)
= \deg(\sigma)$
(by the second statement),
for $g \in \stab_G(\ox)$.  Thus
$$
\theta_\tau(\gamma)
= \phi_d(\gamma)\sum_{g \in K_\sigma\backslash\stab_G(\ox)}
	\theta_\sigma(\lsup g\gamma)
= \indx{\stab_G(\ox)}{K_\sigma}\deg(\sigma)\phi_d(\gamma),
$$
and the third statement follows.
\end{proof}

\section{Characters of Weil representations}
\label{sec:weil}

In this section, we make Lemma \ref{lem:trho-char} more
explicit by computing the character of the representation
$\tilde\phi$ appearing there.
We begin with two general results on extensions of finite fields.
These will be useful in this section, and in \S\ref{ssec:gauss}.

In this section, if $B$ is a non-degenerate bilinear or sesquilinear
form on a vector space $V$ over a field \F, then we will write
$\det B$ for the determinant of the matrix of $B$ with respect
to some fixed but arbitrary basis of $V$.
Since changing the basis does not change the square class of
the resulting determinant, it will not be necessary for our
purposes to specify the particular bases chosen.

\begin{lm}
\label{lem:gen-trace}
Let
\begin{itemize}
\item $\mb E/\F$ be a
degree-$n$ extension of odd-characteristic finite fields,
\item $\mytau$ an element of $\Gal(\mb E/\F)$ with $\mytau^2 = 1$,
and
\item $\Delta$ the determinant of the $\mytau$-Hermitian form
on \mb E given by
$$
(t_1, t_2) \mapsto \tr_{\mb E/\F} (t_1\mytau(t_2)).
$$
\end{itemize}
Then
$$
\sgn_\F(\Delta)
= \bigl(-\sgn_\F(\sgn_{\Gal(\mb E/\F)}(\mytau))\bigr)^{n + 1},
$$
where $\sgn_{\Gal(\mb E/\F)}$ is the
linear character of $\Gal(\mb E/\F)$
whose kernel is the group of squares.
\end{lm}

\begin{proof}
Let $\mysigma$ be a generator of $\Gal(\mb E/\F)$, and
$\set{e_i}{0 \le i < n}$ a basis for \mb E over \F.
The matrix, with
respect to the chosen basis, of the indicated pairing is
$M\dotm\trans{\mytau(M)}$, where $M$ is the $n \times n$ matrix with
$(i, j)$th entry $\mysigma^j e_i$ for $0 \le i, j < n$.
Since $\mysigma$ induces a permutation of the columns of $M$
that has parity opposite to that of $n$, we have that
$\det M \in \F\cross$ (equivalently, $(\det M)^2 \in (\F\cross)^2$)
if and only if $n$ is odd;
i.e., $\sgn_\F\bigl((\det M)^2\bigr) = (-1)^{n + 1}$.
Similarly, $\mytau$ induces a permutation of the columns of
$M$ that is even or odd according as
$\sgn_{\Gal(\mb E/\F)}(\mytau)^{n + 1}$ is $1$ or $-1$,
so
$\Delta = \det M\dotm\mytau(\det M)
= \sgn_{\Gal(\mb E/\F)}(\mytau)^{n + 1}(\det M)^2$.
The result follows.
\end{proof}

There is an analogue of Lemma \ref{lem:gen-trace}
for arbitrary finite cyclic Galois extensions of fields,
with essentially the same proof; but its statement is more
complicated, so we omit it.

\begin{lm}
\label{lem:witt-index}
Let
\begin{itemize}
\item \F be an odd-characteristic finite field,
\item $\mytau$ an automorphism of \F with $\mytau^2 = 1$,
\item $\varepsilon \in \sset{\pm 1}$,
\item $V$ a finite-dimensional \F-vector space,
and
\item $B$ a non-degenerate $(\varepsilon, \mytau)$-Hermitian
form on $V$.
\end{itemize}
Then the Witt index of $B$
(see \cite{jacobson:ba1}*{\S 6.5})
is $\lfloor \dim V/2\rfloor$
unless
\begin{itemize}
\item $\dim V$ is even,
\item $\mytau = 1$,
\item $\varepsilon = 1$,
and
\item $\det B$ is not in the square class of $(-1)^{\dim V/2}$,
\end{itemize}
in which case it is $(\dim V/2) - 1$.
\end{lm}

The Hermitian condition on $B$ means precisely that
it is linear in the first variable, and
$\varepsilon B(v, w) = \mytau(B(w, v))$ for $v, w \in V$.

\begin{proof}
Let $Q : v \mapsto B(v, v)$ be the quadratic form
associated to $B$,
and $\F'$ the fixed field of $\mytau$.
Denote by $N$ the map $t \mapsto t\dotm\mytau(t)$ from
\F to $\F'$.
Note that the image of $N$ is contained in the set of squares in \F.

Let $V = V_+ \oplus V_0 \oplus V_-$ be a Witt
decomposition of $V$, so that $V_+$ and $V_-$ are maximal
totally $Q$-isotropic subspaces of $V$ that are in duality by $B$,
and $V_0 = (V_+ \oplus V_-)^\perp$ is $Q$-anisotropic.
Then the matrix of $B$ on $V_+ \oplus V_-$,
with respect to a suitable basis, is of the form
$\left(\begin{smallmatrix}
0                   & \mytau(M) \\
\varepsilon\trans M & 0
\end{smallmatrix}\right)$
for some matrix $M$; so
$\det B\bigr|_{V_+ \oplus V_-}
= (-\varepsilon)^{\dim V_+}\dotm N(\det M)$
belongs to the square class of $(-\varepsilon)^{\dim V_+}$.
Let \mc B be a $B$-orthogonal basis for $V_0$,
so that
$Q\bigl(\sum_{v \in \mc B} a_v v)
= \sum_{v \in \mc B}  N(a_v)Q(v)$
for any constants $a_v \in \F$.

Suppose that $\mytau = 1$ and $\varepsilon = 1$.
If there are distinct $v_1, v_2 \in \mc B$ such that
$Q(v_2)$ belongs to the square class of $-Q(v_1)$
---
say
$Q(v_2) = -\lambda^2 Q(v_1)$,
with $\lambda \in \F\cross$
---
then $Q(\lambda v_1 + v_2) = 0$, which is a contradiction.
Thus, if $-1 \in (\F\cross)^2$, then no two elements
$Q(v)$, for $v \in \mc B$, lie in the same square class,
implying that $\card{\mc B} \le 2$;
and if $-1 \not\in (\F\cross)^2$,
then all of the elements $Q(v)$, for $v \in \mc B$, lie in the
same square class.
In this latter case, if $\card{\mc B} > 2$, then let
$v_1$, $v_2$, and $v_3$ be distinct elements of \mc B,
and write $Q(v_3) = c_i^2 Q(v_i)$,
with $c_i \in \F\cross$,
for $i = 1, 2$.
Then $Q(c_1\lambda v_1 + c_2\mu v_2 + v_3) = 0$,
where $\lambda, \mu \in \F$ are such that
$\lambda^2 + \mu^2 = -1$,
which is a contradiction.
Thus
\begin{itemize}
\item $\card{\mc B} \le 1$,
or
\item $\mc B = \sset{v_1, v_2}$, and $Q(v_2)$ does not belong
to the square class of $-Q(v_1)$.
\end{itemize}
In the former case, if $\dim V$ is even, then
$\card{\mc B} = 0$, so $\dim V_+ = \dim V/2$,
$V_+ \oplus V_- = V$,
and $\det B$ belongs to the square class of
$(-1)^{\dim V/2}$.
In the latter case, $-\det B\bigr|_{V_0} = -Q(v_1)Q(v_2)$
is a non-square in \F, so
$\det B = (\det B\bigr|_{V_+ \oplus V_-})(\det B\bigr|_{V_0})$
does not belong to the square class of
$(-1)^{\dim V_+}(-1) = (-1)^{\dim V/2}$.

If $\mytau = 1$ and $\varepsilon = -1$, then every vector is
isotropic for $Q$, so $V_0 = \sset 0$.

If $\mytau \ne 1$, then,
since $N$ surjects onto $\F'$, we have that
$\set{Q(v)}{v \in \mc B}$ is linearly independent over $\F'$.
On the other hand, since $B$ is $(\varepsilon, \mytau)$-Hermitian,
we have that $\set{Q(v)}{v \in \mc B}$ is contained in the
$\varepsilon$-eigenspace for $\mytau$ acting on \F, which is
$1$-dimensional over $\F'$.
Thus $\card{\mc B} \le 1$.
\end{proof}

\begin{thm}
\label{thm:gerardin:weil}
(Theorem 4.9.1 of \cite{gerardin:weil}.)
Let $V$ be an \ff-vector space equipped with
a non-degenerate symplectic form
$\langle\cdot, \cdot\rangle$,
and $\zeta$ an additive character of \ff.
Denote by $W^V_\zeta$ the Weil representation of $\SP(V)$
associated to $\zeta$
(defined in \cite{gerardin:weil}*{\S 2.4}%
).
Fix $g \in \SP(V)$.
\begin{inc_enumerate}
\item\label{thm:gerardin:weil(a)}
If $g$ has no non-zero fixed points, then let
$V_+$ be a maximal $g$-invariant totally isotropic subspace of $V$
and put $V_0=V_+^\perp/V_+$.
Then
$$
\theta_{W^V_\zeta}(g)
= \sgn_\ff(
	  (-1)^{\dim V_0/2}
          \det(g\bigr|_{V_+})
          \det(g - 1\bigr|_{V_0})
          ).
$$
\item\label{thm:gerardin:weil(b)}
If $g$ fixes pointwise a line $V_+ \subseteq V$, but $V_+$ does not
have a $g$-invariant complement in $V_+^\perp$, then
$$
\theta_{W^V_\zeta}(g)
= \theta_{W^{V_0}_\zeta}(g),
$$
where $V_0 = V_+^\perp/V_+$.
\item\label{thm:gerardin:weil(c)}
If $g$ fixes pointwise a line $V_+ \subseteq V$, and $V_0$ is a
$g$-invariant subspace of $V_+^\perp$ such that
$V_+^\perp = V_+ \oplus V_0$, then
$$
\theta_{W^V_\zeta}(g)
= \theta_{W^{V_0}_\zeta}(g)
  \sum_{v \in V_0^\perp/V_+}
      \zeta(\langle g v, v\rangle).
$$
\end{inc_enumerate}
\end{thm}

In \cite{yu:supercuspidal}*{\S 11},
there are described
a symplectic structure on $J/J_+$\,,
an action of $\SP(J/J_+)$ on $J/\ker \hat\phi$,
and an extension of the Weil
representation of $\SP(J/J_+)$ associated to $\hat\phi$
to a representation of
$\SP(J/J_+) \ltimes J/\ker \hat\phi$.
We have that $\tilde\phi$ (see Remark \ref{rem:yu-particulars})
is the pull-back of this extension to
$\stab_{G'}(\ox) \ltimes J$
via the map that restricts to the usual projection
$J \to J/\ker \hat\phi$,
and that takes $k \in \stab_{G'}(\ox)$ to the symplectic
transformation of $J/J_+$ induced by the conjugation action
of $k$ on $J$.

\begin{thm}
\label{thm:yu:supercuspidal}
(Theorem 11.5 of \cite{yu:supercuspidal}.)
$\tilde\phi\bigr|_{\sset 1 \ltimes J_+}$
is $\hat\phi$-isotypic, and
$\tilde\phi\bigr|_{G'_{x, 0+} \ltimes \sset 1}$
is $1$-isotypic.
\end{thm}

\begin{pn}
\label{prop:howe:weil}
The character of $\tilde\phi$ vanishes except on conjugacy
classes intersecting $\stab_{G'}(\ox) \ltimes J_+$\,.
\end{pn}

\begin{proof}
This result is proved for an ``abstract'' Weil
representation in \cite{howe:weil}.
The details of how to apply the result in our situation are
in \cite{yu:supercuspidal}*{\S 11}.
\end{proof}

Denote by $\Gamma$ the Galois group $\Gal(F\sep/F)$.
For $\mysigma \in \Gamma$, we will abuse notation
and also denote by $\mysigma$ the corresponding
element of $\Gal(\ol\ff/\ff)$, where \ol\ff is the residue
class field of $F\unram$.

Fix a bounded-modulo-$Z(G)$ element $\gamma \in G'$
(i.e., an element whose orbits in $\rBB(\bG', F)$ are
bounded in the sense of metric spaces).
By Proposition \xref{J-prop:tJdZ=cpct-mod-Z} of
\cite{spice:jordan} and
Remark \xref{exp-rem:0+-approx-is-tJd} of
\cite{adler-spice:good-expansions},
$\gamma$ has a normal $(0{+})$-approximation $(\gamma_0)$
(in $G$ and $G'$).
We assume that $x \in \BB_{0{+}}(\gamma)$.
Let \bT be a maximal tame $F$-torus (hence, a tame maximal
$F$-torus) in $\bG'$,
containing $\gamma_0$, such that $x \in \BB(\bT, F)$;
and let $E/F$ be a tame, Galois, strictly Henselian extension
over which \bT splits.

\begin{notn}
\label{notn:root-constants}
For $\alpha \in \wtilde\Phi(\bG, \bT)$, put
\begin{itemize}
\item $\Gamma_\alpha = \stab_\Gamma \alpha$;
\item $F_\alpha = (F\sep)^{\Gamma_\alpha}$;
\item
$e_\alpha = e(F_\alpha/F)$, the ramification index of
$F_\alpha/F$,
and
$f_\alpha = f(F_\alpha/F)$, the residual degree of
$F_\alpha/F$;
and
\item $\ff_\alpha = \ff_{F_\alpha}$.
\end{itemize}
If $-\alpha \in \Gamma\dota\alpha$, then
\begin{itemize}
\item $\mysigma_\alpha$ is any element of $\Gamma$
such that $\mysigma_\alpha\alpha = -\alpha$;
\item $F_{\pm\alpha}$ is the fixed field of
$\langle\Gamma_\alpha, \mysigma_\alpha\rangle$;
\item $f_{\pm\alpha} = f(F_{\pm\alpha}/F)$;
and
\item $\ff_{\pm\alpha} = \ff_{F_{\pm\alpha}}$.
\end{itemize}
\end{notn}

\begin{notn}
\label{notn:weil}
In the remainder of \S\ref{sec:weil} only,
for $\alpha \in \Phi(\bG, \bT)$,
denote by $\mo V_\alpha$ the image of
$$
\Lie(\bG)_\alpha(E) \cap \Lie(\bG', \bG)(E)_{x, (r, s)}
$$
in
$$
\Lie(\bG', \bG)(E)_{x, (r, s):(r, s{+})},
$$
and by $V_\alpha$ the set of $\Gamma_\alpha$-fixed points in
$\mo V_\alpha$.
More concretely, we have that
$\mo V_\alpha = \sset 0$ if
$\alpha \in \Phi(\bG', \bT)$;
and, if $\alpha \not\in \Phi(\bG', \bT)$, then
$\mo V_\alpha \cong \lsub E\mf u_{(\alpha + s):(\alpha + s){+}}$\,,
where $\alpha + s$ is the affine root with gradient $\alpha$
whose value at $x$ is $s$.
Since \bT is $E$-split, we have that $\Gal(F\sep/E)$ acts
trivially on $\Phi(\bG, \bT)$, so that
$F_\alpha \subseteq E$ --- hence, in particular, $e_\alpha$
is not divisible by $p$ --- for all
$\alpha \in \Phi(\bG, \bT)$.
Put
$\Xi(\phi) = \set{\alpha \in \Phi(\bG, \bT)}
	{V_\alpha \ne \sset 0}$, and
\begin{align*}
\Xi_1(\phi, \gamma) & = \set{\alpha \in \Xi(\phi)}{\alpha(\gamma_0) = 1}
	= \Xi(\phi) \cap \Phi(\CC\bG{0{+}}(\gamma), \bT), \\
\Xi^1(\phi, \gamma) & = \set{\alpha \in \Xi(\phi)}{\alpha(\gamma_0) \ne 1}, \\
\Xi_{\textup{symm}, {-1}}(\phi, \gamma) & = \sett{\alpha \in \Xi(\phi)}
	{${-\alpha} \in \Gamma\dota\alpha$
		and $\alpha(\gamma_0) = -1$}
\subseteq \Xi^1(\phi, \gamma), \\
\Xi_{\textup{symm}}^{-1}(\phi, \gamma) & = \sett{\alpha \in \Xi(\phi)}
	{${-\alpha} \in \Gamma\dota\alpha$
		and $\alpha(\gamma_0) \ne \pm 1$}
\subseteq \Xi^1(\phi, \gamma), \\
\intertext{and}
\Xi^{\textup{symm}}(\phi, \gamma) & = \sett{\alpha \in \Xi(\phi)}
	{${-\alpha} \not\in \Gamma\dota\alpha$
		and $\alpha(\gamma_0) \ne 1$}
\subseteq \Xi^1(\phi, \gamma).
\end{align*}
We will omit $\phi$ and $\gamma$ from the notation when convenient.
Note that all of these sets are $\Gamma \times \sset{\pm 1}$-stable.
(Recall that $(-\alpha)(\gamma)$ is
$\alpha(\gamma)\inv$, \emph{not} $-(\alpha(\gamma))$.)
We denote by
$\dot\Xi_1(\phi, \gamma)$,
$\dot\Xi_{\textup{symm}, {-1}}(\phi, \gamma)$,
and $\dot\Xi_{\textup{symm}}^{-1}(\phi, \gamma)$
sets of representatives for the $\Gamma$-orbits in the
appropriate sets;
and by $\ddot\Xi^{\textup{symm}}(\phi, \gamma)$
a set of representatives for the $\Gamma \times \sset{\pm 1}$-orbits
in $\Xi^{\textup{symm}}$.
Finally, put
$\dot\Xi_{\textup{symm}}(\phi, \gamma)
= \dot\Xi_{\textup{symm}, {-1}} \cup \dot\Xi_{\textup{symm}}^{-1}$,
$\dot\Xi^1(\phi, \gamma)
= \dot\Xi_{\textup{symm}} \cup \pm\ddot\Xi^{\textup{symm}}$,
$\dot\Xi(\phi, \gamma) = \dot\Xi_1 \cup \dot\Xi^1$,
and
$f(\dot\Xi_{\textup{symm}}(\phi, \gamma))
= \sum_{\alpha \in \dot\Xi_{\textup{symm}}}
	f_\alpha$.
\end{notn}

\begin{pn}
\label{prop:theta-tilde-phi}
With notation and assumptions as above, we have
$$
\theta_{\tilde\phi}(\gamma \ltimes 1)
= \smcard{(\CC{G'}{0{+}}(\gamma), \CC G{0{+}}(\gamma))
	_{x, (r, s):(r, s{+})}}^{1/2}
\varepsilon(\phi, \gamma),
$$
where
\indexmem{eps_phi-gamma}{\varepsilon(\phi, \gamma)}
\begin{multline*}
\varepsilon(\phi, \gamma)
= \sgn_\ff(-1)
	^{f(\dot\Xi_{\textup{symm}}(\phi, \gamma))/2}
\dotm
\prod_{\alpha \in \ddot\Xi^{\textup{symm}}(\phi, \gamma)}
	\sgn_{\ff_\alpha}(\alpha(\gamma_0)) \\
\times \prod_{\alpha \in \dot\Xi_{\textup{symm}}^{-1}(\phi, \gamma)}
	\sgn_{\ff_\alpha}(1 - \alpha(\gamma_0)).
\end{multline*}
\end{pn}

We are using Theorem \ref{thm:gerardin:weil}
(Theorem 4.9.1 of \cite{gerardin:weil})
in our calculations.  Since we are only evaluating the
character of the Weil representation at semisimple elements,
it may be convenient for some purposes to use Corollary 4.8.1
of \cite{gerardin:weil} to write the sign $\varepsilon(\phi, \gamma)$
in a different form.  We do not do this here.

\begin{proof}
Write
$\mc V = (G', G)_{x, (r, s):(r, s{+})}$,
$\mc V^{(0{+})} = (\CC{G'}{0{+}}(\gamma), \CC G{0{+}}(\gamma))
	_{x, (r, s):(r, s{+})}$,
$\pmb{\mc V} = (\bG', \bG)(E)_{x, (r, s):(r, s{+})}$,
and
$\mo V = \Lie(\bG', \bG)(E)_{x, (r, s):(r, s{+})}$.
Recall that $\mc V = J/J_+$ carries a symplectic pairing;
we will describe it explicitly below.
By Corollaries 2.3 and 2.4 of \cite{yu:supercuspidal}, we
have an $\Int(\gamma)$-equivariant isomorphism
$\pmb{\mc V} \cong \mo V$
(essentially, the restriction of $\mexp^E_{x, s:r}$)
that restricts to an isomorphism
$j_1 : \mc V \cong \mo V^\Gamma$
(also $\Int(\gamma)$-equivariant, of course).
We have that
$\hat\phi([g_1, g_2]) = \Lambda(X^*[j_1(g_1), j_1(g_2)])$
for $g_1, g_2 \in \mc V$.
The $\Int(\gamma)$-equivariant map
$$
(X_\alpha)_{\alpha \in \dot\Xi}
\mapsto \sum_{\alpha \in \dot\Xi}
	\sum_{\vphantom{\dot\Xi}  \mysigma \in \Gamma/\Gamma_\alpha}
		\mysigma(X_\alpha)
$$
furnishes an isomorphism
$j_2 : \bigoplus_{\alpha \in \dot\Xi} V_\alpha \cong \mo V^\Gamma$
such that
$j_1(\mc V^{(0{+})})$
and
$j_2\bigl(\bigoplus_{\alpha \in \dot\Xi^1} V_\alpha\bigr)$
are complementary.
Thus, there is an $\Int(\gamma)$-equivariant isomorphism
\begin{equation}
\label{eq:weil-V-decomp}
\mc V
\cong \mc V^{(0{+})}
\oplus \bigoplus_{\alpha \in \dot\Xi_{\textup{symm}}}
	V_\alpha
\oplus \bigoplus_{\alpha \in \ddot\Xi^{\textup{symm}}}
	V_{\pm\alpha} =: V,
\end{equation}
where
$V_{\pm\alpha} = V_\alpha \oplus V_{-\alpha}$
for $\alpha \in \ddot\Xi^{\textup{symm}}$.

The (additive) pairing on \mc V is
$(g_1, g_2) \mapsto \tr_{\ff/\F_p} X^*[j_1(g_1), j_1(g_2)]$,
where $\F_p$ is the finite field with $p$ elements.
(In \cite{yu:supercuspidal}*{\S 11}, Yu
works instead with the multiplicative pairing
$(g_1, g_2) \mapsto \hat\phi([g_1, g_2])
$;
but this is identified with our pairing after an appropriate
choice of embedding $\F_p \hookrightarrow \C\cross$.)
We write $B$ for the pairing on $V$ induced by
\eqref{eq:weil-V-decomp}.
Then
$\mc V^{(0{+})}$ is $B$-orthogonal to $V_\alpha$ for
all $\alpha \in \dot\Xi$,
$V_\alpha$ is $B$-orthogonal to $V_\beta$ unless
$-\beta \in \Gamma\dota\alpha$,
and
\begin{equation}
\label{eq:whats-weil-B}
B(X, X')
= e_\alpha\tr_{\ff_\alpha/\F_p} X^*[X, \mysigma_\alpha X']
\quad\text{for $X, X' \in V_\alpha$ with
	$\alpha \in \dot\Xi_{\textup{symm}}$}.
\end{equation}
In particular, the sums on the right-hand side of
\eqref{eq:weil-V-decomp} are $B$-orthogonal.

Write again $\gamma$ for the symplectomorphism of $V$ induced
by $\gamma$.
By Corollary 2.5 of \cite{gerardin:weil},
\begin{equation}
\tag{$\dag$}
\theta_{\tilde\phi}(\gamma \ltimes 1)
= \theta_{W^{\mc V^{(0{+})}}_\zeta}(\gamma)
\dotm
\prod_{\alpha \in \dot\Xi_{\textup{symm}}} \theta_{W^{V_\alpha}_\zeta}(\gamma)
\dotm
\prod_{\alpha \in \ddot\Xi^{\textup{symm}}} \theta_{W^{V_{\pm\alpha}}_\zeta}(\gamma),
\end{equation}
where $\zeta$ is the character induced by the restriction of
$\hat\phi$ to $J_+$ (as in \cite{yu:supercuspidal}*{\S 11}).

By Theorem \ref{thm:gerardin:weil(c)}, since $\gamma$ acts
trivially on $\mc V^{(0{+})}$, we have that
\begin{equation}
\label{eq:weil-trivial}
\theta_{W^{\mc V^{(0{+})}}_\zeta}(\gamma)
= p^{\dim_{\F_p} \mc V^{(0{+})}/2}
= \smcard{(\CC{G'}{0{+}}(\gamma), \CC G{0{+}}(\gamma))
	_{x, (r, s):(r, s{+})}}^{1/2}.
\end{equation}


For $\alpha \in \dot\Xi^1$, we have that $\gamma$ acts on
$V_\alpha$ without non-zero fixed points.  Thus, our remaining
calculations may, and will, use Theorem \ref{thm:gerardin:weil(a)}.

For $\alpha \in \ddot\Xi^{\textup{symm}}$, we have that
$V_\alpha$ is a maximal $\gamma$-invariant totally isotropic
(indeed, a maximal totally isotropic)
subspace of $V_{\pm\alpha}$;
and
$V_\alpha^\perp/V_\alpha = \sset 0$
(the perpendicular taken in $V_{\pm\alpha}$).
Since $\gamma$ acts on $V_\alpha \cong \ff_\alpha$ by multiplication
by $\alpha(\gamma_0)$, we have that
$\det_{\F_p}(\gamma\bigr|_{V_\alpha})
= N_{\ff_\alpha/\F_p}(\alpha(\gamma_0))$;
so, by Theorem \ref{thm:gerardin:weil(a)},
\begin{equation}
\label{eq:weil-non-symm}
\theta_{W^{V_{\pm\alpha}}_\zeta}(\gamma)
= \sgn_{\F_p}(N_{\ff_\alpha/\F_p}(\alpha(\gamma_0)))
= \sgn_{\ff_\alpha}(\alpha(\gamma_0)).
\end{equation}

For $\alpha \in \dot\Xi_{\textup{symm}, {-1}}$,
the restriction of $B$ to $V_\alpha$ is nondegenerate,
so $\dim_{\F_p} V_\alpha$ is even.
Any maximal totally isotropic subspace of $V_\alpha$ is
$\gamma$-invariant, so reasoning as above gives
\begin{equation}
\label{eq:weil-ram-symm}
\theta_{W_\zeta^{V_\alpha}}(\gamma)
= \sgn_{\F_p}(-1)^{\dim_{\F_p} V_\alpha/2}
= \sgn_\ff(-1)^{\dim_\ff V_\alpha/2}
= \sgn_\ff(-1)^{f_\alpha/2}.
\end{equation}
(We have used the fact that, if $\indx\ff{\F_p}$ is even,
then $\sgn_\ff(-1) = 1$; and, if $\indx\ff{\F_p}$ is odd,
then $\sgn_\ff(-1) = \sgn_{\F_p}(-1)$.)

Finally, fix $\alpha \in \dot\Xi_{\textup{symm}}^{-1}$.
Note that $\Gal(F\sep/F_\alpha\unram)$ acts on
$\mo V_\alpha$ by a linear character, and
$\Gal(F_\alpha\unram/F_\alpha)$ acts on $\mo V_\alpha$ via
the natural projection to $\Gal(\ol\ff/\ff_\alpha)$.
Thus, $V_\alpha = \mo V_\alpha^{\Gal(\ol\ff/\ff_\alpha)}$
is a $1$-dimensional $\ff_\alpha$-vector space.
Choose $X_\alpha \in V_\alpha \setminus \sset 0$,
hence an $\ff_\alpha$-linear isomorphism
$\iota_\alpha : V_\alpha \to \ff_\alpha$, and put
$c_\alpha
= X^*[X_\alpha, \mysigma_\alpha X_\alpha] \in \ff_\alpha$,
so that $\mysigma_\alpha c_\alpha = -c_\alpha$.
By \eqref{eq:whats-weil-B},
$\iota_\alpha$ identifies the restriction of $B$
to $V_\alpha$ with the pairing
$(t_1, t_2)
\mapsto
e_\alpha\tr_{\ff_\alpha/\F_p}(
	c_\alpha\dotm t_1\mysigma_\alpha(t_2)
)$
on $\ff_\alpha$.

Suppose that $W$ is a $\gamma$-invariant totally $B$-isotropic
$\F_p$-subspace of $V_\alpha$.
By abuse of notation, we identify $W$ with its image
$\iota_\alpha(W) \subseteq \ff_\alpha$.  Since $\gamma$ acts on $W$
by multiplication by $\alpha(\gamma_0)$, we have that $W$ is
an $\F_p[\alpha(\gamma_0)]$-subspace.
Further, for $t_1, t_2 \in W$, we have that
\begin{multline*}
e_\alpha\tr_{\F_p[\alpha(\gamma_0)]/\F_p}\bigl(
	\lambda
	\tr_{\ff_\alpha/\F_p[\alpha(\gamma_0)]}(
		c_\alpha\dotm t_1\mysigma_\alpha(t_2)
	)
\bigr) \\
= e_\alpha\tr_{\ff_\alpha/\F_p}(
	c_\alpha\dotm(\lambda t_1)\mysigma_\alpha(t_2)
) = 0
\end{multline*}
for all $\lambda \in \ff[\alpha(\gamma_0)]$; so
$B'(t_1, t_2) := \tr_{\ff_\alpha/\F_p[\alpha(\gamma_0)]}(
	c_\alpha\dotm t_1\mysigma_\alpha(t_2)
) = 0$
(since $e_\alpha$ is not divisible by $p$).
That is, $W$ is totally $B'$-isotropic.
Conversely, it is clear that an $\F_p[\alpha(\gamma_0)]$-subspace of
$\ff_\alpha$ that is totally $B'$-isotropic
is carried by $\iota_\alpha\inv$ to a $\gamma$-invariant
totally $B$-isotropic $\F_p$-subspace of $V_\alpha$.

Let
$\ff_\alpha
= (\ff_\alpha)_+ \oplus (\ff_\alpha)_0 \oplus (\ff_\alpha)_-$
be a Witt decomposition for $B'$,
so that $(\ff_\alpha)_+$ is a maximal
totally $B'$-isotropic $\F_p[\alpha(\gamma_0)]$-subspace
of $\ff_\alpha$.
We have that $B'$ is
$({-1}, \mysigma_\alpha)$-Hermitian.
Denote by $\ol{\alpha(\gamma_0)}$ the image in $\ff_\alpha$ of
$\alpha(\gamma_0)$.
Since
$\mysigma_\alpha\ol{\alpha(\gamma_0)}
= \ol{\alpha(\gamma_0)}\,\inv \ne \ol{\alpha(\gamma_0)}$,
we have that $\mysigma_\alpha$ is non-trivial on
$\F_p[\alpha(\gamma_0)]$.
We record two consequences of this fact.
\begin{itemize}
\item $\F_p[\alpha(\gamma_0)]$ is not contained in the
unique quadratic subfield
$\ff_{\pm\alpha} = \ff_\alpha^{\mysigma_\alpha}$ of
$\ff_\alpha$, so $\dim_{\F_p[\alpha(\gamma_0)]} \ff_\alpha$
is odd.
\item Put $Q = \card{\F_p[\alpha(\gamma_0)]^{\mysigma_\alpha}}$.
Then also
$\mysigma_\alpha\ol{\alpha(\gamma_0)} =
\smash{\ol{\alpha(\gamma_0)}}^Q$.
This means that $\smash{\ol{\alpha(\gamma_0)}}^{Q + 1} = 1$,
so
\begin{equation}
\tag{$*$}
\sgn_{\F_p[\alpha(\gamma_0)]}(\alpha(\gamma_0))
= \smash{\ol{\alpha(\gamma_0)}}^{(Q^2 - 1)/2}
= 1.
\end{equation}
\end{itemize}
Now Lemma \ref{lem:witt-index} gives
$$
\dim (\ff_\alpha)_+
= (\dim \ff_\alpha - 1)/2
= \dim (\ff_\alpha)_-,
$$
so $\dim (\ff_\alpha)_0 = 1$
(all dimensions being taken over $\F_p[\alpha(\gamma_0)]$).
Put
$(V_\alpha)_\varepsilon
= \iota_\alpha\inv\bigl((\ff_\alpha)_\varepsilon\bigr)$,
where $\varepsilon \in \sset{{+}, {-}, 0}$.
Then $(V_\alpha)_+$ is a maximal $\gamma$-stable,
totally $B$-isotropic $\F_p$-subspace of $V_\alpha$, and
$(V_\alpha)_+^\perp/(V_\alpha)_+ \cong (V_\alpha)_0$
(the perpendicular taken in $V_\alpha$).
Since $\gamma$ acts on $V_\alpha \cong \ff_\alpha$ by
multiplication by $\alpha(\gamma_0)$, we have that
$\det_{\F_p}(\gamma\bigr|_{(V_\alpha)_+})
= N_{\F_p[\alpha(\gamma_0)]/\F_p}(\alpha(\gamma_0))
	^{(\dim \ff_\alpha  - 1)/2}$
and
$\det_{\F_p}(\gamma - 1\bigr|_{(V_\alpha)_0})
= N_{\F_p[\alpha(\gamma_0)]/\F_p}(\alpha(\gamma_0) - 1)$;
so, by Theorem \ref{thm:gerardin:weil(a)} and ($*$),
\begin{equation}
\label{eq:weil-unram-symm}
\begin{aligned}
\theta_{W^{V_\alpha}_\zeta}(\gamma)
= {} & \sgn_{\F_p}(-1)^{\dim_{\F_p} \F_p[\alpha(\gamma_0)]/2}
	\sgn_{\F_p[\alpha(\gamma_0)]}(\alpha(\gamma_0) - 1) \\
     & \quad \times\sgn_{\F_p[\alpha(\gamma_0)]}(\alpha(\gamma_0))
	^{(\dim \ff_\alpha - 1)/2} \\
= {} & \sgn_\ff(-1)^{f_\alpha/2}\sgn_{\ff_\alpha}(\alpha(\gamma_0) - 1)
\end{aligned}
\end{equation}
(where we have used in the last equality the facts that
$f_\alpha/2$
is the product of $\dim_{\F_p} \F_p[\alpha(\gamma_0)]/2$
and the odd number $\dim_{\F_p[\alpha(\gamma_0)]} \ff_\alpha$,
and that
$\sgn_{\F_p[\alpha(\gamma_0)]} = \sgn_{\ff_\alpha}$
on $\F_p[\alpha(\gamma_0)]\cross$).

Upon combining ($\dag$) with
\eqref{eq:weil-trivial}--\eqref{eq:weil-unram-symm},
we obtain the desired formula.
\end{proof}

\section{Vanishing results for {$\theta_\sigma$}}
\label{sec:vanishing}

In this section, we consider the support of the character
of the representation $\sigma = \sigma_d$ of
$K_\sigma = \stab_{G'}(\ox)G_{x, 0{+}}$ induced from the
representation $\rho' = \rho_d'$ of $K = K^d$.
We will show that it is much smaller than would na\"\i vely
be expected from an understanding of the support of the
character of $\rho'$ (or, better, of $\tilde\rho$).

Recall that we have associated to $\phi = \phi_{d - 1}$ an
element $X^* = X^*_{d - 1} \in \mf z(\fg')^*_{-r} + \fg'^*_{x, (-r){+}}$\,.
For any finite, tamely ramified extension $E/F$, let
$\hat\phi_E$ be the linear character of
$\bG(E)_{x, s{+}:r{+}}$ such that
$\hat\phi_E \circ \mexp^E_{x, s{+}:r{+}}
= \Lambda \circ X^*\bigr|_{\Lie(\bG)(E)_{x, s{+}:r{+}}}$.
We will also view $\hat\phi_E$ as a character of
$\bG(E)_{x, s{+}}$.
This is analogous to the definition of $\hat\phi$ in
\S\ref{sec:JK}; in fact,
$\hat\phi_E\bigr|_{G_{x, s{+}}}
= \hat\phi\bigr|_{G_{x, s{+}}}$\,.

\begin{lm}
\label{lem:phi-trivial}
Suppose that
\begin{itemize}
\item $E/F$ is a finite, tamely ramified extension,
\item $d \in \R$ and $d > s$,
\item \bH is a reductive, compatibly filtered
$E$-subgroup of \bG
(as in Definition \xref{exp-defn:compatibly-filtered} of \cite{adler-spice:good-expansions}),
\item \bH contains an $E$-split maximal torus \bT in $\bG'$,
\item $x \in \AA(\bT, E)$,
and
\item $X^* \in \Lie(\bT)^*(E) + \Lie(\bG')^*(E)_{x, (-d){+}}$\,.
\end{itemize}
Then $\hat\phi_E$ is trivial on $(\bH, \bG)(E)_{x, (r{+}, d)}$.
\end{lm}

\begin{proof}
We may, and do, assume, after
replacing $F$ by $E$, that \bT is $F$-split and
$X^* \in Y^* + \fg'^*_{x, (-d){+}}$
with $Y^* \in \mf t^*$.
Further, there is no harm in taking $d \le r$.
Since
$\hat\phi_F = \Lambda \circ X^* \circ \mexp_{x, s{+}:r{+}}\inv$
as maps on $(H, G)_{x, (r{+}, d):(r{+}, r{+})}$
(indeed, on $G_{x, s{+}:r{+}}$),
and since, by Hypothesis \ref{hyp:mock-exp-concave},
$\mexp_{x, s{+}:r{+}}$ carries
$\Lie(H, G)_{x, (r{+}, d):(r{+}, r{+})}$
onto $(H, G)_{x, (r{+}, d):(r{+}, r{+})}$,
it suffices to show that
$\Lambda \circ X^*$ is trivial on
$\Lie(H, G)_{x, (r{+}, d)}$.

By an easy analogue of
Proposition \xref{exp-prop:heres-a-gp} of \cite{adler-spice:good-expansions},
$$
\Lie(H, G)_{x, (r{+}, d)}
\subseteq \mf h_{x, r+} \oplus
	(\mf t^\perp \cap \fg_{x, d}),
$$
where
$\mf t^\perp
= \bigoplus_{\alpha \in \Phi(\bG, \bT)} \fg_\alpha$.
Since $\depth_x(X^*) = -r$,
we have $\mf h_{x, r+} \subseteq \ker (\Lambda \circ X^*)$.
For $X \in \mf t^\perp \cap \fg_{x, d}$\,,
we have that
$X^*(X) \equiv Y^*(X) = 0 \pmod{F_{0+}}$,
hence that $\Lambda(X^*(X)) = 1$.
\end{proof}

\begin{rk}
\label{rem:phi-trivial}
Preserve the notation of Lemma \ref{lem:phi-trivial}.
Since
$\hat\phi\bigr|_{G_{x, s{+}}} = \hat\phi_F$ and
$X^* \in \mf z(\fg')^* + \fg'^*_{x, (-r){+}}
	\subseteq \mf t^* + \fg'^*_{x, (-r){+}}$\,,
we always have that $\hat\phi$ is trivial on
$(H, G)_{x, (r{+}, r)}$.
If $X^* \in \mf z(\fg')^* + \fg'^*_{x, -s}$
(for example, if Hypothesis \ref{hyp:X*-central} is
satisfied), then,
by applying Lemma \ref{lem:phi-trivial} to a decreasing sequence of
$d$'s converging to $s$, we see that
$\hat\phi$ is trivial on $(H, G)_{x, (r{+}, s{+})}$.
\end{rk}

\begin{pn}
\label{prop:step1-support}
Suppose that $t \in \R_{\ge 0}$ and
$\gamma \in G$.
If
\begin{itemize}
\item $t < r$,
\item $t \le s$,
or $\gamma$ has a normal $t$-approximation and
$x \in \BB_t(\gamma)$,
and
\item $\gamma \in \stab_{G'}(\ox)G_{x, t}
\smallsetminus
\lsup{G_{x, 0+}}(\stab_{G'}(\ox)G_{x, t+})$,
\end{itemize}
then $\dot\theta_\sigma(\gamma) = 0$.
\end{pn}

\begin{proof}
Note that
$\gamma \not\in \lsup{K_\sigma}(\stab_{G'}(\ox)G_{x, t+})$.
Recall that
$\sigma = \Ind_K^{K_\sigma} \rho'$,
so $\supp \theta_\sigma \subseteq \lsup{K_\sigma}\supp \theta_{\rho'}$.
By Remark \ref{rem:yu-particulars} and
Proposition \ref{prop:howe:weil},
$\supp \theta_{\rho'} \subseteq \lsup K(\stab_{G'}(\ox)G_{x, s+})$.
Thus the result is clear for $t \le s$, so we assume that
$t > s$.

By Proposition \xref{exp-prop:aniso-Levi} of \cite{adler-spice:good-expansions},
there is $k \in G_{x, 0{+}}$
such that $\lsup k\ZZ G t(\gamma) \subseteq G'$.
Since $\theta_\sigma(\lsup k\gamma) = \theta_\sigma(\gamma)$
and $x = k x \in \BB_t(\lsup k\gamma)$,
we may, and do, replace $\gamma$ by $\lsup k\gamma$.

Then, since $x \in \BB_t(\gamma)$, we have that
$\gamma_{\ge t} \in G_{x, t}$\,.
Now, for $h \in \CC G t(\gamma)_{x, r - t}$\,,
we have that $[\gamma\inv, h] = [\gamma_{\ge t}\inv, h] \in G_{x, r}$\,;
so, by Lemma \ref{lem:trho-isotyp},
$$
\theta_\sigma(\gamma) = \theta_\sigma(\lsup h\gamma)
= \theta_\sigma(\gamma)\hat\phi([\gamma_{\ge t}\inv, h])
= \theta_\sigma(\gamma)\cdot[\gamma_{\ge t}, \hat\phi](h),
$$
where $[\gamma_{\ge t}, \hat\phi]$ is the character
$g \mapsto \hat\phi([\gamma_{\ge t}\inv, g])$
of $G_{x, r - t}$\,.

If $[\gamma_{\ge t}, \hat\phi]$ is non-trivial on
$\CC G t(\gamma)_{x, r - t}$\,, then we are done;
so suppose that it is trivial there.
Then consider $h \in (\CC G t(\gamma), G)_{x, ((r - t){+}, r - t)}$.
By Lemma \xref{exp-lem:shallow-comm} of \cite{adler-spice:good-expansions},
$[\gamma_{\ge t}\inv, h] \in (\CC G t(\gamma), G)_{x, (r{+}, r)}$.
By Lemma \ref{lem:phi-trivial}
and Remark \ref{rem:phi-trivial},
we have that $\hat\phi$
is trivial on $(\CC G t(\gamma), G)_{x, (r{+}, r)}$.
Therefore,
$[\gamma_{\ge t}, \hat\phi](h) = \hat\phi([\gamma_{\ge t}\inv, h]) = 1$.
Thus $[\gamma_{\ge t}, \hat\phi]$ is trivial on
$\CC G t(\gamma)_{x, r - t}$ and
$(\CC G t(\gamma), G)_{x, ((r - t){+}, r - t)}$, hence,
by Proposition \xref{exp-prop:heres-a-gp} of \loccit,
on $G_{x, r - t}$\,.
By Lemma \ref{lem:phi-commute}, this means that
$\gamma_{\ge t} \in (G', G)_{x, (t, t{+})}$.
Since $\gamma_{< t} \in \ZZ G r(\gamma) \subseteq G'$
and
$\gamma_{< t} \in \stab_G(\ox)$
(by Remarks \xref{exp-rem:approx-facts-in-center}
and \xref{exp-rem:approx-facts-in-stab} of \loccit,
respectively),
we have
$\gamma = \gamma_{< t}\gamma_{\ge t} \in \stab_{G'}(\ox)G_{x, t+}$\,,
which is a contradiction.
\end{proof}

\begin{cor}
\label{cor:step1-support}
If $\gamma$ has a normal $r$-approximation
and $x \in \BB_r(\gamma)$,
then $\dot\theta_\sigma(\gamma) = 0$ unless
$\gamma_{< r} \in \lsup{G_{x, 0+}}\stab_{G'}(\ox)$.
\end{cor}

\begin{proof}
By Proposition \ref{prop:step1-support},
$\theta_\sigma(\gamma) = 0$ unless
$\gamma \in \lsup{G_{x, 0+}}(\stab_{G'}(\ox)G_{x, r})$.
By Corollary \xref{exp-cor:aniso-Levi} of \cite{adler-spice:good-expansions},
$\gamma \in \lsup{G_{x, 0+}}(\stab_{G'}(\ox)G_{x, r})$
if and only if
$\gamma_{< r} \in \lsup{G_{x, 0+}}\stab_{G'}(\ox)$.
\end{proof}

\begin{cor}
\label{cor:wk-step1-support}
If $\gamma$ has a normal $r$-approximation
and $\bG'/Z(\bG)$ is $F$-anisotropic,
then $\dot\theta_\sigma(\gamma) = 0$ unless
$x \in \BB_r(\gamma)$
and
$\gamma_{< r} \in \lsup{G_{x, 0+}}\stab_{G'}(\ox)$.
\end{cor}

\begin{proof}
This follows from
Lemma \xref{exp-lem:aniso-Brgamma} and
Corollary \xref{exp-cor:aniso-Levi} of
\cite{adler-spice:good-expansions}, and
Corollary \ref{cor:step1-support}.
\end{proof}

\begin{cor}
\label{cor:trho-isotyp}
If $\gamma$ has a normal $r$-approximation and $\bG'/Z(\bG)$
is $F$-anisotropic, then
\begin{multline*}
\dot\theta_\sigma(\gamma)
= \dot\theta_\sigma(\gamma_{< r})
	\cdot
	[G_{x, r}](\gamma_{\ge r})\hat\phi(\gamma_{\ge r}) \\
= \dot\theta_\sigma(\gamma_{< r})
	\cdot
	[\BB(\CC\bG r(\gamma), F)](x)
	\cdot
	[G_{x, r}](\gamma_{\ge r})
	\hat\phi(\gamma_{\ge r}) \\
= [\lsup{G_{x, 0+}}G']
		(\gamma_{< r})
	\theta_\sigma(\gamma_{< r})
	\cdot
	[\BB(\CC\bG r(\gamma), F)](x)
	\cdot
	[G_{x, r}](\gamma_{\ge r})
	\hat\phi(\gamma_{\ge r}).
\end{multline*}
\end{cor}

The notation indicates that
$\dot\theta_\sigma(\gamma)
= \theta_\sigma(\gamma_{< r})\hat\phi(\gamma_{\ge r})$
if all the characteristic functions appearing are $1$, and
$\dot\theta_\sigma(\gamma) = 0$ otherwise.

\begin{proof}
By Corollary \ref{cor:wk-step1-support}
$\dot\theta_\sigma(\gamma) = 0$ unless $\gamma_{\ge r} \in G_{x, r}$\,.
By Lemma \ref{lem:trho-isotyp}, if $\gamma_{\ge r} \in G_{x, r}$\,,
then
$\dot\theta_\sigma(\gamma)
= \dot\theta_\sigma(\gamma_{< r})\hat\phi(\gamma_{\ge r})$.
By Corollary \ref{cor:wk-step1-support} again,
$\dot\theta_\sigma(\gamma_{< r}) = 0$ unless
$x \in \BB(\CC\bG r(\gamma), F)$ and
$\gamma_{< r} \in \lsup{G_{x, 0+}}G'$.
If $\gamma_{< r} \in \lsup{G_{x, 0+}}G'$, then $\gamma_{< r}$ 
is in the domain of $\sigma$, so
$\dot\theta_\sigma(\gamma_{< r}) = \theta_\sigma(\gamma_{< r})$.
\end{proof}

\numberwithin{thm}{subsection}
\numberwithin{equation}{subsection}
\section{Induction to {$\stab_{G'}(\ox)G_{x, 0{+}}$}}
\label{sec:induction1}

We have just shown that the character of $\sigma$ vanishes
``far from $G'$\,''.  In this section, we will compute the
character on a large subset of $\stab_{G'}(\ox)$.
By Lemma \ref{lem:trho-isotyp}, we will then have character
values on a large subset of
$\stab_{G'}(\ox)G_{x, r}$.
(To be more precise, unless certain tameness and compactness conditions
are satisfied, we must place mild restrictions on the
elements at which we evaluate the character.  See the
following paragraph and the beginning of
\S\ref{sec:induction2} for details.)
The resulting formula (see Proposition \ref{prop:induction1})
will be expressed in terms of the character of the
representation $\tau_{d - 1}$ of $\stab_{G'}(\ox)$
induced from the representation $\rho_{d - 1}' \otimes \phi$
of $K^{d - 1}$.

In this section, we suppose that
$\gamma \in G'$
has a normal $r$-approximation
$\ugamma = (\gamma_i)_{0 \le i < r}$ in $G$,
and that $x \in \BB_r(\gamma)$.
In particular, by Remark \xref{exp-rem:approx-facts-in-stab} of \cite{adler-spice:good-expansions},
we have that
$\gamma \in \stab_{G'}(\ox)$.
We will eventually (after Corollary \ref{cor:t-small}) also
require that $\gamma$ be semisimple.

\subsection{The Frobenius formula for $\theta_\sigma$}
\label{ssec:frobenius}

The following
\emph{ad hoc} definitions are useful for cutting down the
number of summands appearing in the Frobenius formula.

\begin{dn}
\label{defn:ijt}
For $g \in G_{x, 0+}$\,, put
\indexmem{jg}{j(g)}%
\indexmem{j-perp-g}{\jperp(g)}%
\begin{gather*}
j(g) = \sup \set{j \in \R_{\ge 0} \cup \sset\infty}
       {g \in \dc G_{x, j}} \\
\intertext{and}
\jperp(g) = \sup \set{j(g'g)}{g' \in G'_{x, 0+}}.
\end{gather*}
If $j(g) < \infty$, put
\indexmem{ig}{i(g)}%
\begin{gather*}
i(g) =
\sup \set{i \in \R \cup \sset\infty}
	{g \in \dc(\CC G i(\gamma), G)_{x, (j(g), j(g){+})}} \\
\intertext{and}
t(g) = \sup \set{\depth_x([\gamma\inv, g h])}{h \in G_{x, j(g)+}}.
\end{gather*}
If $\jperp(g) < \infty$, put
\indexmem{i-perp-g}{\iperp(g)}%
\begin{gather*}
\iperp(g) = \sup \set{i(g'g)}
            {g' \in G'_{x, 0+}, j(g'g) = \jperp(g)} \\
\intertext{and}
\tperp(g) = \sup \set{t(g'g)}
            {g' \in G'_{x, 0+}, j(g'g) = \jperp(g)}.
\end{gather*}
\end{dn}

The numbers $i(g)$ and $j(g)$ are different measures of how
far $g$ is from lying in the group $\dc$.
Remember from
Figure~\ref{fig:sears-tower}
on page \pageref{fig:sears-tower}
that $\dc$ looks somewhat like a skyscraper  
that becomes narrower toward the top.
The vertical direction represents depth,
while horizontal motion toward the center is analogous
to moving through successively smaller full-rank
reductive subgroups of $\bG$
(the connected-centralizer subgroups $\CC\bG i(\gamma)$).
The quantity $j(g)$ tells us the vertical distance from
$g$ down to a roof of $\dc$.
(Of course, if $g \in \dc$, then we have
$j(g) = \infty$.)
The quantity $i(g)$ answers the question:
When we've gone down $j(g)$ floors, landing on a roof of the skyscraper,
how far toward the center must we travel in order to hit a wall?

\begin{rk}
\label{rem:local-const}
Since $g, [\gamma\inv, g] \in G_{x, 0+}$\,, we have that
$j(g), t(g) > 0$.  However, it is possible that $i(g) = 0$.

Since $\set{j \in \R_{\ge 0}}{G_{x, j} \ne G_{x, j+}}$ is
discrete, the supremum in the definition of $j(g)$ is
actually a maximum.
Suppose that $j(g) < \infty$.
Since $\set{i \in \R}{\CC G i(\gamma) \ne \CC G{i+}(\gamma)}$
is discrete,
the supremum in the definition of $i(g)$ is actually a
maximum.
If $[\gamma\inv, g h] = 1$ for some $h \in G_{x, j(g)+}$\,,
then obviously the supremum in the definition of $t(g)$ is a
maximum.  Otherwise,
$h \mapsto \depth_x([\gamma\inv, g h])$ is locally constant
on the compact set $G_{x, j(g)+}$\,, so the supremum in the
definition of $t(g)$ is again a maximum.

By Proposition \xref{exp-prop:heres-a-gp} of \cite{adler-spice:good-expansions},
$$
\dc(\CC G{r - 2j(g)}(\gamma), G)_{x, (j(g), j(g){+})}
\subseteq \dc G_{x, j(g)+}\,,
$$
so $i(g) < r - 2j(g)$.

If $h \in G_{x, j(g)+}$\,, then
$i(g h) = i(g)$, $j(g h) = j(g)$, and $t(g h) = t(g)$.
Since $G_{x, s} \subseteq \dc$,
we have that $j(g) < \infty$ if and only if $j(g) < s$,
so the functions $i$, $j$, and $t$ are all invariant under
translation by $G_{x, s}$\,.

The function $g' \mapsto j(g'g)$ is locally constant
on the compact set $G'_{x, 0+}$\,, so the
supremum in the definition of $\jperp(g)$ is actually a maximum.  If
$\jperp(g) < \infty$, then $g' \mapsto i(g'g)$ and $g'
\mapsto t(g'g)$ are also locally constant
on the compact set $\set{g' \in G'_{x, 0+}}{j(g'g) = \jperp(g)}$,
so the suprema in the definitions of $\iperp(g)$ and $\tperp(g)$
are also maxima.

As above, if $\jperp(g) < \infty$ and $h \in G_{x, \jperp(g)+}$\,,
then $\iperp(g h) = \iperp(g)$, $\jperp(g h) = \jperp(g)$, and
$\tperp(g h) = \tperp(g)$.  Furthermore, $\jperp(g) < \infty$ if and
only if $\jperp(g) < s$, so the functions \iperp, \jperp, and \tperp are
all invariant under translation by $G_{x, s}$\,.  Obviously,
they are also invariant under left translation by $G'_{x, 0+}$\,,
hence
(by Proposition \xref{exp-prop:heres-a-gp} of \cite{adler-spice:good-expansions})
also by
$(G', G)_{x, (0{+}, s)}$ and (if $\jperp(g) < \infty$) by
$(G', G)_{x, (0{+}, \jperp(g){+})}$.
\end{rk}

\begin{lm}
\label{lem:t=i+j}
Fix $g \in G_{x, 0+}$\,.  If $j(g) < \infty$, then $t(g) = i(g) + j(g)$.
\end{lm}

\begin{proof}
Put $i_0 = i(g)$, $j_0 = j(g)$, and $t_0 = t(g)$,
so $g \in \dc(\CC G{i_0}(\gamma), G)_{x, (j_0, j_0{+})}$.

Since $i_0 < r - 2j_0$, we have by
Proposition \xref{exp-prop:heres-a-gp}
and Remark \xref{exp-rem:bracket-facts-decomp} of \cite{adler-spice:good-expansions}
that
$$
\dc(\CC G{i_0}(\gamma), G)_{x, (j_0, j_0{+})}
= \dc^{(j_0)}\CC G{i_0}(\gamma)_{x, j_0}G_{x, j_0+}\,.
$$
Choose
$h \in \dc^{(j_0)}g G_{x, j_0+} \cap
\CC G{i_0}(\gamma)_{x, j_0}$\,.
Then there is $k_1 \in G_{x, j_0+}$ such that
$h \in \dc^{(j_0)}g k_1$.

Note that
$[\gamma\inv, h] = [\gamma_{\ge i_0}\inv, h] \in G_{x, i_0 + j_0}$\,.
If $[\gamma\inv, h] \in G_{x, (i_0 + j_0)+}$\,, then
Lemma \xref{exp-lem:bracket} of \loccit gives
$h \in (\CC G{i_0{+}}(\gamma), G)_{x, (j_0, j_0{+})}$,
so $g \in \dc(\CC G{i_0{+}}(\gamma), G)_{x, (j_0, j_0{+})}$,
contradicting the definition of $i_0$.
Thus $\depth_x([\gamma\inv, h]) = i_0 + j_0$.

Suppose that $[\gamma\inv, h] \in G_{x, t_0+}$\,.
Then, by
Remarks \xref{exp-rem:bracket-facts-containment}
and \xref{exp-rem:approx-facts-commutator} of \loccit,
the fact that
$g k_1 h\inv \in \dc^{(j_0)}$
implies that
$[\gamma\inv, g k_1 h\inv] \in G_{x, r - j_0}$\,, hence that
$[\gamma\inv, h] \equiv [\gamma\inv, g k_1] \pmod{G_{x, r - j_0}}$.
By the definition of $t(g)$, we have
$[\gamma\inv, g k_1] \not\in G_{x, t_0+}$\,.
Thus $r - j_0 \le t_0$, so
$[\gamma\inv, h] \in G_{x, r - j_0}$\,.
By Lemma \xref{exp-lem:bracket} of \loccit,
$h \in (\CC G{r - 2j_0}(\gamma), G)_{x, (j_0, j_0{+})}
\subseteq (\CC G{i_0+}(\gamma), G)_{x, (j_0, j_0{+})}$,
so $g \in \dc(\CC G{i_0+}(\gamma), G)_{x, (j_0, j_0{+})}$,
contradicting the definition of $i(g)$.
 
Thus $[\gamma\inv, h] \not\in G_{x, t_0{+}}$\,, so
$i_0 + j_0 \le t_0$.  Suppose that $i_0 + j_0 < t_0$.
Then there is $k_2 \in G_{x, j_0+}$ such that
$[\gamma\inv, g k_2] \in G_{x, t_0} \subseteq G_{x, (i_0 + j_0)+}$\,.
We have $h \in \dc^{(j_0)}G_{x, j_0+} g k_2$ ---
say, $h = h'g k_2$, with $h' \in \dc^{(j_0)}G_{x, j_0{+}}$\,.
Since
$\dc^{(j_0)}
\subseteq \odc{\gamma_{\ge i_0}; x, r}^{(j_0)}
\subseteq [\gamma_{\ge i_0}; x, (i_0 + j_0){+}]$,
and clearly
$G_{x, j_0{+}}
\subseteq [\gamma_{\ge i_0}; x, (i_0 + j_0){+}]$,
we have that
$$
[\gamma\inv, h] = [\gamma_{\ge i_0}\inv, h]
= [\gamma_{\ge i_0}\inv, h']
\dotm\lsup{h'}[\gamma_{\ge i_0}\inv, g k_2]
\in G_{x, (i_0 + j_0){+}}
	\dotm\lsup{h'}[\gamma_{\ge i_0}\inv, g k_2].
$$
By Lemma \xref{exp-lem:bracket} of \cite{adler-spice:good-expansions},
$g k_2 \in [\gamma; x, t_0]$,
where $[\gamma; x, t_0]$ is as in
Definition \xref{exp-defn:fancy-centralizer-no-underline} of \loccit
Since $[\gamma; x, t_0] \subseteq [\gamma_{\ge i_0}; x, t_0]$,
we have by Remark \xref{exp-rem:approx-facts-commutator} of
\loccit that
$[\gamma_{\ge i_0}\inv, g k_2]
\in G_{x, t_0} \subseteq G_{x, (i_0 + j_0)+}$\,,
so also $[\gamma\inv, h] \in G_{x, (i_0 + j_0)+}$\,,
a contradiction.
\end{proof}

\begin{cor}
\label{cor:t=i+j}
Fix $g \in G_{x, 0+}$\,.  If $\jperp(g) < \infty$, then
$\tperp(g) = \iperp(g) + \jperp(g)$.
\end{cor}

\begin{proof}
By Remark \ref{rem:local-const},
there is some $g' \in G'_{x, 0+}$ such that
$i(g'g) = \iperp(g)$ and $j(g'g) = \jperp(g)$.
Thus
$\tperp(g) \ge t(g'g) = i(g'g) + j(g'g) = \iperp(g) + \jperp(g)$.
Similarly, there is some $g'' \in G'_{x, 0+}$ such that
$j(g''g) = \jperp(g)$ and $t(g''g) = \tperp(g)$.
Thus
$\tperp(g) = t(g''g) = i(g''g) + j(g''g) \le \iperp(g) + \jperp(g)$.
\end{proof}

\begin{cor}
\label{cor:t-small}
Fix $g \in G_{x, 0+}$\,.  If $j(g) < \infty$, then
$t(g) < r - j(g)$.
If $\jperp(g) < \infty$, then $\tperp(g) < r - \jperp(g)$.
\end{cor}

From now on, suppose that $\gamma$ is semisimple.
This condition can be removed if desired by
observing that all the functions in which we will be
interested are locally constant;
but,
since we will only use the main result of this section
(Proposition \ref{prop:induction1}) when the semisimplicity
condition is already satisfied,
it is not a serious restriction.

By Lemma \xref{exp-lem:G-approx-is-G'-approx} of \cite{adler-spice:good-expansions},
\ugamma is a normal $r$-approximation to $\gamma$ in $G'$,
so that it makes sense to speak of groups such as
$\CC{G'}{\iperp(g)}(\gamma)$ below.

\begin{lm}
\label{lem:S-in-BC}
Fix $g \in G_{x, 0+}$ with $\jperp(g) < \infty$.
There is
$\gperp \in (G', G)_{x, (0{+}, \jperp(g){+})}\dotm g$
such that
$$
[\gamma\inv, \gperp]
\in (\CC{G'}{\iperp(g)}(\gamma), \CC G{\iperp(g)}(\gamma))
	_{x, (\tperp(g){+}, \tperp(g))}.
$$
\end{lm}

\begin{proof}
%
%
Put $i_0 = \iperp(g)$, $j_0 = \jperp(g)$, and $t_0 = \tperp(g)$,
so $i_0 + j_0 = t_0$ (by Corollary \ref{cor:t=i+j})
and $t_0 < r - j_0$ (by Corollary \ref{cor:t-small}).
Put also $\bH = \CC\bG{i_0}(\gamma)$
and $\bH' = \CC{\bG'}{i_0}(\gamma)$.
By
Remark \ref{rem:local-const},
there is $g' \in G'_{x, 0+}$
such that $i(g'g) = i_0$ and $j(g'g) = j_0$.
By Lemma \ref{lem:t=i+j}, we have that
$t(g'g) = i(g'g) + j(g'g) = i_0 + j_0 = t_0$.
In particular,
$g'g \in \dc(H, G)_{x, (j_0, j_0{+})}$.
By Remark \xref{exp-rem:bracket-facts-decomp}
and Proposition \xref{exp-prop:heres-a-gp} of \cite{adler-spice:good-expansions},
\begin{gather*}
\dc(H, G)_{x, (j_0, j_0{+})}
= \dc^{(j_0)}(H, G)_{x, (j_0, j_0{+})} \\
\intertext{and}
(H, G)_{x, (j_0, j_0{+})}
= (H', G)_{x, (j_0, j_0{+})}
	(H', H)_{x, (j_0{+}, j_0)}.
\end{gather*}
Since the commutator of $G_{x, 0{+}}$ with
$(H, G)_{x, (j_0, j_0{+})} \subseteq G_{x, j_0}$
lies in
$G_{x, j_0{+}} \subseteq (H, G)_{x, (j_0, j_0{+})}$,
we have that $\dc^{(j_0)} \subseteq G_{x, 0{+}}$
normalizes $(H, G)_{x, (j_0, j_0{+})}$.
Thus we may write $g'g = k'g_1k_-$, with
$k' \in (H', G)_{x, (j_0, j_0{+})}$,
$g_1
\in (H', H)_{x, (j_0{+}, j_0)}$,
and
$k_- \in \dc^{(j_0)}$.
Since $\gamma_{\ge i_0} \in H'_{x, i_0}$
(or $\stab_{H'}(\ox)$, if $i_0 = 0$),
Lemma \xref{exp-lem:shallow-comm}
(or Corollary \xref{exp-cor:stab-norm}, if $i_0 = 0$)
of \cite{adler-spice:good-expansions} gives
$$
[\gamma\inv, g_1]
= [\gamma_{\ge i_0}\inv, g_1]
\in (H', H)_{x, (t_0{+}, t_0)}.
$$
Further,
Remarks \xref{exp-rem:bracket-facts-containment} and
\xref{exp-rem:approx-facts-commutator} of \loccit give
$$
[\gamma\inv, k_-] \in H_{x, r - j_0}
\subseteq H_{x, t_0+}\,,
$$
hence $\lsup{g_1}[\gamma\inv, k_-] \in H_{x, t_0{+}}$\,.
Thus, if $\gperp = g_1 k_-$, then
$$
[\gamma\inv, \gperp]
= [\gamma\inv, g_1]\dotm\lsup{g_1}[\gamma\inv, k_-]
\in (H', H)_{x, (t_0{+}, t_0)}.
$$
Since
$\gperp g\inv = k'^{-1}g' \in (G', G)_{x, (0{+}, j_0{+})}$,
we are done.
\end{proof}

\subsection{Gauss sums}
\label{ssec:gauss}

In this section, we consider (in the context of our
character computations) certain sums associated to
non-degenerate quadratic forms on vector spaces over finite
fields.  We call these sums Gauss sums since, for
a $1$-dimensional vector space, they are a special case of
classical Gauss sums (see
\cite{lidl-niederreiter:finite-fields}*{\S5.2}).
We begin with a simple result that computes such objects.

\begin{defn}
\label{defn:basic-gauss-sum}
Recall that, in \S\ref{sec:generalities}, we chose a square
root $\sqrt{-1}$ of $-1$ and used it to construct a
non-trivial additive character $\Lambda$ of \ff.  Put
$$
\mf G_\Lambda(\ff)
= \begin{cases}
-(-1)^{\log_p \card\ff}, & p \equiv 1 \pmod 4 \\
(-\sqrt{-1})^{\log_p \card\ff}, & p \equiv 3 \pmod 4.
\end{cases}
$$
\end{defn}

\begin{lm}
\label{lem:well-known-gauss}
Let $V$ be a finite-dimensional \ff-vector
space, and $B$ a non-degenerate symmetric bilinear form on
$V$.
Put
$\mf G(V, B) = \card V^{-1/2}\sum_{v \in V} \Lambda(B(v, v))$.
Then
$$
\mf G(V, B) = \sgn_\ff(\det B)\mf G_\Lambda(\ff)^{\dim_\ff V}.
$$
\end{lm}

\begin{proof}
Notice that, if $V$ is $1$-dimensional and
$v_0 \in V \smallsetminus \sset 0$, then
\begin{multline*}
\sum_{v \in V} \Lambda(B(v, v))
= \sum_{t \in \ff} \Lambda(t^2 B(v_0, v_0))
\\
= \sum_{t \in \ff} \sgn_\ff(t)\Lambda(t B(v_0, v_0))
= \sgn_\ff(B(v_0, v_0))\sum_{t \in \ff} \sgn_\ff(t)\Lambda(t),
\end{multline*}
where $\sgn_\ff(0) = 1$.
By Theorem 5.15 of \cite{lidl-niederreiter:finite-fields},
this latter sum equals
$\sgn_\ff(B(v_0, v_0))\mf G_\Lambda(\ff)
= \sgn_\ff(\det B)\mf G_\Lambda(\ff)^{\dim V}$.
Now notice that, if $V = \bigoplus_{i \in I} V_i$ is an orthogonal
direct sum decomposition and, for $i \in I$,
$B_i$ denotes the restriction of $B$ to $V_i \times V_i$,
then we have (with the obvious notation)
$\mf G(V, B) = \prod_{i \in I} \mf G(V_i, B_i)$.
By Theorem 6.21 of \cite{lidl-niederreiter:finite-fields},
we are done.
\end{proof}

\begin{rem}
It is also possible to compute a Gauss sum
as in Lemma \ref{lem:well-known-gauss} by re-writing it as
$\sum_{b \in \ff} N_b\Lambda(b)$, where
$N_b = \card{\set{v \in V}{B(v, v) = b}}$ for all $b \in \ff$.
We can then use the explicit computations of $N_b$ in
Theorems 6.26 and 6.27 of
\cite{lidl-niederreiter:finite-fields},
together with the fact that $\mf G_\Lambda(\ff)^2 = \sgn_\ff(-1)$, to
obtain the desired result.
\end{rem}

\begin{defn}
Put
\indexmem{tG-phi-gamma}{\wtilde{\mf G}(\phi, \gamma)}%
\indexmem{G-phi-gamma}{\mf G(\phi, \gamma)}%
\begin{gather*}
\wtilde{\mf G}(\phi, \gamma)
= \sum_{
	g \in \dc_{G'}^{(s)}\CC G{0{+}}(\gamma)_{x, s}
		\backslash\dc^{(s)}
}
	\hat\phi([\gamma\inv, g]) \\
\intertext{and}
\mf G(\phi, \gamma)
= \smabs{\wtilde{\mf G}(\phi, \gamma)}\inv
	\wtilde{\mf G}(\phi, \gamma).
\end{gather*}
\end{defn}

In this subsection, we will compute
$\smabs{\wtilde{\mf G}}
= \smabs{\wtilde{\mf G}(\phi, \gamma)}$
and
$\mf G = \mf G(\phi, \gamma)$.
The proof of the main result
(Proposition \ref{prop:gauss-sum}) is quite close, in
structure and content, to that of Proposition
\ref{prop:theta-tilde-phi}.
A similar result appears in
\cite{waldspurger:loc-trace-form}*{\S VIII.5}.

Recall (see Definition \xref{exp-defn:r-approx} of
\cite{adler-spice:good-expansions}) that,
since \ugamma is a normal $r$-approximation to $\gamma$ in
$G'$, in particular
$\gamma_{< r}$ is tame in $\bG'$.
Let \bT be a maximal $F$-tame
(hence, since $\bG'$ is $F$-tame, an $F$-tame maximal)
torus in $\bG'$ containing $\gamma_{< r}$.
Let $E/F$ be the splitting field of \bT.
Recall that we defined, at the beginning of \S\ref{sec:induction1}, a
character $\hat\phi_E$ of $\bG(E)_{x, s{+}}$ that extends
$\hat\phi\bigr|_{G_{x, s{+}}}$ and is trivial on $\bG(E)_{x, r{+}}$\,.

For $g, g_1, g_2 \in \dc^{(s)}_{\bG(E)}$, put
$$
\fQ(g) = \hat\phi_E([\gamma\inv, g])
\quad\text{and}\quad
\fB(g_1, g_2) = \hat\phi_E\bigl(
	\bigl[[g_2, \gamma\inv], g_1\bigr]
\bigr).
$$
Except in Corollary \ref{cor:Q-and-B}, we will be
interested only in the restrictions of \fQ and \fB
to $\dc^{(s)}$ (respectively, $\dc^{(s)} \times \dc^{(s)}$); however, we could
not find a proof of Corollary \ref{cor:Q-and-B} that did not
involve passing to extension fields.
We will show that \fQ is, in some sense, a quadratic
form (see Corollary \ref{cor:Q-and-B}), so that we can realize
\mf G as a Gauss sum (see Proposition \ref{prop:gauss-sum}).

It is straightforward to verify that
\begin{equation}
\label{eq:Q-and-B}
\fQ(g_1 g_2) = \fQ(g_1)\fQ(g_2)\fB(g_1, g_2)
\quad\text{for $g_1, g_2 \in \dc^{(s)}_{\bG(E)}$.}
\end{equation}

By Proposition \xref{exp-prop:heres-a-gp} of
\cite{adler-spice:good-expansions}, any element
$h \in \dc^{(s)}_{\bG(E)}$ may be written as
\begin{equation}
\label{eq:Q-decomp}
h = \prod_{0 < i < r} h_i,
\quad\text{with $h_i \in \CC\bG i(\gamma)(E)_{x, (r - i)/2}$
for $0 < i < r$.}
\end{equation}
If $g_1, g_2 \in \dc^{(s)}_{\bG(E)}$ and
$g_j = \prod_{0 < i < r} g_{j i}$
are decompositions as in \eqref{eq:Q-decomp}
for $j = 1, 2$, then one verifies inductively that
there are elements
$g'_i \in g_{1 i}g_{2 i}\CC\bG i(\gamma)(E)_{x, ((r - i)/2){+}}$
for $0 < i < r$ such that
$\bigl(\prod_{i_0 < i < r} g_{1 i}\bigr)
	\bigl(\prod_{i_0 < i < r} g_{2 i}\bigr)
= \prod_{i_0 < i < r} g'_i$
for $0 \le i_0 < r$.  In particular, for $i_0 = 0$, we
obtain
\begin{equation}
\label{eq:mult-Q-decomp}
g_1 g_2 = \prod_{0 < i < r} g'_i.
\end{equation}

\begin{lm}
\label{lem:Q-decomp}
If $g \in \dc^{(s)}_{\bG(E)}$ and
$g = \prod_{0 < i < r} g_i$ is a decomposition
as in \eqref{eq:Q-decomp}, then
$\fQ(g) = \prod_{0 < i < r} \fQ(g_i)$.
\end{lm}

\begin{proof}
One verifies inductively that
$[\gamma\inv, g]
\equiv \prod_{i_0 < i < r} [\gamma\inv, g_i]
\pmod{\bG(E)_{x, r{+}}}$
for $0 \le i_0 < r$.
Since $\bG(E)_{x, r{+}} \subseteq \ker \hat\phi_E$, evaluating
$\hat\phi_E$ at both sides of the above identity for $i_0 = 0$
gives the desired result.
\end{proof}

\begin{cor}
\label{cor:Q-const}
\fQ is constant on cosets of
$\odc{\gamma_{< r}; x, r{+}}^{(s)}_{\bG(E)}$.
The restriction of \fQ to $\dc^{(s)}$ is constant on right cosets
of $\dc_{G'}^{(s)}$.
\end{cor}

The appearance of $\odc{\gamma_{< r}; x, r{+}}^{(s)}_{\bG(E)}$ in the
statement of the corollary is somewhat unexpected.
It appears because $\gamma$ itself might not have a normal
$(r{+})$-approximation.
If (as will usually be the case, by Lemma
\xref{exp-lem:simult-approx} of
\cite{adler-spice:good-expansions})
it does have such an approximation, then
$\odc{\gamma; x, r{+}}^{(s)}_{\bG(E)}
= \odc{\gamma_{< r}; x, r{+}}^{(s)}_{\bG(E)}$.

\begin{proof}
By Proposition \xref{exp-prop:heres-a-gp} of
\cite{adler-spice:good-expansions}, any element
$g_+ \in \odc{\gamma_{< r}; x, r{+}}^{(s)}_{\bG(E)}$
may be written as
$g_+ = \prod_{0 < i < r} g_{{+}, i}$,
with
$g_{{+}, i} \in \CC\bG i(\gamma_{< r})(E)_{x, ((r - i)/2){+}}
	= \CC\bG i(\gamma)(E)_{x, ((r - i)/2){+}}$
for $0 < i < r$.
If $g \in \dc^{(s)}_{\bG(E)}$
and $g = \prod_{0 < i < r} g_i$ is a decomposition as in
\eqref{eq:Q-decomp}, then, by \eqref{eq:mult-Q-decomp} and
Lemma \ref{lem:Q-decomp},
$\fQ(g) = \prod_{0 < i < r} \fQ(g_i)$ and
$\fQ(g_+ g) = \prod_{0 < i < r} \fQ(g'_i)$,
where
$g'_i \in g_{{+}, i}g_i\CC\bG i(\gamma)(E)_{x, ((r - i)/2){+}}
	= \CC\bG i(\gamma)(E)_{x, ((r - i)/2){+}} g_i$
for $0 < i < r$,
so it suffices to show that $\fQ(g'_i) = \fQ(g_i)$
for $0 < i < r$.
Indeed, for such $i$, put
$k_i = g'_i g_i\inv \in \CC\bG i(\gamma)(E)_{x, ((r - i)/2){+}}$.
Upon applying Lemma \ref{lem:Q-decomp} again
(this time, with $k_i$
playing the role of $g_{i - \varepsilon}$ for $\varepsilon$
sufficiently small), we obtain that
$\fQ(g'_i) = \fQ(k_i)\fQ(g_i)
	= \hat\phi_E([\gamma\inv, k_i])\fQ(g_i)$.
Further,
$$
[\gamma\inv, k_i] \equiv [\gamma_{< r}\inv, k_i]
	\pmod{\bG(E)_{x, r{+}} \subseteq \ker \hat\phi_E},
$$
so it suffices to show that
$[\gamma_{< r}\inv, k_i] \in \ker \hat\phi_E$.
Since
$$
[\gamma_{< r}\inv, k_i]
\in \bG(E)_{x, ((r + i)/2){+}}
= (\CC\bG r(\gamma), \bG)(E)_{x, (((r + i)/2){+}, ((r + i)/2){+})}
$$
and $\gamma_{< r} \in Z(\CC G r(\gamma))$,
Lemma \xref{exp-lem:center-comm} of
\cite{adler-spice:good-expansions} gives
$$
[\gamma_{< r}\inv, k_i]
\in (\CC\bG r(\gamma), \bG)(E)_{x, (r{+}, ((r + i)/2){+})},
$$
and Lemma \ref{lem:phi-trivial}, Remark
\ref{rem:phi-trivial}, and Hypothesis \ref{hyp:X*-central}
give that
$[\gamma_{< r}\inv, k_i] \in \ker \hat\phi_E$,
as desired.

If $g \in \dc^{(s)}$ and $g' \in \dc_{G'}^{(s)}$,
then, since
$\hat\phi_E\bigr|_{G_{x, s{+}}}
= \hat\phi\bigr|_{G_{x, s{+}}}$
and $\hat\phi$ is invariant under conjugation by
$G'_{x, 0{+}}$\,, we have
\begin{multline*}
\mc Q(g'g) = \hat\phi_E([\gamma\inv, g'g])
= \hat\phi([\gamma\inv, g'g]) \\
= \hat\phi(\gamma\inv)\hat\phi^{g'}(\gamma)
\hat\phi^{g'}([\gamma\inv, g])
= \hat\phi([\gamma\inv, g]) \\
= \hat\phi_E([\gamma\inv, g])
= \mc Q(g).
\text{\qedhere}
\end{multline*}
\end{proof}

\begin{lm}
\label{lem:B-decomp}
If $g_1, g_2 \in \dc^{(s)}_{\bG(E)}$ and
$g_j = \prod_{0 < i < r} g_{j i}$ are decompositions as in
\eqref{eq:Q-decomp} for $j = 1, 2$, then
$\fB(g_1, g_2) = \prod_{0 < i < r} \fB(g_{1 i}, g_{2 i})$.
\end{lm}

\begin{proof}
By \eqref{eq:Q-and-B}, \eqref{eq:mult-Q-decomp},
Lemma \ref{lem:Q-decomp}, and Corollary \ref{cor:Q-const},
we have that, for suitable $g_i'$,
\begin{multline*}
\fQ(g_1)\fQ(g_2)\fB(g_1, g_2)
= \fQ(g_1 g_2) = \prod_{0 < i < r} \fQ(g'_i)
= \prod_{0 < i < r} \fQ(g_{1 i}g_{2 i}) \\
= \prod_{0 < i < r} \fQ(g_{1 i})\fQ(g_{2 i})
	\fB(g_{1 i}, g_{2 i})
= \fQ(g_1)\fQ(g_2)
\prod_{0 < i < r} \fB(g_{1 i}, g_{2 i}).
\text{\qedhere}
\end{multline*}
\end{proof}

For $0 < i < r$, let $Y_i$ be any element of
$\mexp_{x, i:i{+}}\inv(\gamma_i)$.
If $g_1, g_2 \in \dc^{(s)}_{\bG(E)}$ and
$g_j = \prod_{0 < i < r} g_{j i}$
are decompositions
as in \eqref{eq:Q-decomp} for $j = 1, 2$, then put
\begin{equation}
\label{eq:sqrt-B-defn}
\log_\Lambda \sqrt\fB(g_1, g_2)
= \tfrac 1 2\sum_{0 < i < r}
	\ol{X^*\bigl[[Y_i, X_{2i}], X_{1i}\bigr]},
\end{equation}
where
$t \mapsto \ol t$ is the natural map from $E_0$ to $\ff_E$,
and
$X_{j i}
\in (\mexp^E_{x, ((r - i)/2):((r - i)/2){+}})\inv(g_{j i})
	\cap \Lie(\CC\bG i(\gamma))(E)_{x, (r - i)/2}$
for $0 < i < r$ and $j = 1, 2$.
Note that there do exist elements $X_{j i}$ in the indicated
intersection, by Lemma \ref{lem:mock-exp-subgroup}; and that
this definition does not depend on the choices
of the various $Y_i$ and $X_{ji}$.

\begin{lm}
\label{lem:sqrt-B-and-B}
$\log_\Lambda \sqrt\fB$ and \fB are symmetric, and
\begin{align*}
\fB(g_1 g_1', g_2) & = \fB(g_1, g_2)\fB(g_1', g_2), \\
\log_\Lambda \sqrt\fB(g_1 g_1', g_2) & = \log_\Lambda \sqrt\fB(g_1, g_2)
	+ \log_\Lambda \sqrt\fB(g_1', g_2),
\end{align*}
and
$$
\bigl(\Lambda(\log_\Lambda \sqrt\fB(g_1, g_2))\bigr)^2 = \fB(g_1, g_2)
$$
for $g_1, g_1', g_2 \in \dc^{(s)}_{\bG(E)}$.
\end{lm}

Of course, \fB and $\log_\Lambda \sqrt\fB$ exhibit analogous behaviors
in the second variable, by symmetry.

\begin{proof}
The symmetry and multiplicativity of \fB will follow from
those of $\log_\Lambda \sqrt\fB$.

By Lemma \ref{lem:B-decomp} and \eqref{eq:sqrt-B-defn},
it suffices to show the desired facts on
each $\CC\bG i(\gamma)(E)_{x, (r - i)/2}$\,.
Accordingly, fix $0 < i < r$
and $g_1, g_2 \in \CC\bG i(\gamma)(E)_{x, (r - i)/2}$\,.
Let $X_j \in \Lie(\CC\bG i(\gamma))(E)_{x, (r - i)/2}$
satisfy $g_j \in \mexp^E_{x, ((r - i)/2):((r - i)/2){+}}(X_j)$
for $j = 1, 2$.
Then
$$
\bigl[[Y_i, X_1], X_2\bigr]
= \bigl[[Y_i, X_2], X_1\bigr]
+ \bigl[Y_i, [X_1, X_2]\bigr]
$$
by the Jacobi identity.

Since \bT is a maximal torus in
$\CC\bG i(\gamma)$, we have
$\gamma_i \in \ZZ G i(\gamma) \subseteq T$.
Since
$\bigl[\gamma_i, [g_1, g_2]\bigr] \in \CC\bG i(\gamma)(E)_{x, r}$\,,
Lemma \xref{exp-lem:center-comm} of
\cite{adler-spice:good-expansions} gives
$\bigl[\gamma_i, [g_1, g_2]\bigr]
\in (\bT, \CC\bG i(\gamma))(E)_{x, (r{+}, r)}$.
By two applications of Hypothesis \ref{hyp:mock-exp-ad}, we
have that
$\bigl[Y_i, [X_1, X_2]\bigr]
\in (\mexp^E_{x, r:r{+}})\inv\bigl[\gamma_i, [g_1, g_2]\bigr]$,
hence (by Hypothesis \ref{hyp:mock-exp-concave}) that
$$
\bigl[Y_i, [X_1, X_2]\bigr]
\in \Lie(\bT, \CC\bG i(\gamma))(E)_{x, (r{+}, r)}.
$$
By an easy analogue of Proposition
\xref{exp-prop:heres-a-gp} of
\cite{adler-spice:good-expansions},
$\bigl[Y_i, [X_1, X_2]\bigr]
\in \Lie(\bT)(E)_{r{+}}
\oplus (\Lie(\bT)^\perp(E) \cap \Lie(\bG)(E)_{x, r})$,
where
$\Lie(\bT)^\perp
= \bigoplus_{\alpha \in \Phi(\bG, \bT)}
	\Lie(\bG)_\alpha$.
Since
$X^* \in \mf z(\fg')^* + \fg^*_{x, r{+}}
	\subseteq \mf t^* + \fg^*_{x, r{+}}$\,,
we have that
$X^*\bigl[Y_i, [X_1, X_2]\bigr] \in E_{0{+}}$\,,
so
$$
\log_\Lambda \sqrt\fB(g_2, g_1)
= \tfrac 1 2 \ol{ X^*\bigl[[Y_i, X_1], X_2\bigr] }
= \tfrac 1 2 \ol{ X^*\bigl[[Y_i, X_2], X_1\bigr] }
= \log_\Lambda \sqrt\fB(g_1, g_2).
$$
Now fix $g_1' \in \CC\bG i(\gamma)(E)_{x, (r - i)/2}$\,,
and choose $X_1' \in \Lie(\CC\bG i(\gamma))(E)_{x, (r - i)/2}$
such that
$g_1' \in \mexp^E_{x, ((r - i)/2):((r - i)/2){+}}(X_1')$.
Then, since $\mexp^E_{x, ((r - i)/2):((r - i)/2){+}}$ is a
homomorphism, we have that
$g_1 g_1' \in \mexp^E_{x, ((r - i)/2):((r - i)/2){+}}(X_1 + X_1')$,
and it follows immediately that
$\log_\Lambda \sqrt\fB(g_1 g_1', g_2)
= \log_\Lambda \sqrt\fB(g_1, g_2)\log_\Lambda \sqrt\fB(g_1', g_2)$.

Finally, note that
$\bigl[[g_2, \gamma\inv], g_1\bigr]
= \bigl[[g_2, \gamma_{\ge i}\inv], g_1\bigr]
\equiv \bigl[[g_2, \gamma_i\inv], g_1\bigr]
\pmod{G_{x, r{+}}}$
and, by two applications of Hypothesis
\ref{hyp:mock-exp-ad} and the fact that
$\gamma_i\inv \in \mexp^E_{x, r:r{+}}(Y_i)$, that
$\bigl[[g_2, \gamma_i\inv], g_1\bigr]
\in \mexp^E_{x, r:r{+}}\bigl(\bigl[[Y_i, X_2], X_1\bigr]\bigr)$,
so
\begin{multline*}
\fB(g_1, g_2)
= \hat\phi_E\bigl(\bigl[[g_2, \gamma\inv], g_1\bigr]\bigr)
= \Lambda\bigl(X^*\bigl[[Y_i, X_2], X_1\bigr]\bigr) \\
= \Bigl(\Lambda\bigl(\tfrac 1 2
	X^*\bigl[[Y_i, X_2], X_1\bigr]
\bigr)\Bigr)^2
= \bigl(\Lambda(\log_\Lambda \sqrt\fB(g_1, g_2))\bigr)^2.
\text{\qedhere}
\end{multline*}
\end{proof}

\begin{cor}
\label{cor:Q-and-B}
We have $\fQ(g) = \Lambda(\log_\Lambda \sqrt\fB(g, g))$ for $g \in \dc^{(s)}_{\bG(E)}$.
\end{cor}

\begin{proof}
Note that $\gamma_i \in \ZZ G r(\gamma) \subseteq T$ for
$0 < i < r$.
Since the choice of $Y_i \in \mexp_{x, i:i{+}}\inv(\gamma_i)$
does not matter in \eqref{eq:sqrt-B-defn},
by Lemma \ref{lem:mock-exp-subgroup}, we may, and
do, take $Y_i \in \mf t$.
If $0 < i < r$,
$\alpha \in \wtilde\Phi(\CC\bG i(\gamma), \bT)$,
and $g \in \bU_\alpha(E) \cap \bG(E)_{x, (r - i)/2}$\,, then
choose
$$
X \in \Lie(\bU_\alpha)(E) \cap \Lie(\bG)(E)_{x, (r - i)/2}
$$
such that
$$
g \in \mexp^E_{x, ((r - i)/2):((r - i)/2){+}}(X).
$$
(Such an $X$ exists, by Lemma \ref{lem:mock-exp-roots}.)
Since $\Lie(\bU_\alpha)$ is Abelian and preserved by
$\ad(Y_i)$, we have that
$\bigl[[Y_i, X], X\bigr] = 0$,
hence that
$$
\log_\Lambda \sqrt\fB(g, g) = \tfrac 1 2\ol{X^*\bigl[[Y_i, X], X\bigr]} = 0.
$$

If $\alpha = 0$, then
$[\gamma_{< r}\inv, g] = 1 \in \ker \hat\phi_E$.
Otherwise, by Lemma \ref{lem:phi-trivial}, Remark
\ref{rem:phi-trivial}, and Hypothesis \ref{hyp:X*-central},
we have
$$
[\gamma_{< r}\inv, g]
\in \bU_\alpha(E) \cap \bG(E)_{x, (r + i)/2}
\subseteq (\bT, \bG)(E)_{x, (r{+}, s{+})}
\subseteq \ker \hat\phi_E.
$$
In either case,
$$
\fQ(g)
= \hat\phi_E([\gamma\inv, g]) = \hat\phi_E([\gamma_{< r}\inv, g])
= 1 = \Lambda(\log_\Lambda \sqrt\fB(g, g)).
$$

By Definitions \xref{exp-defn:vGvr-split} and
\xref{exp-defn:fancy-centralizer-no-underline} of
\cite{adler-spice:good-expansions},
we have shown equality on a set of semigroup generators for
$\dc^{(s)}_{\bG(E)}$.
Since $g \mapsto \fQ(g)\Lambda(\log_\Lambda \sqrt\fB(g, g))\inv$ is multiplicative
by \eqref{eq:Q-and-B} and Lemma \ref{lem:sqrt-B-and-B},
we have equality everywhere.
\end{proof}

\begin{lm}
\label{lem:B-nondegen}
If $g_2 \in \dc^{(s)}$ is such that
$\fB(g, g_2) = 1$ for all $g \in \dc^{(s)}$, then
$g_2 \in \odc{\gamma; x, r}^{(s)}_{G'}\odc{\gamma_{< r}; x, r{+}}^{(s)}$.
\end{lm}

As remarked after Corollary \ref{cor:Q-const}, the unexpected
appearance of $\odc{\gamma_{< r}; x, r{+}}^{(s)}$ in place
of $\odc{\gamma; x, r{+}}^{(s)}$ compensates for the fact
that $\gamma$ might not have a normal
$(r{+})$-approximation.

\begin{proof}
By Lemma \ref{lem:B-decomp}
(and the fact that, by Proposition
\xref{exp-prop:heres-a-gp} of
\cite{adler-spice:good-expansions},
all the terms in a decomposition \eqref{eq:Q-decomp} of an
$F$-rational element may be taken to be $F$-rational),
it suffices to consider
the case that $g_2 \in \CC G i(\gamma)_{x, (r - i)/2}$
for some $0 < i < r$.
Notice that
$[g_2, \gamma\inv] = [g_2, \gamma_{\ge i}\inv]
\in \CC G i(\gamma)_{x, (r + i)/2}$\,.
Therefore, if
$g \in (\CC G i(\gamma), G)_{x, (((r - i)/2){+}, (r - i)/2)}$,
then Lemma \xref{exp-lem:shallow-comm} of
\cite{adler-spice:good-expansions} gives that
$\bigl[[g_2, \gamma\inv], g\bigr]
\in (\CC G i(\gamma), G)_{x, (r{+}, r)}$,
hence, by Lemma \ref{lem:phi-trivial}
and Remark \ref{rem:phi-trivial},
that
$\bigl[[\gamma\inv, g_2], \hat\phi\bigr](g)
:= \hat\phi\bigl(\bigl[[g_2, \gamma\inv], g\bigr]\bigr)
= 1$.
That is, $\bigl[[\gamma\inv, g_2], \hat\phi\bigr]$
is trivial on $(\CC G i(\gamma), G)_{x, (r{+}, r)}$.
By hypothesis, it is also trivial on
$\CC G i(\gamma)_{x, r}$\,, hence, by
Proposition \xref{exp-prop:heres-a-gp} of
\cite{adler-spice:good-expansions},
on $G_{x, r}$\,.
By Lemma \ref{lem:phi-commute},
$$
g_2 \in (\CC{G'}i(\gamma), \CC G i(\gamma))
	_{x, ((r - i)/2, ((r - i)/2){+})}
\subseteq \odc{\gamma; x, r}^{(s)}_{G'}
	\odc{\gamma_{< r}; x, r{+}}^{(s)},
$$
as desired.
\end{proof}

Via \eqref{eq:Q-and-B}, we may view Lemma
\ref{lem:B-nondegen} as an ``upper bound'' on the size of
the level set of \fQ containing a fixed element $g$.  Since
Corollary \ref{cor:Q-const} describes uniform ``lower
bounds'' on the sizes of the level sets of \fQ,
we have quite precise local constancy
information.

\begin{notn}
\label{notn:gauss}
In addition to Notation \ref{notn:root-constants},
for $\alpha \in \Phi(\bG, \bT)$,
let $\mo V_\alpha$ denote the image of
$$
\Lie(\bG)_\alpha(E) \cap \Lie(\dc^{(s)}_{\bG(E)})
$$
in
$$
\Lie(
	\odc{\gamma; x, r}^{(s)}_{\bG'(E)}
	\odc{\gamma_{< r}; x, r{+}}^{(s)}_{\bG(E)}
)\backslash\Lie(\dc^{(s)}_{\bG(E)}),
$$
and $V_\alpha$ the set of $\Gamma_\alpha$-fixed points in
$\mo V_\alpha$.
(The symbols $\mo V_\alpha$ and $V_\alpha$
had a different meaning in \S\ref{sec:weil}.)
More concretely, we have that
$\mo V_\alpha = \sset 0$
if
$\alpha \in \Phi(\bG', \bT) \cup \Phi(\CC\bG r(\gamma), \bT)$
or
$\alpha \not\in \Phi(\CC\bG{0{+}}(\gamma), \bT)$;
and, if
$\alpha \in \Phi(\CC\bG{0{+}}(\gamma), \bT)
	\smallsetminus
	(\Phi(\bG', \bT) \cup \Phi(\CC\bG r(\gamma), \bT))$,
then
$\mo V_\alpha
\cong \lsub E\mf u_{(\alpha + (r - i)/2):(\alpha + (r - i)/2){+}}$\,,
where $i = \ord(\alpha(\gamma_{< r}) - 1)$
and $\alpha + (r - i)/2$ is the affine root with gradient
$\alpha$ whose value at $x$ is $(r - i)/2$.
Put
$\Upsilon(\phi, \gamma) = \set{\alpha \in \Phi(\bG, \bT)}
	{V_\alpha \ne \sset 0}$,
and
\begin{align*}
\Upsilon_{\textup{symm}, \textup{unram}}(\phi, \gamma)
& = \sett{\alpha \in \Upsilon(\phi, \gamma)}
	{${-\alpha} \in \Gamma\dota\alpha$ and
	$\mysigma_\alpha \ne 1$ on $\ff_\alpha$}, \\
\Upsilon_{\textup{symm}, \textup{ram}}(\phi, \gamma)
& = \sett{\alpha \in \Upsilon(\phi, \gamma)}
	{${-\alpha} \in \Gamma\dota\alpha$ and
	$\mysigma_\alpha = 1$ on $\ff_\alpha$}, \\
\intertext{and}
\Upsilon^{\textup{symm}}(\phi, \gamma)
& = \set{\alpha \in \Upsilon(\phi, \gamma)}
	{-\alpha \not\in \Gamma\dota\alpha}.
\end{align*}
We will omit $\phi$ and $\gamma$ from the notation when convenient.
Note that all of these sets are
$\Gamma \times \sset{\pm 1}$-stable.
We denote by
$\dot\Upsilon_{\textup{symm}, \textup{unram}}(\phi, \gamma)$
and
$\dot\Upsilon_{\textup{symm}, \textup{ram}}(\phi, \gamma)$
sets of representatives for the $\Gamma$-orbits in the appropriate
sets;
and by $\ddot\Upsilon^{\textup{symm}}(\phi, \gamma)$ a set of
representatives for the $\Gamma \times \sset{\pm 1}$-orbits in
$\Upsilon^{\textup{symm}}$.
Finally, put
$\dot\Upsilon_{\textup{symm}}(\phi, \gamma)
= \dot\Upsilon_{\textup{symm}, \textup{unram}}
\cup \dot\Upsilon_{\textup{symm}, \textup{ram}}$,
$\dot\Upsilon(\phi, \gamma)
= \dot\Upsilon_{\textup{symm}} \cup \pm\ddot\Upsilon^{\textup{symm}}$,
and
$f(\dot\Upsilon_{\textup{symm}, \textup{ram}}(\phi, \gamma))
= \sum_{\alpha \in \dot\Upsilon_{\textup{symm}, \textup{ram}}}
	f_\alpha$.
\end{notn}

%

The proof of the following result is a relatively
straightforward application of results from
\cite{adler-spice:good-expansions} that allow us to
combine and manipulate groups of the form $\lsub\bT G_{x, f}$.

\begin{pn}
\label{prop:gauss-sum-card}
\begin{align*}
\smabs{\wtilde{\mf G}(\phi, \gamma)}
= & \Bigindx{\odc{\gamma_{< r}; x, r}}
	{\odc{\gamma_{< r}; x, r}_{G'}G_{x, s}}^{1/2} \\
  & \times\Bigindx{\odc{\gamma_{< r}; x, r{+}}}
	{\odc{\gamma_{< r}; x, r{+}}_{G'}G_{x, s{+}}}^{1/2} \\
  & \times\smcard{(\CC{G'}{0{+}}(\gamma), \CC G{0{+}}(\gamma))
		_{x, (r, s):(r, s{+})}}^{-1/2}
\end{align*}
\end{pn}

\begin{proof}
Since none of the quantities involved change if we replace
$\gamma$ by $\gamma_{< r}$, we do so.
Write
$\mc V
= \dc^{(s)}_{G'}\odc{\gamma; x, r{+}}^{(s)}
	\backslash\dc^{(s)}$.
For $0 < j < s$, we have that
$(\CC{G'}{r - 2j}(\gamma), \CC G{r - 2j}(\gamma))
	_{x, (j, j):(j, j{+})}$
is naturally isomorphic to a subgroup of \mc V.
It is straightforward to check that \mc V is the direct sum
of these subgroups.
Then reasoning as in the proof of Proposition
\ref{prop:theta-tilde-phi} (applied to each direct summand)
shows that
\begin{equation}
\label{eq:gauss-V-decomp}
\mc V
\cong
\bigoplus_{\alpha \in \dot\Upsilon_{\textup{symm}}}
	V_\alpha
\oplus \bigoplus_{\alpha \in \ddot\Upsilon^{\textup{symm}}}
	V_{\pm\alpha} =: V,
\end{equation}
where, as before, we have written
$V_{\pm\alpha} = V_\alpha \oplus V_{-\alpha}$ for
$\alpha \in \ddot\Upsilon^{\textup{symm}}$.
Write $B$ for the pairing on $V$ induced by
\eqref{eq:gauss-V-decomp} (and the pairing $\log_\Lambda \sqrt\fB$ on
\mc V).
Notice that $B$ is \ff-bilinear.
By Lemma \ref{lem:B-nondegen}, $B$ is non-degenerate.

By Corollaries \ref{cor:Q-const} and \ref{cor:Q-and-B},
\begin{align*}
\wtilde{\mf G}
&= \Bigindx{\dc_{G'}^{(s)}\odc{\gamma; x, r{+}}}
	{\dc_{G'}^{(s)}\CC G{0{+}}(\gamma)_{x, s}}
\sum_{g \in \mc V} \fQ(g) \\
&= \Bigindx{\dc_{G'}^{(s)}\odc{\gamma; x, r{+}}}
	{\dc_{G'}^{(s)}\CC G{0{+}}(\gamma)_{x, s}}
\sum_{X \in V} \Lambda(B(X, X)).
\end{align*}
Since $B$ is non-degenerate, we have by Lemma
\ref{lem:well-known-gauss} that (in the notation of that
lemma)
$$
\abs{\sum_{X \in V} \Lambda(B(X, X))}
= \abs{\card V^{1/2}\mf G(V, B)}
= \card V^{1/2}
= \card{\mc V}^{1/2}.
$$
Thus
\begin{equation}
\label{eq:gauss-sum-card-1}
\begin{aligned}
\smabs{\wtilde{\mf G}}
& =  \Bigindx{\dc_{G'}^{(s)}\odc{\gamma; x, r{+}}^{(s)}}
	{\dc_{G'}^{(s)}\CC G{0{+}}(\gamma)_{x, s}} \\
& \qquad\times\Bigindx{\dc^{(s)}}
	{\dc_{G'}^{(s)}\odc{\gamma; x, r{+}}^{(s)}}
	^{1/2} \\
& = \Bigindx{\dc_{G'}^{(s)}\odc{\gamma; x, r{+}}^{(s)}}
	{\dc_{G'}^{(s)}\CC G{0{+}}(\gamma)_{x, s}}^{1/2} \\
& \qquad\times\Bigindx{\dc^{(s)}}
	{\dc_{G'}^{(s)}\CC G{0{+}}(\gamma)_{x, s}}^{1/2}.
\end{aligned}
\end{equation}
(In the second and third expressions being compared, the first terms
are the same,
except that the latter contains an extra exponent of $1/2$.)
By Remarks \xref{exp-rem:bracket-facts-containment}
and \xref{exp-rem:bracket-facts-decomp} of
\cite{adler-spice:good-expansions},
\begin{equation}
\label{eq:dc-card-1}
\dc^{(s)}G_{x, s} = \dc.
\end{equation}
Upon writing $\odc{\gamma; x, r{+}}^{(s)}$ and
$\odc{\gamma; x, r{+}}^{(s{+})}$
as groups of the form $\lsub\bT G_{x, f}$\,,
using Definition
\xref{exp-defn:fancy-centralizer-no-underline} of \loccit,
we see that
$\odc{\gamma; x, r{+}}^{(s)}
= \odc{\gamma; x, r{+}}^{(s{+})}$.
Therefore, again by
Remarks \xref{exp-rem:bracket-facts-containment}
and \xref{exp-rem:bracket-facts-decomp} of
\loccit, we have that
\begin{equation}
\label{eq:dc+-card-1}
\odc{\gamma; x, r{+}}^{(s)}G_{x, s{+}}
	= \odc{\gamma; x, r{+}}.
\end{equation}
By Proposition \xref{exp-prop:heres-a-gp}
and Lemma \xref{exp-lem:more-vGvr-facts} of
\cite{adler-spice:good-expansions},
we have the following
equalities:
\begin{align}
\label{eq:dc-intersect}
\dc^{(s)} \cap \dc_{G'}G_{x, s}
& = \dc_{G'}^{(s)}\CC G{0{+}}(\gamma)_{x, s}\,, \\
\label{eq:dc+-intersect}
\odc{\gamma; x, r{+}}^{(s)}
\cap \odc{\gamma; x, r}_{G'}^{(s)}\CC G{0{+}}(\gamma)_{x, s}
& = \odc{\gamma; x, r{+}}_{G'}^{(s)}\CC G{0{+}}(\gamma)_{x, s}\,, \\
\label{eq:dc+-intersect'}
\odc{\gamma; x, r{+}}^{(s)}
\cap \odc{\gamma; x, r{+}}_{G'}
	(\CC G{0{+}}(\gamma), G)_{x, (s, s{+})}
& = \odc{\gamma; x, r{+}}_{G'}^{(s)}\CC G{0{+}}(\gamma)_{x, s}\,, \\
\intertext{and}
\label{eq:other-intersect}
(\CC{G'}{0{+}}(\gamma), \CC G{0{+}}(\gamma))_{x, (r, s)}
\cap \dc_{G'}G_{x, s{+}}
& = (\CC{G'}{0{+}}(\gamma), \CC G{0{+}}(\gamma))_{x, (r, s{+})}.
\end{align}
By \eqref{eq:dc-card-1} and \eqref{eq:dc-intersect},
we have a bijection
$$
\dc_{G'}^{(s)}\CC G{0{+}}(\gamma)_{x, s}
	\backslash
\dc^{(s)}
\to
\dc_{G'}G_{x, s}\backslash\dc,
$$
so
\begin{equation}
\label{eq:dc-card}
\Bigindx{\dc^{(s)}}{\dc_{G'}^{(s)}\CC G{0{+}}(\gamma)_{x, s}}
= \Bigindx\dc{\dc_{G'}G_{x, s}}.
\end{equation}
By \eqref{eq:dc+-card-1}, since
$G_{x, s+}
\subseteq
\odc{\gamma; x, r{+}}_{G'}(\CC G{0{+}}(\gamma), G)_{x, (s, s{+})}$,
we have that
$$
\odc{\gamma; x, r{+}}_{G'}(\CC G{0{+}}(\gamma), G)_{x, (s, s{+})}
	\backslash\odc{\gamma; x, r{+}}
$$
is naturally in bijection with
$$
\Bigl(\odc{\gamma; x, r{+}}_{G'}(\CC G{0{+}}(\gamma), G)_{x, (s, s{+})}
	\cap \odc{\gamma; x, r{+}}^{(s)}\Bigr)
\backslash\odc{\gamma; x, r{+}}^{(s)},
$$
which, by \eqref{eq:dc+-intersect'}, is just
$$
\odc{\gamma; x, r{+}}_{G'}^{(s)}\CC G{0{+}}(\gamma)_{x, s}
	\backslash\odc{\gamma; x, r{+}}^{(s)}.
$$
By \eqref{eq:dc+-intersect}, this latter set is naturally in
bijection with
$$
\dc_{G'}^{(s)}\CC G{0{+}}(\gamma)_{x, s}
	\backslash\dc_{G'}^{(s)}\odc{\gamma; x, r{+}}^{(s)}.
$$
Thus,
\begin{multline}
\label{eq:dc+-card-2}
\Bigindx{\odc{\gamma; x, r}_{G'}^{(s)}\odc{\gamma; x, r{+}}^{(s)}}
	{\odc{\gamma; x, r}_{G'}^{(s)}\CC G{0{+}}(\gamma)_{x, s}}
\\
= \Bigindx{\odc{\gamma; x, r{+}}}
	{\odc{\gamma; x, r{+}}_{G'}(\CC G{0{+}}(\gamma), G)_{x, (s, s{+})}}.
\end{multline}
By \eqref{eq:other-intersect}, we have an injection
$$
(\CC{G'}{0{+}}(\gamma), \CC G{0{+}}(\gamma))
	_{x, (r, s):(r, s{+})}
\to
\odc{\gamma; x, r{+}}_{G'}G_{x, s{+}}
	\backslash
\odc{\gamma; x, r{+}}.
$$
By Proposition \xref{exp-prop:heres-a-gp} of
\cite{adler-spice:good-expansions},
the cokernel of this injection is
$$
\odc{\gamma; x, r{+}}_{G'}(\CC G{0{+}}(\gamma), G)_{x, (s, s{+})}
	\backslash
\odc{\gamma; x, r{+}},
$$
so
\begin{multline}
\label{eq:dc+-card}
\Bigindx{\odc{\gamma; x, r{+}}}
	{\odc{\gamma; x, r{+}}_{G'}G_{x, s{+}}}
\\
\qquad\qquad
\begin{aligned}
& = \smcard{(\CC{G'}{0{+}}(\gamma), \CC G{0{+}}(\gamma))
	_{x, (r, s):(r, s{+})}
} \\
& \qquad \times \Bigindx{\odc{\gamma; x, r{+}}}
	{\odc{\gamma; x, r{+}}_{G'}(\CC G{0{+}}(\gamma), G)_{x, (s, s{+})}} \\
& = \smcard{(\CC{G'}{0{+}}(\gamma), \CC G{0{+}}(\gamma))
	_{x, (r, s):(r, s{+})}
} \\
& \qquad \times \Bigindx{\odc{\gamma; x, r}_{G'}^{(s)}\odc{\gamma; x, r{+}}^{(s)}}
	{\odc{\gamma; x, r}_{G'}^{(s)}\CC G{0{+}}(\gamma)_{x, s}}.
\end{aligned}
\end{multline}
(where the last equality follows from \eqref{eq:dc+-card-2}).
Upon plugging \eqref{eq:dc-card} and \eqref{eq:dc+-card}
into \eqref{eq:gauss-sum-card-1}, we obtain the desired
formula for $\smabs{\wtilde{\mf G}}$.
\end{proof}

\begin{pn}
\label{prop:gauss-sum}
\begin{align*}
\mf G(\phi, \gamma)
= {} & (-1)^{\card{\dot\Upsilon_{\textup{symm}}(\phi, \gamma)}}
(-\mf G_\Lambda(\ff))^{f(\dot\Upsilon_{\textup{symm}, \textup{ram}}(\phi, \gamma))} \\
& \quad \times \prod_{\alpha \in \dot\Upsilon_{\textup{symm}, \textup{ram}}(\phi, \gamma)}
	\Bigl[\sgn_{\ff_\alpha}\bigl(
		\tfrac 1 2 e_\alpha
		N_{F_\alpha/F_{\pm\alpha}}(w_\alpha)
		d\alpha^\vee(X^*)(\alpha(\gamma_{< r}) - 1)
	\bigr) \\
& \hphantom{
	\quad \times \prod_{\alpha \in \dot\Upsilon_{\textup{symm}, \textup{ram}}(\phi, \gamma)}
}
	\quad \times \sgn_{F_{\pm\alpha}}(\bG_{\pm\alpha})\Bigr],
\end{align*}
where
\begin{itemize}
\item
$\mf G_\Lambda(\ff)$ is as in Definition
\ref{defn:basic-gauss-sum};
\item
$d\alpha^\vee(X^*) = X^*(d\alpha^\vee(1))$;
\item
$w_\alpha$ is any element
of $F_\alpha$, of valuation
$\bigl(r - \ord(\alpha(\gamma_{< r}) - 1)\bigr)/2$,
whose square lies in $F_{\pm\alpha}$;
\item $\bG_{\pm\alpha}$ is the group generated by the root
subgroups $\bU_\alpha$ and $\bU_{-\alpha}$ of \bG;
and
\item
$\sgn_{F_{\pm\alpha}}(\bG_{\pm\alpha})$ is $+1$ or $-1$
according as $\bG_{\pm\alpha}$ is or is not
$F_{\pm\alpha}$-split, respectively.
\end{itemize}
\end{pn}

We will show in the proof that an element $w_\alpha$ as
in the statement exists.

\begin{proof}
As in the proof of Proposition \ref{prop:gauss-sum-card}, we
may, and do, replace $\gamma$ by $\gamma_{< r}$.
Recall the notation $B$ and $V$ from the proof of
Proposition \ref{prop:gauss-sum-card},
and the elements $Y_i$ chosen before Lemma
\ref{lem:sqrt-B-and-B}.
By the way we defined $\log_\Lambda \sqrt\fB$ and the
isomorphism in \eqref{eq:gauss-V-decomp},
$V_\alpha$ is $B$-orthogonal to $V_\beta$ unless
$-\beta \in \Gamma\dota\alpha$, and
\begin{multline}
\label{eq:whats-gauss-B}
B(X, X')
= \sum_{\mysigma \in \Gamma/\Gamma_\alpha}
	\mysigma\bigl(\tfrac 1 2
		X^*\bigl[[Y_i, X], \mysigma_\alpha X'\bigr]
	\bigr)
\\
\quad\text{for $X, X' \in V_\alpha$
	with $\alpha \in \dot\Upsilon_{\textup{symm}}$,}
\end{multline}
where $i = \ord(\alpha(\gamma) - 1)$.
In particular, the sums on the right-hand side of
\eqref{eq:gauss-V-decomp} are $B$-orthogonal.

Put $\mf G = \mf G(\phi, \gamma)$.
We showed in the proof of Proposition
\ref{prop:gauss-sum-card} that
$\mf G = \mf G(V, B)$, in the notation of Lemma
\ref{lem:well-known-gauss}; so, by that lemma,
we have that
$\mf G = \sgn_\ff(\det B)\mf G_\Lambda(\ff)^{\dim_\ff V}$.
By \eqref{eq:gauss-V-decomp},
\begin{equation}
\tag{$\dag$}
\mf G
= \prod_{\alpha \in \dot\Upsilon_{\textup{symm}}}
	\sgn_\ff(\det B\bigr|_{V_\alpha})
	\mf G_\Lambda(\ff)^{\dim_\ff V_\alpha}
\dotm
\prod_{\alpha \in \ddot\Upsilon^{\textup{symm}}}
	\sgn_\ff(\det B\bigr|_{V_{\pm\alpha}})
	\mf G_\Lambda(\ff)^{\dim_\ff V_{\pm\alpha}}.
\end{equation}
We will use in our calculations below the fact that
$\mf G_\Lambda(\ff)^2 = \sgn_\ff(-1)$.

For $\alpha \in \ddot\Upsilon^{\textup{symm}}$,
the matrix of the restriction of $B$ to $V_{\pm\alpha}$,
with respect to a suitable basis, is of the form
$\left(\begin{smallmatrix}
0        & M \\
\trans M & 0
\end{smallmatrix}\right)$
for some matrix $M$.  Thus the determinant of this
restriction is in the square class of
$(-1)^{\dim_\ff V_{\pm\alpha}/2}$, so
\begin{equation}
\label{eq:gauss-non-symm}
\sgn_\ff(\det B\bigr|_{V_{\pm\alpha}})
\mf G_\Lambda(\ff)^{\dim_\ff V_{\pm\alpha}}
= 1.
\end{equation}

Fix $\alpha \in \dot\Upsilon_{\textup{symm}}$ and
put $i = \ord(\alpha(\gamma) - 1)$.
Since $\alpha \in \Phi(\CC\bG{0{+}}(\gamma), \bT)$,
we have that $i > 0$.
As in the proof of Corollary \ref{cor:Q-and-B}, we may, and
do, assume that $Y_i \in \mf t$.
As in the proof of Proposition \ref{prop:theta-tilde-phi},
we have an isomorphism
$\iota_\alpha : V_\alpha \cong \ff_\alpha$.
Put
$X_\alpha = \iota_\alpha\inv(1)$ and
$c_\alpha
= \ol{ X^*\bigl[[Y_i, X_\alpha], \mysigma_\alpha X_\alpha\bigr] }
\in \ff_{\pm\alpha}$.
By \eqref{eq:whats-gauss-B}, $\iota_\alpha$ identifies the
restriction of $B$ to $V_\alpha$ with the pairing
$$
(t_1, t_2)
\mapsto \tfrac 1 2 e_\alpha\tr_{\ff_\alpha/\ff}(
		c_\alpha\dotm t_1\mysigma_\alpha(t_2)
	)
$$
on $\ff_\alpha$.
The determinant of this pairing is
\begin{equation}
\tag{$*$}
\bigl(\tfrac 1 2 e_\alpha\bigr)^{f_\alpha}N_{\ff_\alpha/\ff}(c_\alpha)\Delta,
\end{equation}
where $\Delta$ is the determinant of
$(t_1, t_2)
\mapsto \tr_{\ff_\alpha/\ff}(t_1\mysigma_\alpha(t_2))$.

If $\alpha \in \dot\Upsilon_{\textup{symm}, \textup{unram}}$,
then $f_\alpha = \dim_\ff V_\alpha$ is even
(so $\bigl(\tfrac 1 2 e_\alpha\bigr)^{f_\alpha}$ is a square)
and $\mysigma_\alpha$ is a $(\dim_\ff V_\alpha/2)$th
power of a generator of $\Gal(\ff_\alpha/\ff)$, so
$\sgn_{\Gal(\ff_\alpha/\ff)}(\mysigma_\alpha)
= (-1)^{\dim_\ff V_\alpha/2}$
(where $\sgn_{\Gal(\ff_\alpha/\ff)}$ is as in Lemma
\ref{lem:gen-trace}).
Further,
we have that
$N_{\ff_\alpha/\ff}(c_\alpha)
= N_{\ff_{\pm\alpha}/\ff}(c_\alpha)^2 \in (\ff\cross)^2$.
By Lemma \ref{lem:gen-trace},
$$
\sgn_\ff(\Delta)
= \bigl(-\sgn_\ff(-1)^{\dim_\ff V_\alpha/2}\bigr)^{f_\alpha + 1}
= -\sgn_\ff(-1)^{\dim_\ff V_\alpha/2},
$$
so ($*$) gives
\begin{multline}
\label{eq:gauss-unram-symm}
\sgn_\ff(\det B\bigr|_{V_\alpha})
\mf G_\Lambda(\ff)^{\dim_\ff V_\alpha}
= \sgn_\ff(\Delta) (\mf G_\Lambda(\ff)^2)^{\dim_\ff V_\alpha / 2} \\
= -\sgn_\ff(-1)^{\dim_\ff V_\alpha / 2} \sgn_\ff(-1)^{\dim_\ff V_\alpha / 2}
= -1.
\end{multline}

If $\alpha \in \dot\Upsilon_{\textup{symm}, \textup{ram}}$,
then $F_\alpha/F_{\pm\alpha}$ is totally ramified.
Since $\mysigma_\alpha\alpha(\gamma) = \alpha(\gamma)\inv$,
we have that
$i = \ord(\alpha(\gamma) - 1) \in \ord(F_\alpha\cross)
	\smallsetminus \ord(F_{\pm\alpha}\cross)$.
Similarly, since
$\mysigma_\alpha d\alpha^\vee(X^*) = -d\alpha^\vee(X^*)$,
we have that
$-r \in \ord(F_\alpha\cross)
	\smallsetminus \ord(F_{\pm\alpha}\cross)$.
Since $\ord(F_\alpha\cross)/\ord(F_{\pm\alpha}\cross) \cong \Z/2\Z$,
we have that $r - i \in \ord(F_{\pm\alpha}\cross)$,
so $(r - i)/2 \in \ord(F_\alpha\cross)$.
Let $\varpi_\alpha$ be a uniformizer of $F_\alpha$ that is
negated by $\mysigma_\alpha$, and let
$w_\alpha$ be a power of $\varpi_\alpha$ that has valuation
$(r - i)/2$.  In particular, $w_\alpha^2 \in F_{\pm\alpha}$.

Put
$H_\alpha = d\alpha^\vee(1) \in \Lie(\bG_{\pm\alpha})(F_\alpha)$,
so that $d\alpha^\vee(X^*) = X^*(H_\alpha)$.
Then $\Lie(\bG_{\pm\alpha})$ is the sum of the
$\alpha$-weight space $\Lie(\bG)_\alpha$, the
$(-\alpha)$-weight space $\Lie(\bG)_{-\alpha}$, and the
Cartan subgroup $\pmb{\mf t}^\alpha$ spanned by $H_\alpha$.
Since
$[X_\alpha, \mysigma_\alpha X_\alpha]$
and
$H_\alpha$ both belong to the $1$-dimensional
$F_\alpha$-space $\pmb{\mf t}^\alpha(F_\alpha)$
and $H_\alpha \ne 0$,
we have that there is a constant $t_\alpha \in F_\alpha$
so that
$[X_\alpha, \mysigma_\alpha X_\alpha] = t_\alpha H_\alpha$.
Since $H_\alpha$ and $[X_\alpha, \mysigma_\alpha X_\alpha]$
are both negated by $\mysigma_\alpha$, we have that
$t_\alpha \in F_{\pm\alpha}$.
Then
\begin{equation*}
c_\alpha
= \ol{ X^*\bigl[[Y_i, X_\alpha], \mysigma_\alpha X_\alpha\bigr] }
= \ol{ X^*(d\alpha(Y_i)t_\alpha H_\alpha) }
= \ol{ t_\alpha d\alpha^\vee(X^*)d\alpha(Y_i) }.
\end{equation*}
We claim that $\ord(t_\alpha) = r - i$.
Indeed, since $\depth_x(H_\alpha) = 0$,
we have that
$\ord(t_\alpha)
= \depth_x([X_\alpha, \mysigma_\alpha X_\alpha])
\ge r - i$.
Suppose that we had $\ord(t_\alpha) > r - i$.
Since $\ord(d\alpha^\vee(X^*)) \ge -r$
and $\ord(d\alpha(Y_i)) \ge i$,
this would mean that
$c_\alpha$ was the projection to $\ff_E$ of an element
of $E_{0{+}}$\,; that is, that $c_\alpha = 0$.
(Both inequalities of valuations in the previous
sentence are actually equalities; but we do not need this.)
Thus $V_\alpha$ would be totally $B$-isotropic --- which, by
orthogonality of the sum in \eqref{eq:gauss-V-decomp},
would be a contradiction of the non-degeneracy of $B$.

We claim that
$t_\alpha N_{F_\alpha/F_{\pm\alpha}}(w_\alpha)\inv$
projects to a
square in $\ff_\alpha\cross$ if and only if $\bG_{\pm\alpha}$
is $F_{\pm\alpha}$-split; i.e., if and only if
$\Lie(\bG_{\pm\alpha})$ is $F_{\pm\alpha}$-isomorphic to
$\mf{sl}_2$.
Once we have shown this, we will have that
\begin{multline}
\tag{$**$}
\sgn_\ff(N_{\ff_\alpha/\ff}(c_\alpha))
= \sgn_{\ff_\alpha}(c_\alpha) \\
= \sgn_{\ff_\alpha}\bigl(
	N_{F_\alpha/F_{\pm\alpha}}(w_\alpha)
	d\alpha^\vee(X^*)d\alpha(Y_i)
\bigr)
\sgn_{F_{\pm\alpha}}(\bG_{\pm\alpha}).
\end{multline}

If the image in $\ff_\alpha$ of
$t_\alpha N_{F_\alpha/F_{\pm\alpha}}(w_\alpha)\inv$
equals $\ol\theta^{\,-2}$ for some
$\ol\theta \in \ff_\alpha\cross$,
then, since $\ff_\alpha = \ff_{\pm\alpha}$ and $p \ne 2$, we
may find an element $\theta \in F_{\pm\alpha}\cross$ such that
$t_\alpha N_{F_\alpha/F_{\pm\alpha}}(w_\alpha)\inv
= \theta^{-2}$.
Then the unique $F_\alpha$-linear map
$\Lie(\bG_{\pm\alpha}) \to \mf{sl}_2$ satisfying
\begin{multline*}
X_\alpha
\mapsto \frac{w_\alpha}{2\theta}
\left(\begin{matrix}
	-1        & \varpi_\alpha\inv \\
	-\varpi_\alpha & 1
\end{matrix}\right),
\quad
\mysigma_\alpha X_\alpha
\mapsto \frac{\mysigma_\alpha w_\alpha}{2\theta}
\left(\begin{matrix}
	-1       & -\varpi_\alpha\inv \\
	\varpi_\alpha & 1
\end{matrix}\right), \\
H_\alpha
\mapsto \left(\begin{matrix}
	0        & \varpi_\alpha^{-1} \\
	\varpi_\alpha & 0
\end{matrix}\right)
\end{multline*}
is an $F_{\pm\alpha}$-isomorphism of Lie algebras.

Suppose, on the other hand, that
$\iota : \Lie(\bG_{\pm\alpha}) \to \mf{sl}_2$ is an
$F_{\pm\alpha}$-isomorphism of Lie algebras.
Then the co-character lattice of
$\iota(\pmb{\mf t}^\alpha)$ contains a
simultaneous $(-1)$-eigenvector for every element of
$\mysigma_\alpha\dotm\Gal(F\sep/F_\alpha)$,
so the Cartan $F_{\pm\alpha}$-subalgebra
$\iota(\pmb{\mf t}^\alpha)$ is
$\GL_2(F_{\pm\alpha})$-conjugate to
the Cartan $F_{\pm\alpha}$-subalgebra $\pmb{\mf t}'$ spanned by 
$\left(\begin{smallmatrix}
0        & \varpi_\alpha\inv \\
\varpi_\alpha & 0
\end{smallmatrix}\right)$.
Replace $\iota$ by its composition
with the indicated conjugation.
Note that $d\alpha \circ \iota\inv$ is a weight for the
adjoint action of $\pmb{\mf t}'$ on $\mf{sl}_2$, hence is of the
form $\pm d\alpha'$,
where $d\alpha'$ is the functional on $\pmb{\mf t}'$ sending
$\left(\begin{smallmatrix}
0        & \varpi_\alpha\inv \\
\varpi_\alpha & 0
\end{smallmatrix}\right)$ to $2$.
After further conjugating by
$\left(\begin{smallmatrix}
0  & 1 \\
-1 & 0
\end{smallmatrix}\right)$ if necessary, we may, and do,
assume that
$d\alpha \circ \iota\inv = d\alpha'$.
Then there is a
constant $t'_\alpha \in F_\alpha$ such that
$$
\iota(X_\alpha)
= t'_\alpha\left(\begin{matrix}
	-1        & \varpi_\alpha\inv \\
	-\varpi_\alpha & 1
\end{matrix}\right),
$$
hence
$$
\iota(\mysigma_\alpha X_\alpha)
= \mysigma_\alpha(\iota(X_\alpha))
= \mysigma_\alpha(t'_\alpha)\left(\begin{matrix}
	-1       & -\varpi_\alpha\inv \\
	\varpi_\alpha & 1
\end{matrix}\right).
$$
Recall that $\iota(H_\alpha) \in \pmb{\mf t}'(F_\alpha)$
is a scalar multiple of
$\left(\begin{smallmatrix}
0             & \varpi_\alpha\inv \\
\varpi_\alpha & 0
\end{smallmatrix}\right)$.
Since
$$
d\alpha'(\iota(H_\alpha)) = d\alpha(H_\alpha) = 2
= d\alpha' 
	\left(\begin{matrix}
	0             & \varpi_\alpha\inv \\
	\varpi_\alpha & 0
	\end{matrix}\right),
$$
in fact
$\iota(H_\alpha) =
	\left(\begin{smallmatrix}
	0             & \varpi_\alpha\inv \\
	\varpi_\alpha & 0
	\end{smallmatrix}\right)$.
Thus,
$$
\iota(t_\alpha H_\alpha)
= \iota([X_\alpha, \mysigma_\alpha X_\alpha])
= [\iota(X_\alpha), \iota(\mysigma_\alpha X_\alpha)]
= N_{F_\alpha/F_{\pm\alpha}}(2t'_\alpha)\iota(H_\alpha).
$$
That is, $t_\alpha = N_{F_\alpha/F_{\pm\alpha}}(2t'_\alpha)$.
Since $\ord(t_\alpha) = r - i$, we have $\ord(t'_\alpha) = (r-i)/2$.
Then
$$
t_\alpha N_{F_\alpha/F_{\pm\alpha}}(w_\alpha)\inv
= N_{F_\alpha/F_{\pm\alpha}}(2t'_\alpha w_\alpha\inv)
\equiv (2t'_\alpha w_\alpha\inv)^2
\pmod{(F_\alpha\cross)_{0{+}}},
$$
i.e., $t_\alpha N_{F_\alpha/F_{\pm\alpha}}(w_\alpha)\inv$
projects to the square of $\ol{2 t'_\alpha w_\alpha\inv}$.

By Lemma \ref{lem:gen-trace},
$\sgn_\ff(\Delta) = (-1)^{f_\alpha + 1}$.
Now ($*$) and ($**$) give
\begin{multline}
\label{eq:gauss-ram-symm}
\sgn_\ff(\det B\bigr|_{V_\alpha})\mf G_\Lambda(\ff)^{\dim_\ff V_\alpha} \\
\begin{aligned}
& = (-1)^{f_\alpha + 1}
\bigl(\sgn_\ff(\tfrac 1 2 e_\alpha)\mf G_\Lambda(\ff)\bigr)^{f_\alpha} \\
& \qquad\times\sgn_{\ff_\alpha}(
	N_{F_\alpha/F_{\pm\alpha}}(w_\alpha)
	d\alpha^\vee(X^*)d\alpha(Y_i)
)
\sgn_{F_{\pm\alpha}}(\bG_{\pm\alpha})
\end{aligned}
\\
\begin{aligned}
& = -(-\mf G_\Lambda(\ff))^{f_\alpha} \\
& \qquad\times\sgn_{\ff_\alpha}(
	\tfrac 1 2 e_\alpha
	N_{F_\alpha/F_{\pm\alpha}}(w_\alpha)
	d\alpha^\vee(X^*)d\alpha(Y_i)
)
\sgn_{F_{\pm\alpha}}(\bG_{\pm\alpha}).
\end{aligned}
\end{multline}
We have used that
$\sgn_{\ff_\alpha}(n) = \sgn_\ff(n)^{f_\alpha}$
for $n \in \ff$.
(Note that, if $f_\alpha$ is even, then
$\tfrac1 2 e_\alpha \in \ff\cross \subseteq (\ff_\alpha\cross)^2$,
and
$\mf G_\Lambda(\ff)^{f_\alpha} = \sgn_\ff(-1)^{f_\alpha/2}$.)

Upon combining ($\dag$) with
\eqref{eq:gauss-non-symm}--\eqref{eq:gauss-ram-symm},
and the facts that
\begin{itemize}
\item
$d\alpha(Y_i)
\in (\alpha(\gamma_i) - 1) + E_{i{+}}$
(by Lemma \ref{lem:mock-exp-root-value}),
\item
$\alpha(\gamma_i)
\in \alpha((\gamma_{< r})_{\ge i}) + E_{i{+}}$\,,
and
\item
$\alpha((\gamma_{< r})_{\ge i}) = \alpha(\gamma_{< r})$,
\end{itemize}
we obtain the desired formula.
\end{proof}

\subsection{A formula for $\theta_\sigma$ on $G'$}
\label{ssec:step1-formula}

The following easy technical result on integration is probably
well known, but we could not find a reference.

\begin{lm}
\label{lem:silly-integration}
Suppose that $A$ is a locally compact, Hausdorff
topological group, and $B$ and $C$ are closed
subgroups of $A$ such that
\begin{itemize}
\item $B\backslash A$ carries a quotient measure;
\item the image of $C$ under the projection
$A \to B\backslash A$ is open; and
\item $B \cap C$ is compact.
\end{itemize}
Then, for any right Haar measures $da$ and $db$,
on $A$ and $B$, respectively,
there is a right Haar measure $dc$ on $C$ such
that
$$
\int_{B C} f(a)da
= \int_C \int_B f(b c)db\,dc
$$
for all continuous, compactly supported functions $f$ on $A$.
\end{lm}

\begin{proof}
The choice of $da$ and $db$ fixes a choice of
quotient measure $d\dot a$ on $B\backslash A$.
Since $(B \cap C)\backslash C$ embeds naturally as
an open subset of $B\backslash A$, $d\dot a$
induces a quotient measure $d\dot c$ on $(B \cap C)\backslash C$.
Let $db'$ be the Haar measure on $B \cap C$ such that
$\meas_{db'}(B \cap C) = 1$.
The choice of $d\dot c$ and
$db'$ fixes a choice of measure $dc$ on $C$.  Then
\begin{multline*}
\int_{B C} f(a)da = \int_A [B C](a)f(a)da \\
= \int_{B\backslash A} \int_B [B C](b a)f(b a)db\,d\dot a \\
= \int_{B\backslash A}  [B C](a)\int_B f(b a)db\,d\dot a.
\end{multline*}
Since the support of the outer integral is
contained in (the image of) $(B \cap C)\backslash
C$, we have by our choice of $d\dot c$ that
\begin{multline*}
\int_{B C} f(a)da = \int_{(B \cap C)\backslash C} \int_B f(b c)db\,d\dot c \\
= \int_{(B \cap C)\backslash C}
	\int_{B \cap C} \int_B f(b b' c)db\,db'\,d\dot c \\
= \int_C \int_B f(b c)db\,dc.
\text{\qedhere}
\end{multline*}
\end{proof}

Recall that $\gamma$ is semisimple and $x \in \BB_r(\gamma)$.

\begin{pn}
\label{prop:step1-formula1}
$$
\theta_\sigma(\gamma)
= \sum_{g \in \dc_{G'}^{(s)}\CC G{0{+}}(\gamma)_{x, s}\backslash\dc^{(s)}}
      \theta_{\tilde\rho}(\lsup g\gamma).
$$
\end{pn}

\begin{proof}
By the Frobenius formula,
$$
\theta_\sigma(\gamma)
= \sum_{g \in \stab_{G'}(\ox)G_{x, s}\backslash\stab_{G'}(\ox)G_{x, 0+}}
	\dot\theta_{\tilde\rho}(\lsup g\gamma).
$$
Since
$\stab_{G'}(\ox)G_{x, s} \cap G_{x, 0+}
= G'_{x, 0+}G_{x, s}$ ---
which, by
Proposition \xref{exp-prop:heres-a-gp} of \cite{adler-spice:good-expansions}, is
$(G', G)_{x, (0{+}, s)}$
--- the indexing set for the sum is naturally in bijection with
$(G', G)_{x, (0{+}, s)}\backslash G_{x, 0+}$\,.
Thus
\begin{equation}
\tag{$*$}
\theta_\sigma(\gamma) =
\int_{G_{x, 0+}}
	\dot\theta_{\tilde\rho}(\lsup g\gamma)
dg,
\end{equation}
where $dg$ is the Haar measure on $G_{x, 0+}$
normalized so that
$(G', G)_{x, (0{+}, s)}$ has measure $1$.

Put
$S_\infty = G'_{x, 0+}\dc
= \set{g \in G_{x, 0{+}}}{\jperp(g) = \infty}$,
and, for $i_0, j_0 \in \R$,
put 
$$
S_{i_0j_0} = \set{g \in G_{x, 0+}}
	{\iperp(g) = i_0, \jperp(g) = j_0}.
$$
Note that the sets $S_{i_0j_0}$, together with $S_\infty$,
form a partition of $G_{x,0+}$\,.
By Remark \ref{rem:local-const}, they are open.
We will show that the portion of ($*$) taken
over each $S_{i_0j_0}$ vanishes, so that the integral may be
taken instead over $S_\infty$.

Fix $i_0, j_0 \in \R$, and put
\begin{itemize}
\item $t_0 = i_0 + j_0$,
\item $H = \CC G{i_0}(\gamma)$,
\item $H' = \CC{G'}{i_0}(\gamma)$,
\item $B = (G', G)_{x, (0{+}, j_0{+})}$, and
\item
$C = \set{g \in (H, G)_{x, (0{+}, t_0{+})}}
         {[\gamma\inv, g] \in (H', H)_{x, (t_0{+}, t_0)}}$.
\end{itemize}
We claim that $S_{i_0j_0} \subseteq B C$.

This is obvious if $S_{i_0j_0} = \emptyset$, so assume that
$S_{i_0j_0} \ne \emptyset$.
Fix $s \in S_{i_0j_0}$\,, so $\tperp(s) = t_0$ (by Corollary
\ref{cor:t=i+j}).
By Remark \ref{rem:local-const} and Corollary \ref{cor:t-small},
$i_0 < r - 2j_0$
and $t_0 < r - j_0$.
By Lemma \ref{lem:S-in-BC},
there is
$$
b_1 \in (G', G)_{x, (0{+}, j_0{+})} = B
$$
such that
$[\gamma\inv, b_1\inv s] \in (H', H)_{x, (t_0{+}, t_0)}$.
By Lemma \xref{exp-lem:bracket} of \cite{adler-spice:good-expansions},
$b_1\inv s \in [\gamma; x, t_0]$,
where $[\gamma; x, t_0]$ is as in Definition
\xref{exp-defn:fancy-centralizer-no-underline} of \loccit
Put $s_1 = b_1\inv s$.  By Remark \ref{rem:local-const},
$\iperp(s_1) = i_0$, $\jperp(s_1) = j_0$, and $\tperp(s_1) = t_0$.
Thus
$$
s_1 \in G'_{x, 0+}\dc(H, G)_{x, (j_0, j_0{+})}.
$$
By Proposition \xref{exp-prop:heres-a-gp}
and Remark \xref{exp-rem:bracket-facts-decomp} of
\cite{adler-spice:good-expansions},
\begin{gather*}
\dc(H, G)_{x, (j_0, j_0{+})}
= \dc^{(j_0)}(H, G)_{x, (j_0, j_0{+})}; \\
(H, G)_{x, (j_0, j_0{+})}
= (H', G)_{x, (j_0, j_0{+})}(H', H)_{x, (j_0{+}, j_0)};
\end{gather*}
and, since the commutator of
$G_{x, 0{+}}$ with
$(H', G)_{x, (j_0, j_0{+})} \subseteq G_{x, j_0}$
lies in
$G_{x, j_0{+}} \subseteq (H', G)_{x, (j_0, j_0{+})}$,
we have that
$\dc^{(j_0)} \subseteq G_{x, 0{+}}$
normalizes $(H', G)_{x, (j_0, j_0{+})}$.
Thus we may write $s_1 = k'k_-k$, with
$k' \in G'_{x, 0+}(H', G)_{x, (j_0, j_0{+})}$,
$k_- \in \dc^{(j_0)}$, and
$k \in (H', H)_{x, (j_0{+}, j_0)}$.
By Proposition \xref{exp-prop:heres-a-gp} of \loccit,
$k' \in (G', G)_{x, (0{+}, j_0{+})} = B$.
By Remark \xref{exp-rem:bracket-facts-containment} of
\loccit,
\begin{multline}
\label{eq:bracket-containments}
\dc^{(j_0)} \subseteq
[\gamma; x, r - j_0]^{(j_0)}  \\
\subseteq
[\gamma; x, r - j_0] \cap \CC G{r - 2j_0}(\gamma)_{x, 0+}
\subseteq [\gamma; x, t_0{+}] \cap H_{x, 0+}\,.
\end{multline}
Also,
$(H', H)_{x, (j_0{+}, j_0)} \subseteq H_{x, j_0}
\subseteq [\gamma; x, t_0]$.
Since $s_1 \in [\gamma; x, t_0]$,
this means that
$k' \in [\gamma; x, t_0] \cap (G', G)_{x, (0{+}, j_0{+})}$.
By Lemma \xref{exp-lem:more-vGvr-facts} and
Proposition \xref{exp-prop:heres-a-gp} of \loccit,
$$
[\gamma; x, t_0] \cap (G', G)_{x, (0{+}, j_0{+})}
= [\gamma; x, t_0]_{G'}^{(j_0{+})}
	([\gamma; x, t_0] \cap G_{x, j_0{+}}).
$$
Write $k' = k''k'_+$, with
$k'' \in [\gamma; x, t_0]_{G'}^{(j_0{+})}$
and $k'_+ \in [\gamma; x, t_0] \cap G_{x, j_0+}$\,.
By Remark \xref{exp-rem:bracket-facts-containment} of
\loccit,
$[\gamma; x, t_0]_{G'}^{(j_0{+})} \subseteq H'_{x, 0+}$\,, so
$$
[\gamma\inv, k'']
\in H'_{x, 0{+}} \cap G'_{x, t_0}
= H'_{x, t_0}
\subseteq (H', H, G)_{x, (t_0, t_0{+}, t_0)}.
$$
By Lemma \xref{exp-lem:perp-commute} of \loccit,
$$
[\gamma\inv, k'_+]
\in (H, G)_{x, (t_0{+}, t_0)}
\subseteq (H', H, G)_{x, (t_0, t_0{+}, t_0)}.
$$
Since the commutator of $G_{x, 0{+}}$ with
$(H', H, G)_{x, (t_0, t_0{+}, t_0)} \subseteq G_{x, t_0}$
lies in
$G_{x, t_0{+}} \subseteq (H', H, G)_{x, (t_0, t_0{+}, t_0)}$,
in particular $k'' \in G_{x, 0{+}}$ normalizes
$(H', H, G)_{x, (t_0, t_0{+}, t_0)}$.  Thus
$$
[\gamma\inv, k'] \in (H', H, G)_{x, (t_0, t_0{+}, t_0)}.
$$

Now, using \eqref{eq:bracket-containments} and
imitating the above argument that
$[\gamma\inv, k''] \in H'_{x,t_0}$\,, we see that
$[\gamma\inv, k_-] \in H_{x, t_0{+}}
\subseteq (H', H)_{x, (t_0{+}, t_0)}$.
Also, by Lemma \xref{exp-lem:shallow-comm}
(or Corollary \xref{exp-cor:stab-norm}, if $i_0 = 0$)
of \cite{adler-spice:good-expansions},
$$
[\gamma\inv, k] = [\gamma_{\ge i_0}\inv, k]
\in (H', H)_{x, ((i_0 + j_0){+}, i_0 + j_0)}
= (H', H)_{x, (t_0{+}, t_0)}.
$$
Thus, since $k_- \in H_{x,0{+}}$ normalizes
$(H', H)_{x, (t_0{+}, t_0)}$,
$$
[\gamma\inv, k_-k] \in (H', H)_{x, (t_0{+}, t_0)}.
$$
Thus $s = b c$, where $b := b_1k' \in B$ and $c := k_-k \in C$.

Now we claim that
$$
\int_{S_{i_0j_0}} \dot\theta_{\tilde\rho}(\lsup g\gamma)dg = 0.
$$
Once again, this is obvious if $S_{i_0j_0} = \emptyset$, so
suppose $S_{i_0j_0} \ne \emptyset$.
Note that $(H', H)_{x, ((r - t_0){+}, r - t_0)}
\subseteq B$.
By Lemma \ref{lem:silly-integration}, applied to
the function sending $g \in G$ to
$[S_{i_0j_0}](g)\dot\theta_{\tilde\rho}(\lsup g\gamma)$,
we have that
\begin{multline}
\label{eq:Sij-int}
\int_{S_{i_0j_0}} \dot\theta_{\tilde\rho}(\lsup g\gamma)dg
= (\text{const})\int_B \int_C
	[S_{i_0j_0}](b c)\dot\theta_{\tilde\rho}(\lsup{b c}\gamma)dc\,db \\
= (\text{const})\int_B \int_C
	\int_{(H', H)_{x, ((r - t_0){+}, r - t_0)}}
	[S_{i_0j_0}](b h c)\dot\theta_{\tilde\rho}(\lsup{b h c}\gamma)
		dh\,dc\,db.
\end{multline}
Suppose $b \in B$, $h \in (H', H)_{x, ((r - t_0){+}, r - t_0)}$,
and $c \in C$.
By Remark \ref{rem:local-const},
$[S_{i_0j_0}](b h c) = [S_{i_0j_0}](c)$.
We have that
\begin{equation}
\label{eq:bc-conj1}
\lsup{h c}\gamma
= [h, \gamma]
  \cdot\lsup c\gamma
  \cdot\bigl[[c, \gamma\inv], h\bigr].
\end{equation}
By Lemma \xref{exp-lem:shallow-comm}
(or Corollary \xref{exp-cor:stab-norm}, if $i_0 = 0$)
of \cite{adler-spice:good-expansions},
\begin{equation}
\label{eq:bc-conj2}
[h, \gamma] = [h, \gamma_{\ge i_0}] 
\in (H', H)_{x, (r - t_0 + i_0){+}, r - t_0 + i_0)}
\subseteq (G', G)_{x, (r{+}, r - j_0)}.
\end{equation}
By the definition of the group $C$, we have
$[c, \gamma\inv] \in G_{x, t_0}$\,.
Thus, since $h \in G_{x, r - t_0}$\,, we have
\begin{equation}
\label{eq:bc-conj3}
\bigl[[c, \gamma\inv], h\bigr] \in G_{x, r}\,.
\end{equation}
Combining \eqref{eq:bc-conj1}--\eqref{eq:bc-conj3}
gives
\begin{multline*}
\lsup{b h c}\gamma
\in \lsup b(G', G)_{x, (r{+}, r - j_0)}
     \cdot\lsup{b c}\gamma
     \cdot\lsup b\bigl[[c, \gamma\inv], h\bigr] \\
\subseteq (G', G)_{x, (r{+}, r - j_0)}
          \cdot\lsup{b c}\gamma
          \cdot\bigl[[c, \gamma\inv], h\bigr]
          G_{x, r+}
\end{multline*}
for $b \in B = (G', G)_{x, (0{+}, j_0{+})}$.
The containment on the second line follows from
the fact that, by Corollary \xref{exp-cor:master-comm} of
\cite{adler-spice:good-expansions},
$(G', G)_{x, (r{+}, r - j_0)}$ is
normalized by $B = (G', G)_{x, (0{+}, j_0)}$.
Now, by Lemma \ref{lem:trho-isotyp}
and the fact that
$(G', G)_{x, (r{+}, r - j_0)} \subseteq \ker \hat\phi$,
$$
\dot\theta_{\tilde\rho}(\lsup{b h c}\gamma) =
\hat\phi\bigl(\bigl[[c, \gamma\inv], h\bigr]\bigr)
\dot\theta_{\tilde\rho}(\lsup{b c}\gamma)
= \bigl[[\gamma\inv, c], \hat\phi\bigr](h)
	\dot\theta_{\tilde\rho}(\lsup{bc}\gamma),
$$
where $\bigl[[\gamma\inv, c], \hat\phi]$ is the character of
$G_{x, r - t_0}$ given by
$g \mapsto \hat\phi\bigl(\bigl[[c, \gamma\inv], g\bigr]\bigr)$.

In particular, the inner integral in \eqref{eq:Sij-int} is
$0$ unless
$c \in C \cap S_{i_0j_0}$ and
$\bigl[[\gamma\inv, c], \hat\phi]$
is trivial on $(H', H)_{x, ((r - t_0){+}, r - t_0)}$.

Fix $c \in C \cap S_{i_0j_0}$ for which the indicated
character is trivial.
If $g \in G'_{x, r - t_0}$\,, then,
by two applications of Hypothesis \ref{hyp:mock-exp-ad}
$$
\bigl[[\gamma\inv, c], \hat\phi\bigr](g)
= \hat\phi\bigl(\bigl[[c, \gamma\inv], g\bigr]\bigr)
= 1.
$$

If $g \in (H, G)_{x, ((r - t_0){+}, r - t_0)}$, then,
since
$[c, \gamma\inv] \in (H', H)_{x, (t_0{+}, t_0)} \subseteq
	H_{x, t_0}$\,,
we have by Lemma \xref{exp-lem:shallow-comm} of \cite{adler-spice:good-expansions} that
$\bigl[[c, \gamma\inv], g\bigr] \in (H, G)_{x, (r{+}, r)}$.
By Lemma \ref{lem:phi-trivial} and Remark \ref{rem:phi-trivial},
$(H, G)_{x, (r{+}, r)} \subseteq \ker \hat\phi$,
so $g \in \ker \bigl[[\gamma\inv, c], \hat\phi\bigr]$.

We have seen that $\bigl[[\gamma\inv, c], \hat\phi\bigr]$
is trivial on the group generated by
$(H', H)_{x, ((r - t_0){+}, r - t_0)}$,
$G'_{x, r - t_0}$\,,
and $(H, G)_{x, ((r - t_0){+}, r - t_0)}$,
which, by Proposition \xref{exp-prop:heres-a-gp} of \cite{adler-spice:good-expansions}, is all of
$G_{x, r - t_0}$\,.
By Lemma \ref{lem:phi-commute},
this means that
$[\gamma\inv, c] \in (G', G)_{x, (0{+}, t_0{+})}$.
Since $c \in C$, also
$[\gamma\inv, c] \in (G', G)_{x, (t_0{+}, t_0)}$; so
in fact $[\gamma\inv, c] \in G_{x, t_0+}$\,.  This
contradicts the fact that $t_0 = \tperp(c)$.
Thus the inner integral in \eqref{eq:Sij-int} is
always $0$, so
$\int_{S_{i_0j_0}} \dot\theta_{\tilde\rho}(\lsup g\gamma)dg = 0$,
as desired.

By Remarks \xref{exp-rem:bracket-facts-containment}
and \xref{exp-rem:bracket-facts-decomp} of \cite{adler-spice:good-expansions},
$G_{x, s}\dc^{(s)} = \dc$;
and, by Proposition \xref{exp-prop:heres-a-gp} of \loccit,
$G'_{x, 0{+}}G_{x, s} = (G', G)_{x, (0{+}, s)}$;
so $G'_{x,0{+}}\dc = (G', G)_{x, (0{+}, s)}\dc^{(s)}$.
By Lemma \ref{lem:silly-integration},
there is a measure $dh$ on $\dc^{(s)}$ such that
$$
\int_{G'_{x, 0{+}}\dc} f(g)dg
= \int_{\dc^{(s)}} \int_{(G', G)_{x, (0{+}, s)}} f(g h)dg\,dh
$$
for all continuous functions $f$ on
$G_{x, 0{+}}$\,.
By definition,
$\meas_{dg}((G', G)_{x, (0{+}, s)}) = 1$,
so
\begin{multline*}
\meas_{dh}(\dc^{(s)})
= \meas_{dg}(G'_{x, 0{+}}\dc) \\
= \Bigindx{G'_{x, 0{+}}\dc}{(G', G)_{x, (0{+}, s)}} \\
= \Bigindx{\dc^{(s)}}{\dc^{(s)} \cap (G', G)_{x, (0{+}, s)}}.
\end{multline*}
By Lemma \xref{exp-lem:more-vGvr-facts} of \cite{adler-spice:good-expansions},
$\dc^{(s)} \cap (G', G)_{x, (0{+}, s)}
= \dc_{G'}^{(s)}\CC G{0{+}}(\gamma)_{x, s}$\,.
Thus
\begin{multline*}
\theta_\sigma(\gamma)
= \int_{G'_{x, 0+}\dc}
      \dot\theta_{\tilde\rho}(\lsup g\gamma)dg 
= \int_{\dc^{(s)}} \int_{(G', G)_{x, (0{+}, s)}}
      \dot\theta_{\tilde\rho}(\lsup {g h}\gamma)dg\,dh \\
= \int_{\dc^{(s)}}
      \dot\theta_{\tilde\rho}(\lsup h\gamma) dh
= \sum_{g \in \dc_{G'}^{(s)}\CC G{0{+}}(\gamma)_{x, s}\backslash\dc^{(s)}}
      \dot\theta_{\tilde\rho}(\lsup g\gamma).
\text{\qedhere}
\end{multline*}
\end{proof}

\begin{pn}
\label{prop:induction1}
\begin{align*}
\theta_\sigma(\gamma)
& = \Bigindx{\odc{\gamma_{< r}; x, r}}
	{\odc{\gamma_{< r}; x, r}_{G'}G_{x, s}}^{1/2} \\
& \qquad\times\Bigindx{\odc{\gamma_{< r}; x, r{+}}}
	{\odc{\gamma_{< r}; x, r{+}}_{G'}G_{x, s{+}}}^{1/2} \\
& \qquad\times\mf G(\phi, \gamma_{< r})
\varepsilon(\phi, \gamma_{< r})\theta_{\tau_{d - 1}}(\gamma),
\end{align*}
where
$\mf G(\phi, \gamma_{< r})$ is as in Proposition \ref{prop:gauss-sum}
and
$\varepsilon(\phi, \gamma_{< r})$ is as in
Proposition \ref{prop:theta-tilde-phi}.
\end{pn}

\begin{proof}
By Lemma \ref{lem:trho-isotyp}, applied to
$\sigma = \sigma_d$ and $\tau_{d - 1}$ (using the fact that
$G'_{x, r_{d - 1}} \subseteq G'_{x, r_{d - 2}{+}}$),
it suffices to verify the desired equality when
$\gamma = \gamma_{< r}$.
Then, by Proposition \ref{prop:step1-formula1} and Lemma
\ref{lem:trho-isotyp}, we have
\begin{multline*}
\theta_\sigma(\gamma)
= \sum_{g \in \dc_{G'}^{(s)}\CC G{0{+}}(\gamma)_{x, s}\backslash\dc^{(s)}}
	\dot\theta_{\tilde\rho}(\lsup g\gamma) \\
= \theta_{\tilde\rho}(\gamma)
\sum_{g \in \dc_{G'}^{(s)}\CC G{0{+}}(\gamma)_{x, s}\backslash\dc^{(s)}}
	\hat\phi([\gamma\inv, g])
= \theta_{\tilde\rho}(\gamma)\smabs{\wtilde{\mf G}}
	\mf G
\end{multline*}
(where $\wtilde{\mf G} = \wtilde{\mf G}(\phi, \gamma)$ and
$\mf G = \mf G(\phi, \gamma)$ are the quantities calculated
in \S\ref{ssec:gauss}).
If $\gamma \in \lsup{\stab_{G'}(\ox)}K^{d - 1}$ ---
say, $\gamma = \lsup gk$,
with $g \in \stab_{G'}(\ox)$ and $k \in K^{d - 1}$
--- then Lemma \ref{lem:trho-char}
and Proposition \ref{prop:theta-tilde-phi}
show that
\begin{multline*}
\theta_{\tilde\rho}(\gamma) = \theta_{\tilde\rho}(k)
= \theta_{\tilde\phi}(k \ltimes 1)\theta_{\tau_{d - 1}}(k)
= \theta_{\tilde\phi}(\gamma \ltimes 1)\theta_{\tau_{d - 1}}(\gamma) \\
= \smcard{(\CC{G'}{0{+}}(\gamma), \CC G{0{+}}(\gamma))_{x, (r, s):(r, s{+})}}
	^{1/2}
\varepsilon(\phi, \gamma)\theta_{\tau_{d - 1}}(\gamma).
\end{multline*}
As in the proof of Lemma \ref{lem:trho-isotyp}, we see that
$K \cap \stab_{G'}(\ox) = K^{d - 1}$.
Thus, since $J \subseteq K$, we have
$\lsup{\stab_{G'}(\ox)J}K \cap \stab_{G'}(\ox)
= \lsup{\stab_{G'}(\ox)}(K \cap \stab_{G'}(\ox))
= \lsup{\stab_{G'}(\ox)}K^{d - 1}$.
Since $\gamma \in \stab_{G'}(\ox)$, we have that, if
$\gamma \not\in \lsup{\stab_{G'}(\ox)}K^{d - 1}$ ---
so that $\theta_{\tau_{d - 1}}(\gamma) = 0$
--- then $\gamma \not\in \lsup{\stab_{G'}(\ox)J}K$ ---
so that again
$$
\theta_{\tilde\rho}(\gamma) = 0
= \smcard{(\CC{G'}{0{+}}(\gamma), \CC G{0{+}}(\gamma))
	_{x, (r, s):(r, s{+})}
}^{1/2}
\varepsilon(\phi, \gamma)\theta_{\tau_{d - 1}}(\gamma).
$$

To complete the proof, we note that, by Proposition
\ref{prop:gauss-sum-card},
$$
\smabs{\wtilde{\mf G}}
\dotm
\smcard{(\CC{G'}{0{+}}(\gamma), \CC G{0{+}}(\gamma))
	_{x, (r, s):(r, s{+})}}^{1/2}
$$
equals
\[
\Bigindx\dc{\dc_{G'}G_{x, s}}^{1/2} 
\Bigindx{\odc{\gamma; x, r{+}}}
	{\odc{\gamma; x, r{+}}_{G'}G_{x, s{+}}}^{1/2}.
\qedhere
\]
\end{proof}

\numberwithin{thm}{section}
\numberwithin{equation}{section}
\section{Induction to $G$}
\label{sec:induction2}

In this section, we compute
the character of the
representation $\tau = \tau_d$ of $\stab_G(\ox)$ induced from the
representation $(\sigma, K_\sigma)$ whose character we
computed in \S\ref{sec:induction1}.
If $\bG'/Z(\bG)$ is $F$-anisotropic, then we also compute the
character of the representation $\pi = \pi_d$
of $G$ induced from $(\sigma, K_\sigma)$.
As in \S\ref{sec:induction1}, unless certain tameness and
compactness hypotheses are satisfied, we must place mild
restrictions on the elements that we consider.

Namely,
we fix throughout this section an element $\gamma \in G$,
and assume that
$\gamma$ has a normal $r$-approximation;
but, unless otherwise stated, we do \emph{not} assume that
$\gamma \in G'$ or
$x \in \BB_r(\gamma)$.
By Lemma \xref{exp-lem:simult-approx} of \cite{adler-spice:good-expansions},
under suitable assumptions on \bG,
any bounded-modulo-$Z(G)$ element of $G$
that belongs to a tame $F$-torus
will do.
By \cite{casselman:jacquet} or \cite{deligne:support},
$\Theta_\pi(\gamma) = 0$ unless
$\gamma$ is bounded modulo $Z(G)$,
and the domain of $\tau$ is already bounded modulo $Z(G)$;
so, under these assumptions,
we need only require that $\gamma$ be tame.
(Remember that an element or subgroup of $G$ is said to be bounded
modulo $Z(G)$ if its orbits in $\rBB(\bG, F)$ are bounded in
the sense of metric spaces.)

\begin{lm}
\label{lem:finite-double-coset}
If \bM is a Levi $F$-subgroup of \bG
and $\delta \in G\sss$,
then
$$
M
\backslash\set{g \in G}{\lsup g\delta \in M}
/C_G(\delta)\conn
$$
is finite.
\end{lm}

Note that \bM above need not be an $F$-Levi subgroup (i.e.,
a Levi component of a parabolic $F$-subgroup).

\begin{proof}
Put $\bH = C_\bG(\delta)$
and $\mc C = \set{g \in G}{\lsup g\delta \in M}$.
Since every $(N_G(M), H\conn)$-double coset is a
finite union of $(M, H\conn)$-double cosets,
it suffices to show that
$N_G(M)\backslash\mc C/H\conn$ is finite.

Let \bS be a maximal torus in \bG containing $\delta$.
For $g \in \mc C$, we have that
$Z(\bM^g)\conn \subseteq \bS \subseteq \bH\conn$, so
$\bM^g \cap \bH\conn = C_{\bH\conn}(Z(\bM^g)\conn)$ is a Levi subgroup
(necessarily defined over $F$) of $\bH\conn$.
Consider the $H$-equivariant map $f$ from
$N_G(M)\backslash\mc C$
to the set of Levi $F$-subgroups of $\bH\conn$
that sends $N_G(M)g$
to $\bM^g \cap \bH\conn$.
We claim that $f$ is finite-to-one.
Indeed, for $g_0 \in \mc C$, fix a torus $\bT^{g_0}$
that is maximal in $\bM^{g_0} \cap \bH\conn$, hence in \bG.
Then $N_G(M)g \mapsto \bM^g$
is an injection from the fiber of $f$ over $\bM^{g_0} \cap \bH\conn$
into the set of Levi subgroups of \bG
containing $\bT^{g_0}$, which is finite.

Thus there is a finite-to-one map from
$N_G(M)\backslash\mc C/H\conn$ to the set of
$H\conn$-orbits of Levi $F$-subgroups of $\bH\conn$.
Recall that there are only finitely many $\bH\conn(F\sep)$-orbits of
Levi subgroups of $\bH\conn$.
Thus it suffices to show that every
such orbit contains at most finitely many $H\conn$-orbits of Levi
$F$-subgroups.

Accordingly, fix a Levi subgroup $\bL \subseteq \bH\conn$.
Clearly, it suffices to consider the case where \bL is
$F$-rational.
Then the intersection of the $\bH\conn(F\sep)$-orbit of \bL with the 
set of Levi $F$-subgroups of $\bH\conn$ is
$$
\sett{\bL^h}
{$h \in \bH\conn(F\sep)$
and
$h\sigma(h)\inv \in N_{\bH\conn}(\bL)(F\sep)
	\text{ for }\sigma \in \Gal(F\sep/F)$},
$$
which is naturally in $H\conn$-equivariant bijection with
$(N_{\bH\conn}(\bL)\backslash\bH\conn)(F)$.
Thus, it suffices to show that $(\bL\backslash\bH\conn)(F)/H\conn$
is finite.
Standard Galois cohomology arguments show that this latter
set
is in bijection with the kernel of the natural map
$H^1(F\sep/F, \bL(F\sep)) \to H^1(F\sep/F, \bH\conn(F\sep))$.
Under the assumption that
$F$ has characteristic $0$
and that
\bG is $F\unram$-split and $F$-quasisplit,
\cite{debacker-reeder:depth-zero-sc}*{\S\S2.2--2.3},
describes a bijection of
$H^1(F\unram/F, \bL(F\unram))$ with
the set of torsion points in a certain finite quotient of
the lattice of cocharacters of a certain torus (see
Corollary 2.3.3 of \loccit).
However, it is observed there that the splitness and
quasisplitness assumptions are unnecessary
(although we need to take the torus \bT of \S2.3 of \loccit
to be the centralizer of a
maximal $F\unram$-split torus containing a maximal $F$-split
torus); and it can be
checked that the proof also does not require $\chr F = 0$.
Thus, $H^1(F\unram/F, \bL(F\unram))$ is finite.
Since $H^1(F\sep/F\unram, \bL(F\sep)) = \sset 0$ (as
observed in
\cite{adler-spice:good-expansions}*{\S\xref{exp-sec:buildings}}),
we have by
\cite{serre:galois}*{\S I.5.8(a)} that
$H^1(F\sep/F, \bL(F\sep))$,
hence \emph{a fortiori} the desired kernel,
is also finite.
\end{proof}

For the remainder of this paper, we fix a normal
$r$-approximation to $\gamma$
(hence to all of its conjugates and truncations),
so that $\gamma_{< r}$ is a well defined element.
For the remainder of this section, we put
$\bH = C_\bG(\gamma_{< r})$.
(Note that Proposition \xref{exp-prop:unique-approx} of \cite{adler-spice:good-expansions} guarantees
only that $\bH\conn$, not necessarily \bH itself, is
determined by $\gamma$; but, since we have chosen a specific
normal $r$-approximation, there is no ambiguity.)

We need to prove a result analogous to Lemmata 10.0.5 and 10.0.6
of \cite{debacker-reeder:depth-zero-sc}.  First, we
prove an analogue of Lemma 7.0.9 of \emph{loc.{} cit.}

\begin{lm}
\label{lem:tail-regular}
Suppose that $\gamma$ is regular semisimple in \bG.  Then
$\gamma_{\ge r}$ is regular semisimple in \bH.
\end{lm}

\begin{proof}
By Definition \xref{exp-defn:r-approx} and
Lemma \xref{exp-lem:compare-centralizers} of \cite{adler-spice:good-expansions},
$\gamma \in \CC G r(\gamma) = H\conn$.
Thus, there is a torus containing both $\gamma$ and
$\gamma_{< r}$, hence also $\gamma_{\ge r}$.
In particular, $\gamma_{\ge r} \in H\conn$ is semisimple,
so it suffices to show that
$C_\bH(\gamma_{\ge r})\conn$ is contained in a torus.
We have that
$C_\bH(\gamma_{\ge r})
\subseteq C_\bG(\gamma_{< r}\gamma_{\ge r})
= C_\bG(\gamma)$,
so
$C_\bH(\gamma_{\ge r})\conn
\subseteq C_\bG(\gamma)\conn$.
By regularity of $\gamma$, we have that $C_\bG(\gamma)\conn$ is a torus.
The proof is complete.
\end{proof}

\begin{lm}
\label{lem:char-pi-cpt-supp}
Suppose that $\bG'/Z(\bG)$ is $F$-anisotropic
and $\gamma$ is regular semisimple.
If $\cpt_H$ is a compact open subgroup of $H\conn$,
then
$$
g \mapsto
\int_{\cpt_H} \dot\theta_\sigma(\lsup{g k}\gamma)dk
$$
is compactly supported on $G/Z(G)$.
\end{lm}

\begin{proof}
By Lemma \ref{lem:finite-double-coset} (with $\bM = \bG'$),
since $K_\sigma$ contains $G'$,
the set of $(K_\sigma, H\conn)$-double cosets in $G$ containing an
element $g$ with $\lsup g\gamma_{< r} \in G'$ is finite.
By Corollary \ref{cor:wk-step1-support}, the support of the
function occurring in the statement is contained in the union of such double
cosets.  Thus, it suffices to show that the restriction of
the indicated function to any $(K_\sigma, H\conn)$-double coset
has compact support.

Fix a double coset $K_\sigma g H\conn$ in $G$.
Since $\dot\theta_\sigma$ is invariant under conjugation by
the compact-modulo-$Z(G)$ group $K_\sigma$,
it suffices to show that
$$
h \mapsto
\int_{\cpt_H} \dot\theta_\sigma(\lsup{g h k}\gamma)dk
$$
is compactly supported on $H\conn$, modulo $Z(G)$.
Suppose that $h \in H\conn$ and $k \in \cpt_H \subseteq H\conn$.
Then
$\lsup{g h k}\gamma
= (\lsup g\gamma_{< r})(\lsup{g h k}\gamma_{\ge r})$.
Therefore, by Corollary \ref{cor:trho-isotyp},
\begin{align*}
\dot\theta_\sigma(\lsup{g h k}\gamma)
= {} & [\lsup{G_{x, 0+}}G'](\lsup g\gamma_{< r})
	\dot\theta_\sigma(\lsup g\gamma_{< r})
	{\dotm}
	[\BB(\bH, F)]((g h k)\inv x) \\
  & \quad \times
	[G_{x, r}](\lsup{g h k}\gamma_{\ge r})
	\hat\phi(\lsup{g h k}\gamma_{\ge r}).
\end{align*}
Since $\lsup{h k}\gamma_{\ge r} \in H\conn$, we have that
$\lsup{g h k}\gamma_{\ge r} \in G_{x, r}$ if and only if
$\lsup{h k}\gamma_{\ge r} \in H\conn \cap G_{g\inv x, r}$\,.
Thus it suffices to show that
\begin{equation}
\tag{$*$}
h \mapsto
\int_{\cpt_H}
	[\BB(\bH, F)]((g h k)\inv x)
	\cdot
	[H\conn \cap G_{g\inv x, r}](\lsup{h k}\gamma_{\ge r})
	\hat\phi^g(\lsup{h k}\gamma_{\ge r})dk
\end{equation}
is compactly supported on $H\conn$, modulo $Z(G)$, whenever
$\lsup g\gamma_{< r} \in \lsup{G_{x, 0+}}G'$.
Since ($*$) does not change if we replace $g$ by an element
of $K_\sigma g$, we need only consider the case that
$\lsup g\gamma_{< r} \in G'$.

If $g\inv x \not\in \BB(\bH, F)$, then the function ($*$) vanishes.
Suppose that $g\inv x \in \BB(\bH, F)$ (as well as
$\lsup g\gamma_{< r} \in G'$).
Then, by Remark \xref{exp-rem:approx-facts-in-stab} of \cite{adler-spice:good-expansions},
$\lsup g\gamma_{< r} \in \stab_G(\ox)$.
Let
$\Sigma_H = (\vec\bH, \vec\phi_H, \vec r_H, x_H, \rho_{H, 0}')$
be a cuspidal datum
(see Definition \ref{defn:cusp-dat})
such that
\begin{enumerate}
\item
$\vec\bH = (\bH^0 \subseteq \bH^1)$,
where
$\bH^0
= \bH\conn \cap \lsup{g\inv}\bG'$ and
$\bH^1 = \bH\conn$;
\item
$\vec\phi_H = (\phi_{H, 0}, 1)$,
where $\phi_{H, 0} = \phi^g\bigr|_{H^0}$;
\item
$\vec r_H = (r, r)$;
and
\item
$x_H = g\inv x$.
\end{enumerate}
(Note that $\bH^0$, $\phi_{H, 0}$ and $x_H$ all
depend on $g$ as well as on $H$.)
As in \S\ref{sec:JK},
there are associated to the datum $\Sigma_H$ a
compact-modulo-$Z(H)$ open subgroup $K_{\Sigma_H}$ of $H$
and a representation $\rho_{\Sigma_H}'$ of $K_H$ such that
$\pi_{\Sigma_H} = \Ind_{K_{\Sigma_H}}^H \rho_{\Sigma_H}'$ is
an irreducible supercuspidal representation of $H$.
Put
$K_{\sigma_H} = K_{\Sigma_H}H_{x, 0+}$
and
$\sigma_H
= \Ind_{K_{\Sigma_H}}^{K_{\sigma_H}} \rho_{\Sigma_H}'$.
Now we are in the situation of \S\ref{sec:induction1}
(with $(G, \sigma)$ there replaced by $(H\conn, \sigma_H)$).

By Corollary \ref{cor:trho-isotyp}, for $h \in H\conn$ and
$k \in \cpt_H$, we have
\begin{multline*}
\dot\theta_{\sigma_H}(\lsup{h k}\gamma_{\ge r})
= [\BB(\bH, F)]((h k)\inv x_H)
	\cdot
	[H_{x_H, r}](\lsup{h k}\gamma_{\ge r})
	\hat\phi_{H, 0}(\lsup{h k}\gamma_{\ge r}) \\
= [\BB(\bH, F)]((g h k)\inv x)
	\cdot
	[G_{x, r}](\lsup{g h k}\gamma_{\ge r})
	\hat\phi^g(\lsup{h k}\gamma_{\ge r}).
\end{multline*}
Since $\dot\theta_{\sigma_H}$ is a sum of matrix coefficients of 
the supercuspidal representation $\pi_H$, it is a cusp form
(or `supercusp form', in the language of
\cite{hc:harmonic}*{\S I.3}) on $H\conn$.
In particular, by Lemma 23 of \cite{hc:harmonic}
(the proof of which does not depend on $\chr F$ being $0$)
and our Lemma \ref{lem:tail-regular},
($*$) is compactly supported on $H\conn$, modulo $Z(H\conn)$.
Since $\gamma_{< r} \in G'$,
we have that
$Z(H\conn) = Z(C_G(\gamma_{< r})\conn) \subseteq G'$
is compact modulo $Z(G)$, so
($*$) is also compactly supported on $H\conn$, modulo $Z(G)$.
\end{proof}

The portion of the following result concerning $\Theta_\pi$
is the analogue
of Lemma 10.0.4 of \cite{debacker-reeder:depth-zero-sc}.

\begin{thm}
\label{thm:char-tau|pi-1}
If $x \in \BB_r(\gamma)$, then
\begin{equation}
\label{eq:char-tau}
\theta_\tau(\gamma)
= \phi_d(\gamma)
\sum_g
	\theta_\sigma(\lsup g\gamma_{< r}) \\
	\hat\mu_{X^*}^{\stab_{\lsup g H}(\ox)}(
		\mexp_x\inv(\lsup g\gamma_{\ge r})
	).
\end{equation}
If $\bG'/Z(\bG)$ is $F$-anisotropic and $\gamma \in G$
is regular semisimple, then
\begin{equation}
\label{eq:char-pi}
\Theta_\pi(\gamma)
= \phi_d(\gamma)
\sum_g
	\theta_\sigma(\lsup g\gamma_{< r})
	\hat\mu_{X^*}^{\lsup g H}(
		\mexp_x\inv(\lsup g\gamma_{\ge r})
	).
\end{equation}
The sums run over those double
cosets in $\stab_{G'}(\ox)G_{x, 0+}\backslash\stab_G(\ox)/\stab_H(\ox)$
or $G'G_{x, 0+}\backslash G/H$, respectively,
containing an element $g$ such that
$\lsup g\gamma_{< r} \in G'$
and $x \in \BB_r(\lsup g\gamma)$.

Here,
$\hat\mu_{X^*}^{\stab_{\lsup g H}(\ox)}$ is the function
representing the distribution
\eqref{eq:mu-stab-ox} below,
and both it and $\hat\mu_{X^*}^{\lsup g H}$ are defined with
respect to the Haar measure on $\lsup g H/Z(G)$
normalized so that $\meas(K_\sigma \cap \lsup g H/Z(G)) = 1$.
\end{thm}

If we used a suitable exponential map in place of $\mexp_x$
(one that, among other things, was conjugation invariant and
defined on all the filtration lattices $\fg_{y, r}$ for
$y \in \BB(\bG, F)$), then
the sums in \eqref{eq:char-tau} and \eqref{eq:char-pi}
could be extended over all
double cosets containing an element $g$ such that
$\lsup g\gamma_{< r} \in G'$, since Lemma \ref{lem:orbital-cancel}
shows that the extra summands would vanish.  However, it is
more convenient for our purposes to restrict the sum, so
that we do not have to assume the existence of a suitable
exponential map,
and so that we can apply Proposition \ref{prop:induction1}
(which is subject to the assumptions in force through all of
\S\ref{sec:induction1}, including that
$\gamma_{\ge r} \in G_{x, r}$).

Note that the orbital integrals appearing in
\eqref{eq:char-pi} are taken over the $F$-rational points
of possibly \emph{disconnected} groups $\lsup g H$.
By Lemma \ref{lem:disc-orbital-int}, it is easy to describe
them as sums of orbital integrals over the connected groups
$\lsup gH\conn$ if necessary.

\begin{proof}
Recall that
$K_\sigma = \stab_{G'}(\ox)G_{x, 0{+}}$\,,
so $K_\sigma = G'G_{x, 0{+}}$ if $\bG'/Z(\bG)$ is
$F$-anisotropic.

First, we compute $\theta_\tau(\gamma)$ (in case $x \in \BB_r(\gamma)$).
By the Frobenius formula, we have that
$\theta_\tau(\gamma)
= \phi_d(\gamma)
\sum_{g \in K_\sigma\backslash\stab_G(\ox)}
       \dot\theta_\sigma(\lsup g\gamma)$,
so
\begin{multline}
\label{eq:char-tau-big-supp}
\theta_\tau(\gamma)
= \phi_d(\gamma)
\sum_{g \in  K_\sigma\backslash\stab_G(\ox)/\stab_H(\ox)} \,
	\sum_{g' \in K_\sigma\backslash K_\sigma g\stab_H(\ox)}
	\dot\theta_\sigma(\lsup{g'}\gamma) \\
= \phi_d(\gamma)
\sum_g \, \sum_{
	h' \in (\stab_{\lsup g H}(\ox) \cap K_\sigma)
		\backslash\stab_{\lsup g H}(\ox)
}
	\dot\theta_\sigma(\lsup{h'g}\gamma) \\
= \phi_d(\gamma)
\sum_g \dot\theta_\sigma(\lsup g\gamma_{< r})
	\sum_{h'} \hat\phi(\lsup{h'g}\gamma_{\ge r})
\end{multline}
(where the equality on the last line follows from Lemma
\ref{lem:trho-isotyp}, and the fact that
$(\lsup{h'g}\gamma)_{< r} = \lsup g\gamma_{< r}$
for all $g$ and $h'$ as above).
An easy formal calculation, using
Hypothesis \ref{hyp:strong-mock-exp-equi} and
the fact that
$\meas(K_{\sigma} \cap \stab_{\lsup g H}(\ox)/Z(G))
= \meas(K_\sigma \cap \lsup g H/Z(G)) = 1$,
shows that
\begin{multline*}
\tag{$*$}
\sum_{
	h' \in (\stab_{\lsup g H}(\ox) \cap K_\sigma)
		\backslash\stab_{\lsup g H}(\ox)}
	\hat\phi(\lsup{h'g}\gamma_{\ge r})
= \sum_{h'}
	\Lambda(X^*(\lsup{h'}\mexp_x\inv(\lsup g\gamma_{\ge r})) \\
= \hat\mu_{X^*}^{\stab_{\lsup g H}(\ox)}(
		\mexp_x\inv(\lsup g\gamma_{\ge r})
	),
\end{multline*}
where $\hat\mu_{X^*}^{\stab_{\lsup g H}(\ox)}$ is the
function representing the distribution
\begin{equation}
\label{eq:mu-stab-ox}
f \mapsto
\int_{\stab_{\lsup g H}(\ox)/Z(G)}
	\hat f(\Ad^*(h')\inv X^*)
d\dot h'
\end{equation}
on $\lsup g\mf h^*$.
Now fix a double coset $K_\sigma g\stab_H(\ox)$ with $g \in \stab_G(\ox)$.
If there is no element $g'$ in the double coset such that
$\lsup{g'}\gamma_{< r} \in G'$, then, by
Corollary \ref{cor:step1-support}, we have that
$\dot\theta_\sigma(\lsup g\gamma_{< r}) = 0$,
so the summand corresponding to $g$ on the last line of
\eqref{eq:char-tau-big-supp} vanishes.
If $\lsup g\gamma_{< r} \in G'$, then,
since $\lsup g\gamma_{< r} \in \stab_G(\ox)$,
we have
$\lsup g\gamma_{< r} \in \stab_{G'}(\ox)$.
In particular, $\lsup g\gamma_{< r}$ is in the domain of
$\sigma$, so
$\dot\theta_\sigma(\lsup g\gamma_{< r})
= \theta_\sigma(\lsup g\gamma_{< r})$,
and the summands in \eqref{eq:char-tau} and,
by ($*$),
the last line of \eqref{eq:char-tau-big-supp}
corresponding to $g$ are the same.
Note that
$g x \in \BB_r(\lsup g\gamma)$,
so, since $\ox = g\ox$, also $x \in \BB_r(\lsup g\gamma)$.

Next, we compute $\Theta_\pi(\gamma)$ (in case $\bG'/Z(\bG)$
is $F$-anisotropic).
By Harish-Chandra's integral formula,
for any compact open subgroup \cpt of $G$,
\begin{equation}
\label{eq:char-pi-1}
\Theta_\pi(\gamma)
= \frac{\deg(\pi)}{\deg(\sigma)}\phi_d(\gamma)
\int_{G/Z(G)} \int_\cpt
	\dot\theta_\sigma(\lsup{g'c}\gamma)dc\,d\dot g',
\end{equation}
where $d\dot g'$ is a Haar measure on $G/Z(G)$, and
$dc$ is the Haar measure on \cpt normalized so that
$\meas(\cpt) = 1$.
(In characteristic $0$, this was proven for supercuspidal
representations --- in particular, for $\pi$ --- by
Harish-Chandra in \cite{hc:harmonic}.  In
\cite{rader-silberger:submersion}, Rader and Silberger
demonstrated an analogue of this result for discrete series
representations.
In \cite{adler-debacker:mk-theory}*{Appendix B},
Prasad provided a characteristic-free proof of a submersion
principle of Harish-Chandra.
Since the proof of the integral formula 
for characters (Theorem 12 of \cite{hc:harmonic}*{p.~60})
relies only on the submersion principle and Lemma 23 of
\cite{hc:harmonic}*{p.~59}, and since the proof of the
latter does not
depend on $\chr F$ being $0$, the correctness of the
integral formula in any characteristic follows.)

For the remainder of the proof,
we will assume that $\bG'/Z(\bG)$ is $F$-anisotropic.
We claim that the inner integral
in \eqref{eq:char-pi-1}
may be replaced by an
integral over $\cpt_H := \cpt \cap H\conn$.  (This is the analogue
of Lemma 10.0.7 of \cite{debacker-reeder:depth-zero-sc}.)
Indeed,
$$
\begin{aligned}
\frac{\deg(\sigma)}{\deg(\pi)}\phi_d(\gamma)\inv\Theta_\pi(\gamma)
& = \int_{G/Z(G)} \int_\cpt
	\dot\theta_\sigma(\lsup{g'c}\gamma)dc\,d\dot g' \\
& = \int_{G/Z(G)} \int_{\cpt_H}
	\int_\cpt \dot\theta_\sigma(\lsup{g'c k}\gamma)
	dc\,dk\,d\dot g' \\
& = \int_\cpt \int_{G/Z(G)}
	\int_{\cpt_H} \dot\theta_\sigma(\lsup{g'c k}\gamma)
	dk\,d\dot g'\,dc \\
& = \int_{G/Z(G)} \int_{\cpt_H}
	\dot\theta_\sigma(\lsup{g'k}\gamma)
	dk\,d\dot g',
\end{aligned}
$$
where $dk$ is the Haar measure on $\cpt_H$ normalized so
that $\meas(\cpt_H) = 1$.
(The equalities on the second and fourth lines
come from routine Haar measure manipulations.  The
interchange of integrals on the third line is justified by
Lemma \ref{lem:char-pi-cpt-supp}.)
Thus,
\begin{equation}
\label{eq:char-pi-H-int}
\Theta_\pi(\gamma)
= \phi_d(\gamma)
\sum_{g \in K_\sigma\backslash G/H}
	\frac{\deg(\pi)}{\deg(\sigma)}
	\int_{K_\sigma g H/Z(G)} \int_{\cpt_H}
		\dot\theta_\sigma(\lsup{g'k}\gamma)
	dk\,d\dot g'.
\end{equation}

Fix a double coset $K_\sigma g H$ in $G$.
Since
$g' \mapsto \dot\theta_\sigma(\lsup{g'k}\gamma)$
is invariant under left translation by $K_\sigma$,
we have that
\begin{multline}
\label{eq:char-pi-H-int-term}
\int_{K_\sigma g H/Z(G)} \int_{\cpt_H}
	\dot\theta_\sigma(\lsup{g'k}\gamma)
dk\,d\dot g'
\\
\begin{aligned}
& = \int_{K_\sigma(\lsup g H)/Z(G)} \int_{\lsup g\cpt_H}
	\dot\theta_\sigma(\lsup{y k'g}\gamma)
dk'\,d\dot y \\
& = \int_{K_\sigma(\lsup g H)/\lsup g H} \int_{\lsup g H/Z(G)} \int_{\lsup g\cpt_H}
	\dot\theta_\sigma(\lsup{y h'k'g}\gamma)
dk'\,d\dot h'\,\frac{dy}{dh} \\
& = \meas(K_\sigma(\lsup g H)/\lsup g H)\int_{\lsup g H/Z(G)} \int_{\lsup g\cpt_H}
	\dot\theta_\sigma(\lsup{h'k'g}\gamma)
dk\,d\dot h',
\end{aligned}
\end{multline}
where $d\dot y$ is the Haar measure on $G/Z(G)$ used to
compute $\deg(\pi)$,
$d\dot h'$ is Haar measure on $\lsup g H/Z(G)$ normalized as
in the statement of the theorem,
$dk'$ is the Haar measure on $\lsup g\cpt_H$ normalized so
that $\meas(\lsup g\cpt_H) = 1$,
and $dy/dh$ is the Haar measure on $G/H$ deduced from $d\dot y$
and $d\dot h$.
If there is no element $g'$ in the double coset such that
$\lsup{g'}\gamma_{< r} \in G'$ and $x \in \BB_r(\lsup{g'}\gamma)$,
then, by Corollary \ref{cor:wk-step1-support}, the summand
in \eqref{eq:char-pi-H-int} corresponding to $g$ vanishes.
Otherwise, we may, and do, assume that
$\lsup g\gamma_{< r} \in G'$ and $x \in \BB_r(\lsup g\gamma)$.
Now, by Corollary \ref{cor:trho-isotyp},
$$
\dot\theta_\sigma(\lsup{h'k'g}\gamma)
= \theta_\sigma(\lsup g\gamma_{< r})
	\dotm{[(\lsup g H)_{x, r}](\lsup{h'k'g}\gamma_{\ge r})}
	\hat\phi(\lsup{h'k'g}\gamma_{\ge r})
$$
for $h' \in \lsup g H$ and $k' \in \lsup g\cpt_H$;
and
$$
\meas(K_\sigma(\lsup g H)/\lsup g H)
= \meas(K_\sigma/Z(G))
\meas(K_\sigma \cap \lsup g H/Z(G))\inv
= \frac{\deg(\sigma)}{\deg(\pi)},
$$
where the last equality follows from the normalization
$\meas(K_\sigma \cap \lsup g H/Z(G)) = 1$
and the fact that
$\pi = \Ind_{K_\sigma}^G \sigma$.
Combining these two facts with
Lemma \ref{lem:orbital-cancel} (with $Z = Z(G)$)
and \eqref{eq:char-pi-H-int-term},
we see that the summands
in \eqref{eq:char-pi} and \eqref{eq:char-pi-H-int}
corresponding to $g$ are the same.
\end{proof}

We would like a way of describing the sum in Theorem
\ref{thm:char-tau|pi-1} as running over a set of conjugates
of $\gamma$, not over a set of elements conjugating
$\gamma$.  However, really we are interested only in
conjugates of $\gamma_{< r}$, not of $\gamma$.
We define below an equivalence relation $\simm{d - 1}$ on the set
$\mc T((\bG^i, \dotsc, \bG^d), (r_i, \dotsc, r_d))$
that makes this precise,
and then sum over equivalence classes for this relation in
Corollary \ref{cor:char-tau|pi-1}.
Since we will need them later, in fact we define a family of
equivalence relations $\simm i$.

\begin{dn}
\label{defn:equiv}
For $0 \le i < d$, let
\indexmem{simm-i}{\simm{i}}
$\simm{i}$ be the equivalence relation on
$\mc T((\bG^i, \dotsc, \bG^d), (r_i, \dotsc, r_d))$
such that,
for two elements $\delta$ and $\delta'$ of that set,
$\delta \simm{i} \delta'$ if and only if
$\delta'_{< r_j} \in \lsup{\stab_{G^j}(\ox)}\delta_{< r_j}$
for all $i \le j < d$.
\end{dn}

\begin{cor}
\label{cor:char-tau|pi-1}
If $x \in \BB_r(\gamma)$, then
\begin{equation}
\label{eq:char-tau-no-g}
\theta_\tau(\gamma)
= \phi_d(\gamma)
\sum
	\theta_\sigma(\gamma'_{< r})
	\hat\mu^{\stab_{H'}(\ox)}_{X^*}(
		\mexp_x\inv(\gamma'_{\ge r})
	).
\end{equation}
If $\bG'/Z(\bG)$ is $F$-anisotropic and $\gamma$ is regular
semisimple, then
\begin{equation}
\label{eq:char-pi-no-g}
\Theta_\pi(\gamma)
= \phi_d(\gamma)
\sum
	\theta_\sigma(\gamma'_{< r})
	\hat\mu^{H'}_{X^*}(
		\mexp_x\inv(\gamma'_{\ge r})
	).
\end{equation}
Here,
$\hat\mu_{X^*}^{\stab_{H'}(\ox)}$
and $\hat\mu_{X^*}^{H'}$ are defined with respect to the
Haar measure on $H'/Z(G)$ normalized so that
$\meas(K_\sigma \cap H'/Z(G)) = 1$,
and
the sums are taken over $\simm{d - 1}$-equivalence
classes of elements
$\gamma'
\in \lsup{\stab_G(\ox)}\gamma \cap \mc T((\bG', \bG), (r, r_d))$
(respectively,
$\gamma' \in \lsup G\gamma \cap \mc T((\bG', \bG), (r, r_d))$
with $x \in \BB_r(\gamma')$).
\end{cor}

The notations
$\mc T((\bG', \bG), (r, r_d))$ and $\simm{d - 1}$
are as in Definitions \ref{defn:trunc}
and \ref{defn:equiv}, respectively.
By abuse of notation, we have written $\bH'$ in place of
$C_\bG(\gamma'_{< r_{d - 1}})$, even though this group
depends on $\gamma'$.

As observed after Theorem \ref{thm:char-tau|pi-1}, by Lemma
\ref{lem:disc-orbital-int},
we may describe the orbital integrals over the possibly disconnected
groups $H'$ as sums of orbital integrals over ${H'}\conn$.

\begin{proof}
Let \mc G be a subgroup of $G$ containing
$\stab_{G'}(\ox)G_{x, 0{+}}$\,, and put
$\mc H = H \cap \mc G$
and
$\mc C = \sett{g \in \mc G}
	{$\lsup g\gamma \in \mc T((\bG', \bG), (r, r_d))$
	and
	$x \in \BB_r(\lsup g\gamma)$}$.

First, we claim that the natural map
$f_1 : \stab_{G'}(\ox)\backslash\mc G/\mc H
	\to \stab_{G'}(\ox)G_{x, 0{+}}\backslash\mc G/\mc H$
furnishes a bijection of
$\stab_{G'}(\ox)\backslash\mc C/\mc H$
with the set of $(\stab_{G'}(\ox)G_{x, 0{+}}, \mc H)$-double
cosets containing an element of \mc C.
The map is clearly surjective, so it suffices to show that
it is injective.
Suppose that $g_1, g_2 \in \mc C$ are such that
$$
\stab_{G'}(\ox)G_{x, 0{+}}g_1\mc H
= \stab_{G'}(\ox)G_{x, 0{+}}g_2\mc H.
$$
Since
$\stab_{G'}(\ox)G_{x, 0{+}}
= G_{x, 0{+}}\stab_{G'}(\ox)$,
we have that
$$
G_{x, 0{+}}g_1\mc H \cap \stab_{G'}(\ox)g_2\mc H \ne \emptyset
$$
--- say $k g_1\mc H = g'g_2\mc H$, with
$k \in G_{x, 0{+}}$ and $g' \in \stab_{G'}(\ox)$.
Then
$\lsup{k g_1}\gamma_{< r} = \lsup{g'g_2}\gamma_{< r}
	\in G'$.
Since $\lsup{g_1}\gamma_{< r} \in G'$
and $x \in \BB_r(\lsup{g_1}\gamma_{< r})$,
we have by Lemma \xref{exp-lem:rigidity}
and Corollary \xref{exp-cor:compare-centralizers}
of \cite{adler-spice:good-expansions} that
$k \in G'_{x, 0{+}}(\lsup{g_1}H)_{x, 0{+}}$\,.
In particular, $k g_1 \in \stab_{G'}(\ox)g_1 H$.
Since also $k g_1 \in \mc G$,
we have
$$
k g_1 \in \stab_{G'}(\ox)g_1 H \cap \mc G
\subseteq
\stab_{G'}(\ox)g_1 (H \cap \mc G)
=
\stab_{G'}(\ox)g_1 \mc H.
$$
Since $g'g_2 \in k g_1\mc H$, we have that
$g_2$ belongs to the same $(\stab_{G'}(\ox), \mc H)$-double
coset as $g_1$, as desired.

Second, notice that the map $f_2$ on
$\stab_{G'}(\ox)\backslash\mc C/H$
that sends a double coset
$\stab_{G'}(\ox)g\mc H$ to
the $\stab_{G'}(\ox)$-orbit of $\lsup g\gamma_{< r}$
is a well defined injection.

Third, consider the map $f_3$ from
$\mc S
= \set{\gamma' \in \lsup{\mc G}\gamma \cap \mc T((\bG', \bG), (r, r_d))}
	{x \in \BB_r(\gamma')}$
to the set of $\stab_{G'}(\ox)$-orbits in
$\lsup{\mc G}\gamma_{< r} \cap G'$ that sends an element
$\gamma' \in \mc S$
to the $\stab_{G'}(\ox)$-orbit of $\gamma'_{< r}$.
By definition, two elements
$\gamma', \gamma''
\in \lsup{\mc G}\gamma \cap \mc T((\bG', \bG), (r, r_d))$
have the same image if and only if
$\gamma' \simm{d - 1} \gamma''$.
Thus, the induced map on $\simm{d - 1}$-equivalence classes
in \mc S is an injection.
It is easy to see that the images of $f_2$ and $f_3$ are the
same.

Now we consider the composition
$f_3\inv \circ f_2 \circ f_1\inv$.
This furnishes a bijection of the set of
$(\stab_{G'}(\ox)G_{x, 0{+}}, \mc H)$-double cosets
containing an element of \mc C into the set of
$\simm{d - 1}$-equivalence classes in \mc S.
If $\mc G = \stab_G(\ox)$, then $\mc H = \stab_H(\ox)$;
the specified set of double
cosets is the indexing set for the sum in \eqref{eq:char-tau};
and the set of equivalence classes is the indexing set for the sum in
\eqref{eq:char-tau-no-g}.
It is easy to check that the summands match term-by-term,
so \eqref{eq:char-tau-no-g} holds.
Similarly, we demonstrate \eqref{eq:char-pi-no-g} by taking
$\mc G = G$ (and observing that $\stab_{G'}(\ox) = G'$ when
$\bG'/Z(\bG)$ is $F$-anisotropic).
\end{proof}

We now prove a single-orbit result in the spirit of
Murnaghan--Kirillov theory (see
\cites{debacker:thesis,
adler-debacker:mk-theory,
murnaghan:chars-sln,
murnaghan:chars-u3,
murnaghan:chars-classical,
murnaghan:chars-gln
}).
When $F$ has characteristic zero and $p$ is large,
the second statement is a special case of Theorem 5.3.1
of \cite{jkim-murnaghan:charexp}.

\begin{cor}
\label{cor:char-tau|pi-1-germ}
Suppose that there exists a bijection
$\mexp : \bigcup_{y \in \BB(\bG, F)} \fg_{y, 0{+}}
	\to \bigcup_{y \in \BB(\bG, F)} G_{y, 0{+}}$
such that, for all $y \in \BB(\bG, F)$, the restriction
$\mexp\bigr|_{\fg_{y, 0{+}}}$ has image in $G_{y, 0{+}}$ and
satisfies Hypothesis
\ref{hyp:strong-mock-exp} (for all tame maximal $F$-tori \bT
with $y \in \BB(\bT, F)$).

Fix $\gamma \in G$ such that $\gamma \in G_{y, r}$
for some $y \in \BB(\bG, F)$.
If $\gamma \in G_{x, r}$ (i.e., if we may take $y = x$), then
$$
\theta_\tau(\gamma)
= \phi_d(\gamma)\indx{\stab_G(\ox)}{K_\sigma}\inv
\deg(\tau)\hat\mu^{\stab_G(\ox)}_{X^*}(\mexp\inv(\gamma)).
$$
If $\bG'/Z(\bG)$ is $F$-anisotropic and
$\gamma$ is regular semisimple, then
$$
\Theta_\pi(\gamma)
= \phi_d(\gamma)\deg(\pi)\hat\mu^G_{X^*}(\mexp\inv(\gamma)).
$$
Here, $\hat\mu_{X^*}^{\stab_G(\ox)}$ and $\hat\mu_{X^*}^G$
are defined with respect to the Haar measure on $G/Z(G)$
normalized so that $\meas(K_\sigma/Z(G)) = 1$.
\end{cor}

\begin{proof}
Note that $\gamma$ trivially has a normal $r$-approximation,
and that $\gamma_{< r} = 1$, so $\bH = \bG$.
Since $\gamma = \gamma_{\ge r}$,
we have that
$x \in \BB_r(\gamma)$ if and only if
$\gamma \in G_{x, r}$\,.

Put $Y = \mexp\inv(\gamma)$.
If $\lsup G\gamma \cap G_{x, r} = \emptyset$, then,
by Hypothesis \ref{hyp:strong-mock-exp-equi},
$\lsup G Y \cap \fg_{x, r} = \emptyset$.
Therefore, by Lemma \ref{lem:orbital-cancel},
$\hat\mu^G_{X^*}(Y) = 0$.
If also $\bG'/Z(\bG)$ is $F$-anisotropic,
then, by Theorem \ref{thm:char-tau|pi-1},
we have that $\dot\theta_\sigma(\lsup g\gamma) = 0$
for $g \in G$, hence (by the Frobenius formula)
that $\Theta_\pi(\gamma) = 0$.

Thus we may, and do, assume that
$\gamma \in G_{x, r}$\,.
In particular, equation \eqref{eq:char-tau} holds,
and the sum on the right-hand side of that equation has a single summand,
so it becomes
\begin{equation}
\tag{\ref{eq:char-tau}$'$}
\theta_\tau(\gamma)
= \phi_d(\gamma)
\theta_\sigma(1)
\hat\mu^{\stab_G(\ox)}_{X^*}(Y)
= \phi_d(\gamma)
\deg(\sigma)
\hat\mu^{\stab_G(\ox)}_{X^*}(Y).
\end{equation}
Since
$\tau = \Ind_{K_\sigma}^{\stab_G(\ox)} \sigma \otimes \phi_d$,
we have that
$\deg(\tau) = \indx{\stab_G(\ox)}{K_\sigma}\deg(\sigma)$,
so
$$
\theta_\tau(\gamma)
= \phi_d(\gamma)\indx{\stab_G(\ox)}{K_\sigma}\inv
\deg(\tau)
\hat\mu^{\stab_G(\ox)}_{X^*}(Y).
$$

The second equality follows similarly from
\eqref{eq:char-pi} and the fact that
$\deg(\pi) = \meas(K_\sigma/Z(G))\deg(\sigma) = \deg(\sigma)$.
\end{proof}

\section{The full character formula}
\label{sec:full} 
Here we unroll the inductive formulas from \S\S\ref{sec:induction1}
and \ref{sec:induction2},
preserving the hypotheses of \S\ref{sec:induction2}.
In particular,
$\gamma$ is an element of $G$ with a normal
$r_{d - 1}$-approximation, which we fixed for
definiteness.
Thus, the elements $\gamma_{< r_i}$ are
unambiguously defined for $0 \le i < d$.
Choosing such an approximation also fixes approximations to
all truncations and conjugates of $\gamma$.

\begin{thm}
\label{thm:full-char}
If $x \in \BB_r(\gamma)$, then
\begin{equation}
\label{eq:full-char-tau}
\begin{aligned}
\theta_\tau(\gamma) =
\phi_d(\gamma)
  \sum &
	c(\vec\phi, \gamma'_{< r_{d - 1}})
  \Bigl(
	\prod_{i = 0}^{d - 1}
		\mf G(\phi_i, \gamma'_{< r_i})
		\varepsilon(\phi_i, \gamma'_{< r_i})
	\Bigr) \\
  & \times
	\Bigl(\prod_{i = 0}^{d - 1}
		\phi_i(\gamma'_{< r_i})
	\Bigr)
	\theta_{\rho_0'}(\gamma'_0)
	\prod_{i = 0}^{d - 1}
		\hat\mu^{\stab_{H^{i\,\prime}}(\ox)}_{X_i^*}(
			\mexp_x\inv(\gamma'_{(i)})
		).
\end{aligned}
\end{equation}
If $\bG'/Z(\bG)$ is $F$-anisotropic and $\gamma \in G$
is regular semisimple, then
\begin{equation}
\label{eq:full-char-pi}
\begin{aligned}
\Theta_\pi(\gamma) =
\phi_d(\gamma)
  \sum &
  	c(\vec\phi, \gamma'_{< r_{d - 1}})
  \Bigl(
	\prod_{i = 0}^{d - 1}
		\mf G(\phi_i, \gamma_{< r_i}')
		\varepsilon(\phi_i, \gamma_{< r_i}')
	\Bigr) \\
  & \times
	\Bigl(\prod_{i = 0}^{d - 1}
		\phi_i(\gamma'_{< r_i})
	\Bigr)
	\Theta_{\pi'_0}(\gamma'_0)
	\prod_{i = 0}^{d - 1}
		\hat\mu^{H^{i\,\prime}}_{X_i^*}(
			\mexp_x\inv(\gamma'_{(i)})
		).
\end{aligned}
\end{equation}
Here,
\begin{align*}
c(\vec\phi, \gamma'_{< r_{d - 1}})
= \prod_{i = 0}^{d - 1}
	& \Bigindx{\odc{\gamma'_{< r_i}; x, r_i}_{G^{i + 1}}}
		{\odc{\gamma'_{< r_i}; x, r_i}_{G^i}
			G^{i + 1}_{x, s_i}
		}^{1/2} \\
	& \times\Bigindx{\odc{\gamma'_{< r_i}; x, r_i{+}}_{G^{i + 1}}}
		{\odc{\gamma'_{< r_i}; x, r_i{+}}_{G^i}
			G^{i + 1}_{x, s_i{+}}
		}^{1/2},
\end{align*}
and
\begin{itemize}
\item
$\hat\mu_{X_i^*}^{\stab_{H^{i\,\prime}}(\ox)}$
and $\hat\mu_{X_i^*}^{H^{i\,\prime}}$ are defined with
respect to the Haar measure on
$H^{i\,\prime}/Z(G)$ normalized
so that $\meas(K_{\sigma_{i + 1}} \cap H^{i\,\prime}/Z(G)) = 1$,
\item
$\gamma'_{(i)} = (\gamma'_{< r_{i + 1}})_{\ge r_i}$
when $0 \le i < d - 1$,
and
\item
$\gamma'_{(d-1)} = \gamma'_{\ge r_{d - 1}}$.
\end{itemize}
The sums are taken over $\simm0$-equivalence classes of
elements
$\gamma' \in \lsup{\stab_G(\ox)}\gamma \cap \mc T(\vbG, \vec r)$
(respectively, $\gamma' \in \lsup G\gamma \cap \mc T(\vbG, \vec r)$
such that $x \in \BB_{r_{d - 1}}(\gamma')$).
\end{thm}

The notations $\mc T(\vbG, \vec r)$ and $\simm0$ are as in
Definitions \ref{defn:trunc} and \ref{defn:equiv},
respectively.
As in Corollary \ref{cor:char-tau|pi-1}, we have written
$\bH^{i\,\prime}$ in place of $C_{\bG^{i + 1}}(\gamma'_{< r_i})$
for $0 \le i < d$.

Note that, in particular, if $\bG'/Z(\bG)$ is
$F$-anisotropic, then the character of $\pi$ is supported on
conjugacy classes intersecting $\mc T(\vbG, \vec r)$.
A similar statement holds for the character of $\tau$.

Recall that the various roots of unity \mf G were defined and
computed in \S\ref{ssec:gauss}.

As observed after Theorem \ref{thm:char-tau|pi-1}, by Lemma
\ref{lem:disc-orbital-int},
we may describe the orbital integrals over possibly disconnected
groups in the above formula as sums of orbital integrals
over connected groups.

\begin{proof}
\emph{In this proof only}, we write $r_d$ for $\infty$.  This
conflicts with the notation in the rest of the paper, but it
makes the equations appearing below (for example,
\eqref{eq:full-char-tau-over-mc-S}) simpler.

For $0 \le i < d$, we may
apply Proposition \ref{prop:induction1}
and Corollary \ref{cor:char-tau|pi-1},
with $(\bG^i, \bG^{i+1})$
in place of $(\bG',\bG)$, to see that,
for all $\delta \in \stab_{G^{i+1}}(\ox)$
such that $x\in \BB_{r_i}(\delta)$,
\begin{equation}
\tag{$*_i$}
\begin{aligned}
\theta_{\tau_{i+1}} (\delta)
=
\phi_{i+1}(\delta)
\sum
&
\Bigindx{\odc{\delta_{< r_i}; x, r_i}_{G^{i + 1}}}
	{\odc{\delta_{< r_i}; x, r_i}_{G^i}
		G^{i + 1}_{x, s_i}
	}^{1/2} \\
&
\times
\Bigindx{\odc{\delta_{< r_i}; x, r_i{+}}_{G^{i + 1}}}
	{\odc{\delta_{< r_i}; x, r_i{+}}_{G^i}
		G^{i + 1}_{x, s_i{+}}
	}^{1/2} \\
&
\times\mf G(\phi_i, \delta'_{< r_i})
\varepsilon(\phi_i, \delta'_{< r_i}) \\
&
\times\theta_{\tau_i}(\delta'_{< r_i})
\hat\mu_{X^*_i}^{\stab_{H^{i\,\prime}}(\ox)}(
	\mexp_x\inv(\delta'_{\ge r_i})
),
\end{aligned}
\end{equation}
the sum taken over $\simm{i}$-equivalence classes of
elements
$\delta'
\in \lsup{\stab_{G^{i + 1}}(\ox)}\delta
	\cap \mc T((\bG^i, \bG^{i + 1}), (r_i, r_{i + 1}))$.
(The condition $x \in \BB_{r_i}(\delta')$ is automatically
satisfied here.)
Note that the set $\BB_{r_i}(\delta)$ and the equivalence
relation $\simm i$ are both constructed in the setting of
some ambient group, which is suppressed from the notation.
However, it is easy to see that changing the ambient group
from \bG to $\bG^{i + 1}$ corresponds simply to restricting
the equivalence relation $\simm i$;
and, since $x \in \BB(\bG^{i + 1}, F)$,
Lemma \xref{exp-lem:G-approx-is-G'-approx} of
\cite{adler-spice:good-expansions} shows that we do not need
to worry about what is the ambient group for the construction
of $\BB_{r_i}(\delta)$.

Put $\gamma^{(0)} = \gamma$.
We apply ($*_{d - 1}$) to describe
$\theta_\tau(\gamma) = \theta_{\tau_d}(\gamma)$ in
terms of the values of $\theta_{\tau_{d - 1}}$ at
truncations of various conjugates $\gamma^{(1)}$ of
$\gamma^{(0)} = \gamma^{(0)}_{< r_d}$;
then ($*_{d - 2}$) to describe each
$\theta_{\tau_{d - 1}}(\gamma^{(1)})$ in terms
of the values of $\theta_{\tau_{d - 2}}$ at truncations of
various conjugates $\gamma^{(2)}$ of
$\gamma^{(1)}_{< r_{d - 1}}$;
and so forth to obtain
\begin{equation}
\label{eq:full-char-tau-over-mc-S}
\begin{aligned}
\theta_\tau(\gamma)
= \phi_d(\gamma)\sum
& \underbrace{\Bigl(\prod_{i = 0}^{d - 1}
	\Bigindx{\odc{\gamma^{(d - i)}_{< r_i}; x, r_i}_{G^{i + 1}}}
		{\odc{\gamma^{(d - i)}_{< r_i}; x, r_i}_{G^i}
			G^{i + 1}_{x, s_i}
		}^{1/2}
\Bigr)}_{\text I} \\
&
\times
\underbrace{\Bigl(\prod_{i = 0}^{d - 1}
	\Bigindx{\odc{\gamma^{(d - i)}_{< r_i}; x, r_i{+}}_{G^{i + 1}}}
		{\odc{\gamma^{(d - i)}_{< r_i}; x, r_i{+}}_{G^i}
			G^{i + 1}_{x, s_i{+}}
		}^{1/2}
\Bigr)}_{\text I'} \\
& \times
\underbrace{\Bigl(\prod_{i = 0}^{d - 1}
	\mf G(\phi_i, \gamma^{(d - i)}_{< r_i})
	\varepsilon(\phi_i, \gamma^{(d - i)}_{< r_i})
\Bigr)}_{\text{II}} \\
& \times\underbrace{\Bigl(\prod_{i = 0}^{d - 2}
	\phi_{i + 1}(\gamma^{(d - i)})
\Bigr)
\theta_{\tau_0}(\gamma^{(d)}_{< r_0})}_{\text{III}}
\underbrace{\prod_{i = 0}^{d - 1}
	\hat\mu_{X_i^*}^{\stab_{H^{i\,\prime}}(\ox)}(
		\mexp_x\inv(
			\gamma^{(d - i)}_{\ge r_i}
		)
	)
}_{\text{IV}},
\end{aligned}
\end{equation}
the sum taken over the collection \mc S of $d$-tuples
$\bigl([\gamma^{(d - i)}]_i\bigr)_{i = 0}^{d - 1}$
with
$\gamma^{(d - i)}
\in \lsup{\stab_{G^{i + 1}}(\ox)}\gamma^{(d - i - 1)}_{< r_{i + 1}}
	\cap \mc T((\bG^i, \bG^{i + 1}), (r_i, r_{i + 1}))$
for $0 \le i < d$.
(Here, $[\delta]_i$ denotes the $\simm i$-equivalence class
of an element $\delta$ for $0 \le i < d$.)

If $\gamma' \in \lsup {\stab_G(\ox)}\gamma \cap \mc T(\vbG, \vec r)$,
then
$S(\gamma')
:= \bigl([\gamma'_{< r_{i + 1}}]_i\bigr)_{i = 0}^{d - 1}$
lies in \mc S.
It is an easy consequence of the definitions that,
for $\gamma'' \in \lsup {\stab_G(\ox)}\gamma \cap \mc T(\vbG, \vec r)$,
we have $S(\gamma') = S(\gamma'')$ if and only if
$\gamma' \simm0 \gamma''$.
On the other hand, suppose that
$\vec\gamma^{\,\prime}
= \bigl([\gamma^{(d - i)}]_i\bigr)_{i = 0}^{d - 1}
\in \mc S$.
Then, by definition, there are elements
$g_i \in \stab_{G^{i + 1}}(\ox)$ such that
$\gamma^{(d - i)}
= \lsup{g_i}\gamma^{(d - i - 1)}_{< r_{i + 1}}$
for $0 \le i \le d$.
One checks that
$\gamma' := \lsup{g_0\dotsb g_{d - 1}}\gamma
\in \lsup{\stab_G(\ox)}\gamma \cap \mc T(\vbG, \vec r)$,
and
$S(\gamma') = \vec\gamma^{\,\prime}$.
Thus, $S$ induces a bijection from the set of
$\simm0$-equivalence classes in
$\lsup {\stab_G(\ox)}\gamma \cap \mc T(\vbG, \vec r)$
onto \mc S,
so that we may regard the sum in
\eqref{eq:full-char-tau-over-mc-S} as running over the
former set.
Upon doing so, we notice that the product of terms (I) and
($\text I'$) becomes $c(\vec\phi, \gamma'_{< r_{d - 1}})$.
We calculate the remaining terms
appearing in \eqref{eq:full-char-tau-over-mc-S} as follows.
\begin{enumerate}[(I)]
\addtocounter{enumi}{1}
\item This matches with the corresponding term in
\eqref{eq:full-char-tau}.
\item
Since
$\theta_{\tau_0}(\gamma'_{< r_0})
= \phi_0(\gamma'_{< r_0})\theta_{\rho'_0}(\gamma'_{< r_0})$
and
$\theta_{\rho'_0}(\gamma'_{< r_0})
= \theta_{\rho'_0}(\gamma'_0)$,
this becomes
$\bigl(
	\prod_{i = 0}^{d - 1} \phi_i(\gamma'_{< r_i})
\bigr)
\theta_{\rho'_0}(\gamma'_0)$.
\item When we replace $\gamma^{(d - i)}$ by $\gamma'_{< r_{i + 1}}$,
the element
$\gamma^{(d - i)}_{\ge r_i}$
becomes
$(\gamma'_{< r_{i + 1}})_{\ge r_i} = \gamma'_{(i)}$,
even when $i = d - 1$ (because we have set $r_d = \infty$ in
this proof).  Thus, this matches up with the corresponding
term in \eqref{eq:full-char-tau}.
\end{enumerate}
Since \eqref{eq:full-char-tau-over-mc-S} holds, and can be
matched term-by-term with \eqref{eq:full-char-tau}, we also
have that \eqref{eq:full-char-tau} holds.

The argument carries over essentially unchanged to prove
\eqref{eq:full-char-pi} holds.  We sketch the few minor
differences.
Instead of using ($*_{d - 1}$), we
apply Proposition \ref{prop:induction1} and Corollary
\ref{cor:char-tau|pi-1} to obtain
\begin{equation}
\tag{${**}_{d - 1}$}
\begin{aligned}
\Theta_\pi(\gamma)
=
\phi_d(\gamma)
\sum
&
\Bigindx{\odc{\gamma^{(1)}_{< r_{d - 1}}; x, r_{d - 1}}}
	{\odc{\gamma^{(1)}_{< r_{d - 1}}; x, r_{d - 1}}_{G^{d - 1}}
		G^d_{x, s_{d - 1}}
	}^{1/2} \\
&
\times
\Bigindx{\odc{\gamma^{(1)}_{< r_{d-1}}; x, r_{d - 1}{+}}}
	{\odc{\gamma^{(1)}_{< r_{d-1}}; x, r_{d - 1}{+}}_{G^{d - 1}}
		G^d_{x, s_{d - 1}{+}}
	}^{1/2} \\
&
\times\mf G(\phi_{d - 1}, \gamma^{(1)}_{< r_{d - 1}})
\varepsilon(\phi, \gamma^{(1)}_{< r_{d - 1}}) \\
&
\times\theta_{\tau_{d - 1}}(\gamma^{(1)}_{< r_{d - 1}})
\hat\mu_{X^*_{d - 1}}^{H^{d - 1\,\prime}}(
	\mexp_x\inv(\gamma^{(1)}_{\ge r_{d - 1}})
),
\end{aligned}
\end{equation}
the sum taken over $\simm{d - 1}$-equivalence classes of
elements
$\gamma^{(1)} \in \lsup G\gamma \cap \mc T((\bG', \bG), (r, r_d))$
such that $x \in \BB_{r_{d - 1}}(\gamma^{(1)})$.
We then apply ($*_{d - 2}$), \tdotsc, ($*_0$) as before to
obtain a sum over a collection of $d$-tuples
$\bigl([\gamma^{(d - i)}]_i\bigr)_{i = 0}^{d - 1}$,
but this time we require only that
$\gamma^{(d - i)}
\in \lsup{G^{i + 1}}\gamma^{(d - i - 1)}_{< r_{i + 1}}
	\cap \mc T((\bG^i, \bG^{i + 1}), (r_i, r_{i + 1}))$
and $x \in \BB_{r_i}(\gamma^{(d - i)})$
for $0 \le i < d$
(that is, we allow $\gamma^{(d - i)}$ to range
over a \mbox{$G^{i + 1}$-,} not just a
\mbox{$\stab_{G^{i + 1}}(\ox)$-},
orbit).
The key point here is that
$G^{i + 1} = \stab_{G^{i + 1}}(\ox)$ unless
$i = d - 1$.
Again, we use (a natural extension of) the map
$S$ to identify this collection of $d$-tuples with the set
of $\simm0$-equivalence classes of elements
$\gamma' \in \lsup G\gamma \cap \mc T(\vbG, \vec r)$
with $x \in \BB_r(\gamma')$.
Finally, we note that $\rho'_0 = \pi'_0$.
\end{proof}

\appendix

\section{Mock-exponential map}
\label{sec:mock-exp}

In this section, \bG is a reductive algebraic $F$-group
(possibly disconnected) that splits over a tame extension of
$F$.  We state below two useful hypotheses (Hypotheses
\ref{hyp:mock-exp} and \ref{hyp:strong-mock-exp}) regarding
mock-exponential maps.  Proposition
\ref{prop:strong-mock-exp-holds} will show that the first
always holds, and the second holds when \bG is connected.
(Note that, with appropriate modifications in the statement of
Hypothesis
\ref{hyp:mock-exp-concave},
these hypotheses also make sense
for groups that are not tame.
However, they are not always satisfied for such groups.)

\begin{nhyp}
\label{hyp:mock-exp}
There exists a family
\indexmem{exp-E-xtu}{\mexp^E_{x, t:u}}
$$
(\mexp^E_{x, t:u} : \Lie(\bG)(E)_{x, t:u} \to \bG(E)_{x, t:u})
	_{\substack{
		\text{$E/F$ finite and tamely ramified} \\
		x \in \BB(\bG, E) \\
		t, u \in \tR_{> 0} \\
		t \le u \le 2t
	}}
$$
of isomorphisms such that,
given
\begin{itemize}
\item a finite, tamely ramified extension $E/F$,
\item $x \in \BB(\bG, E)$,
\item $t_1, t_2, u_1 \in \tR_{> 0}$ with $t_1 \le u_1 \le 2t_1$,
\item $X_j \in \Lie(\bG)(E)_{x, t_j}$ for $j = 1, 2$,
and
\item $g_j \in \mexp^E_{x, t_j:t_j{+}}(X_j)$ for $j = 1, 2$,
\end{itemize}
the following statements hold.
\begin{inc_enumerate}
\item\label{hyp:mock-exp-field-descend}
If $L$ is a field intermediate between $E$ and $F$
and $x \in \BB(\bG, L)$,
then the restriction of $\mexp^E_{x, t_1:u_1}$ to
$\Lie(\bG)(L)_{x, t_1:u_1}$ is $\mexp^L_{x, t_1:u_1}$\,.
\item\label{hyp:mock-exp-galois-equi}
If $E/F$ is Galois, then $\mexp^E_{x, t_1:u_1}$ is
$\Gal(E/F)$-equivariant.
\item\label{hyp:mock-exp-subquotient}
If $u_2 \in \tR_{> 0}$ and
$t_1 \le t_2 \le u_2 \le u_1 \le 2t_1$,
then $\mexp^E_{x, t_1:u_1}(X_2) \subseteq \mexp^E_{x, t_2:u_2}(X_2)$.
\item\label{hyp:mock-exp-Ad}
$(\Ad(g_1) - 1)X_2
\in (\mexp^E_{x, (t_1 + t_2):(t_1 + t_2){+}})\inv[g_1, g_2]$.
\item\label{hyp:mock-exp-ad}
$[X_1, X_2]
\in (\mexp^E_{x, (t_1 + t_2):(t_1 + t_2){+}})\inv[g_1, g_2]$.
\item\label{hyp:mock-exp-concave}
If
	\begin{itemize}
	\item \bT is a tame maximal $E$-torus,
	\item $x \in \BB(\bT, E)$,
	\item $f$ is a $\Gal(F\sep/E)$-invariant concave function on
$\wtilde\Phi(\bG, \bT)$,
and
	\item $t_1 \le f(\alpha) \le u_1$ for all
$\alpha \in \wtilde\Phi(\bG, \bT)$,
	\end{itemize}
then
$\mexp^E_{x, t_1:u_1}$ restricts to an isomorphism of the image
in $\Lie(\bG)(E)_{x, t_1:u_1}$ of $\lsub\bT\Lie(\bG)(E)_{x, f}$
with the image in $\bG(E)_{x, t_1:u_1}$ of $\lsub\bT\bG(E)_{x, f}$\,.
\end{inc_enumerate}
(For convenience, we have regarded $\mexp^E_{x, t_1:u_1}$ also
as a function on $\Lie(\bG)(E)_{x, t_1}$\,.)
\end{nhyp}

We will often suppress a superscript $F$, writing
$\mexp_{x, t:u}$ instead of $\mexp^F_{x, t:u}$\,.

\begin{rk}
\label{rem:mock-exp-Ad}
Note that, with the notation of Hypothesis \ref{hyp:mock-exp-Ad},
$\Ad(h)X \equiv X \pmod{\Lie(\bG)(E)_{x, t_1 + t_2}}$.
Analogous statements for $\Int$ and $\ad$ are proven in
Proposition 1.4.1 of \cite{adler:thesis}.
\end{rk}

\begin{rk}
\label{rem:mock-exp-suff-large}
Given Hypothesis \ref{hyp:mock-exp-field-descend}, it suffices
to verify Hypotheses
\ref{hyp:mock-exp-galois-equi}--\ref{hyp:mock-exp-concave} only
for $E$ ``sufficiently large tame Galois'', in the sense that, given a
finite, tamely ramified field extension $E/F$, there is a
finite, tamely ramified Galois superextension $M/F$ of $E/F$
for which the hypotheses hold.
Indeed, it is a straightforward application of Hypothesis
\ref{hyp:mock-exp-field-descend} that, if Hypotheses
\ref{hyp:mock-exp-galois-equi}--\ref{hyp:mock-exp-ad} hold
for such an $M$, then they hold for $E$ as well.

Now let $E$, $x$, $t_1$, and $u_1$ be as usual,
and let \bT and $f$ be as in Hypothesis \ref{hyp:mock-exp-concave}.
Let $M/F$ be a (finite, tame) Galois superextension of $E/F$ such that
Hypothesis \ref{hyp:mock-exp-concave} is satisfied for $M$,
$x$, $t_1$, $u_1$, \bT, and $f$.  
By further enlarging $M$ if necessary, we may, and do,
assume that \bG is $M$-split.

Let $\ol u_1$ be the constant function on
$\wtilde\Phi(\bG, \bT)$ with value $u_1$.
By assumption, $\mexp^M_{x, t_1:u_1}$ induces an isomorphism of
$\lsub\bT\Lie(\bG)(M)_{x, f:\ol u_1}$
with $\lsub\bT\bG(M)_{x, f:\ol u_1}$\,.
By Hypothesis \ref{hyp:mock-exp-galois-equi}, the
isomorphism is \mbox{$\Gal(M/F)$-,}
hence certainly
\mbox{$\Gal(M/E)$-,} equivariant; so we obtain an
isomorphism
$\lsub\bT\Lie(\bG)(M)_{x, f:\ol u_1}^{\Gal(M/E)}
\to \lsub\bT\bG(M)_{x, f:\ol u_1}^{\Gal(M/E)}$.
By Proposition \xref{exp-prop:H1-iso} of \cite{adler-spice:good-expansions}, we have
$H^1(M/F, \lsub\bT\bG(M)_{x, \ol u_1}) = \sset 0$,
so that the codomain of the isomorphism is
$\lsub\bT\bG(M)_{x, f}^{\Gal(M/E)}
	/
\lsub\bT\bG(M)_{x, \ol u_1}^{\Gal(M/E)}$.
By Lemma \xref{exp-lem:tame-descent} of \loccit, this is just
$\lsub\bT\bG(E)_{x, f:\ol u_1}$\,.
A similar, but easier, calculation shows that
the domain of the isomorphism is
$\lsub\bT\Lie(\bG)(E)_{x, f:\ol u_1}$\,.
Thus, Hypothesis \ref{hyp:mock-exp-concave} is satisfied for
$E$, $x$, $t_1$, $u_1$, \bT, and $f$, as desired.
\end{rk}

For the next three results, suppose that
\begin{itemize}
\item $E/F$ is a finite, tamely ramified extension,
\item \bT is a tame maximal $E$-torus in \bG,
\item $x \in \BB(\bT, E)$,
\item $t, u \in \tR_{> 0}$ with $t \le u \le 2t$,
and
\item $\mexp^E_{x, t:u}$ is as in Hypothesis
\ref{hyp:mock-exp}.
\end{itemize}

\begin{lm}
\label{lem:mock-exp-subgroup}
Suppose that \bH is a reductive $E$-subgroup of \bG
containing \bT.
Then $\mexp^E_{x, t:u}$ restricts to an isomorphism of
$\Lie(\bH)(E)_{x, t:u}$ with $\bH(E)_{x, t:u}$\,.
\end{lm}

\begin{proof}
Apply Hypothesis \ref{hyp:mock-exp-concave},
with $f$ the concave function on $\wtilde\Phi(\bG, \bT)$
that takes the value $t$ on
$\wtilde\Phi(\bH, \bT)$ and $u$ elsewhere.
\end{proof}

\begin{lm}
\label{lem:mock-exp-root-value}
If $Y \in \Lie(\bT)(E)_t$ and
$h \in \mexp^E_{x, t:t{+}}(Y) \cap \bT(E)$,
then
$\alpha(h) - 1 \equiv d\alpha(Y) \pmod{E_{t{+}}}$
for $\alpha \in \wtilde\Phi(\bG, \bT)$.
\end{lm}

\begin{proof}
Fix $\alpha\in \wtilde\Phi(\bG, \bT)$
and choose a non-zero, positive-depth
element $X \in \Lie(\bG)(E)_\alpha$.
Put $t' = \depth_x(X)$, and let $g$ be an element of
$\mexp^E_{x, t':t'{+}}(X)$.
Then, by Hypotheses
\ref{hyp:mock-exp-Ad} and \ref{hyp:mock-exp-ad}, we have
that
\begin{align*}
(\alpha(h) - 1)X  & = (\Ad(h) - 1)X \\
\intertext{and}
d\alpha(Y)X  & = [Y, X]
\end{align*}
lie in $(\mexp^E_{x, (t + t'):(t + t'){+}})\inv[h, g]$,
so
$(\alpha(h) - 1)X \equiv d\alpha(Y)X
\pmod{\Lie(\bG)(E)_{x, (t + t'){+}}}$.
Since $X \not\in \Lie(\bG)(E)_{x, t'{+}}$,
we have
$\alpha(h) - 1 \equiv d\alpha(Y) \pmod{E_{t{+}}}$.
\end{proof}

\begin{lm}
\label{lem:mock-exp-roots}
Suppose that \bT is $E$-split
and
$\alpha \in \wtilde\Phi(\bG, \bT)$.
Then $\mexp^E_{x, t:u}$ induces an isomorphism
of $\lsub E\mf u_{(\alpha + t):(\alpha + u)}$
with $\lsub E U_{(\alpha + t):(\alpha + u)}$,
where, for $c \in \tR$,
$\alpha + c$ denotes the unique affine root with gradient $\alpha$
satisfying $(\alpha + c)(x) = c$.
\end{lm}

\begin{proof}
Apply Hypothesis \ref{hyp:mock-exp-concave}
with $f$ the concave function on $\wtilde\Phi(\bG, \bT)$
that takes the value $t$ at $\alpha$ and $u$
elsewhere.
\end{proof}

\begin{nhyp}
\label{hyp:strong-mock-exp}
There exists a family
\indexmem{exp-Tx}{\mexp_{\bT, x}}
$$
(\mexp_{\bT, x} : \fg_{x, 0{+}} \to G_{x, 0{+}})
	_{\substack{
		\text{\bT a tame maximal $F$-torus in \bG} \\
		x \in \BB(\bT, F)
	}}
$$
of bijections such that, given
\begin{itemize}
\item a tame maximal $F$-torus \bT in \bG,
\item $x \in \BB(\bT, F)$,
\item $t \in \tR_{> 0}$,
and
\item $X \in \fg_{x, t}$\,,
\end{itemize}
the following statements hold.
\begin{inc_enumerate}
\item\label{hyp:strong-mock-exp-refines}
If $u \in \tR_{> 0}$ with $t \le u \le 2t$, then
$\mexp_{\bT, x}(X) \in \mexp_{x, t:u}(X)$.
\item\label{hyp:strong-mock-exp-equi}
For $g \in G$, we have
$\lsup g X \in \fg_{x, t}$ if and only if
$\lsup g (\mexp_{\bT, x}(X)) \in G_{x, t}$\,,
in which case
$\lsup g (\mexp_{\bT, x}(X))
\in \mexp_{x, t:t{+}}(\lsup g X)$.
\item\label{hyp:strong-mock-exp-subgroup}
If \bH is a reductive $F$-subgroup of \bG containing \bT,
then
$\mexp_{\bT, x}(\mf h_{x, 0{+}}) = H_{x, 0{+}}$\,.
\end{inc_enumerate}
\end{nhyp}

\begin{pn}
\label{prop:strong-mock-exp-holds}
Hypothesis \ref{hyp:mock-exp} holds.
If \bG is connected, then
Hypothesis \ref{hyp:strong-mock-exp} also holds.
\end{pn}

In fact, one can imitate the argument of Proposition
1.6.7 of \cite{adler:thesis} to show that Hypothesis
\ref{hyp:strong-mock-exp} holds even if \bG is not assumed
to be connected; but we do not need to do so here.

\begin{proof}
The only difference between the hypotheses in question for
\bG and $\bG\conn$ is in Hypothesis
\ref{hyp:strong-mock-exp-equi}, so it suffices to assume
throughout that \bG is connected.

See \cite{adler:thesis}*{\S 1.5} for a description of a family of bijections
$$
(\varphi_{\bT(E), x; t, u} : \Lie(\bG)(E)_{x, t:u} \to \bG(E)_{x, t:u})
	_{\substack{
		\text{$E/F$ finite and tamely ramified} \\
		\text{\bT a tame maximal $E$-torus in \bG} \\
		x \in \BB(\bT, E) \\
		t, u \in \R \\
		t \le u \le 2t
	}}
$$
(in fact, of isomorphisms, by Proposition 1.6.2(a) of \loccit).
By Corollary 1.6.6 of \loccit,
given $E$, $x$, $t$, and $u$, the map
$\varphi_{\bT(E), x; t, u}$ does not depend on the choice of 
tame maximal $E$-torus \bT such that $x \in \BB(\bT, E)$; so
we may write $\mexp^E_{x, t:u} = \varphi_{\bT(E), x; t, u}$
for any such choice.
These maps are the \emph{Moy--Prasad isomorphisms}.
It is easy to extend the definition of $\mexp^E_{x, t:u}$ to
all $t, u \in \tR$ (not just $t, u \in \R$)
satisfying $t \le u \le 2t$.

By construction, the resulting family satisfies
Hypotheses
\ref{hyp:mock-exp-field-descend}--\ref{hyp:mock-exp-subquotient}.
Hypothesis \ref{hyp:mock-exp-ad} is just Proposition
1.6.2(b) of \cite{adler:thesis}.
Hypothesis \ref{hyp:mock-exp-Ad} is Propositions 1.6.2(b)
and 1.6.3 of \loccit
By Remark \ref{rem:mock-exp-suff-large}, it suffices to show
Hypothesis \ref{hyp:mock-exp-concave} in case \bG is
$E$-split, in which case it follows (as in the proof of
Proposition 1.9.2 of \cite{adler:thesis})
from Definition \xref{exp-defn:vGvr-split} of \cite{adler-spice:good-expansions} and
the construction of $\mexp^E_{x, t:u}$\,.

In \cite{adler:thesis}*{\S 1.5} one also finds a
family of bijections
$$
(\varphi_{T, x} : \fg_{x, 0{+}} \to G_{x, 0{+}})
	_{\substack{
		\text{\bT a tame maximal $F$-torus in \bG} \\
		x \in \BB(\bT, F).
	}}
$$
For fixed \bT and $x$, the bijection $\varphi_{T, x}$ is
constructed (as in \S 1.3 of \loccit)
from the various bijections $\mexp_{x, t:t{+}}$
after choosing representatives
of the $\fg_{x, t{+}}$-cosets in $\fg_{x, t}$
and the $G_{x, t{+}}$-cosets in $G_{x, t}$
for $t \in \R$.
Regardless of the choice of representatives, the
construction ensures that the maps
$\mexp_{\bT, x} := \varphi_{T, x}$ satisfy Hypothesis
\ref{hyp:strong-mock-exp-refines};
and it is observed in Remark 4.1.1 of
\cite{adler-debacker:mk-theory} that they satisfy
Hypothesis \ref{hyp:strong-mock-exp-equi}.
However, we must make our choices with some delicacy to
ensure that Hypothesis \ref{hyp:strong-mock-exp-subgroup} is
satisfied.

The choices are made as follows.
Fix $t \in \R$.
Note that, if $\set{\bH^j}{j \in J}$ is a family of (not
necessarily connected) reductive
$F$-subgroups of \bG containing \bT, then
$\bigcap_{j \in J} \bH_j$ is also (not necessarily
connected) reductive;
and, by Lemma \xref{exp-lem:more-vGvr-facts} of \cite{adler-spice:good-expansions},
$$
\bigcap_{j \in J} H^j_{x, t}G_{x, t{+}}
= \bigcap_{j \in J} (H^{j\,\circ}, G)_{x, (t, t{+})}
= \Bigl(\bigl(\bigcap_{j \in J} H^j\bigr)\conn, G\Bigr)_{x, (t, t{+})}
\subseteq \Bigl(\bigcap_{j \in J} H^j_{x, t}\Bigr)G_{x, t{+}}\,.
$$
Thus, if $g \in G_{x, t}$\,, then there exists an element
$h \in g G_{x, t{+}}$ such that $h \in H$ whenever
\bH is a reductive $F$-subgroup of \bG containing \bT
such that $H \cap g G_{x, t{+}} \ne \emptyset$.
We choose any such element $h$ as the representative for the
coset $g G_{x, t{+}}$\,.
Similarly,
we choose, for every $X \in \fg_{x, t}$\,,
a coset representative $Y$ for $X + \fg_{x, t{+}}$ such that
$Y \in \mf h$ whenever \bH is a reductive $F$-subgroup of
\bG containing \bT such that
$\mf h \cap X + \fg_{x, t{+}} \ne \emptyset$.

By Lemma \ref{lem:mock-exp-subgroup}, if $t \in \R$,
$X \in \fg_{x, t}$\,,
and $g \in G_{x, t}$\,, then the sets of reductive
$F$-subgroups \bH of \bG, containing \bT, such that
$\mf h \cap X + \fg_{x, t{+}} \ne \emptyset$,
and such that $H \cap g G_{x, t{+}} \ne \emptyset$,
are the same.
Thus it follows from the construction that Hypothesis
\ref{hyp:strong-mock-exp-subgroup} is satisfied for
$\mexp_{\bT, x}$.
\end{proof}

%

\section{An orbital-integral formula}

In this section, suppose that
\begin{itemize}
\item \bG is a reductive algebraic $F$-group (possibly disconnected),
split over a tame extension of $F$,
satisfying Hypotheses \ref{hyp:mock-exp}
and \ref{hyp:strong-mock-exp},
\item $(x, r) \in \BB(\bG, F) \times \R_{> 0}$,
\item $X^* \in \fg^*_{x, -r} \smallsetminus \fg^*_{(-r){+}}$
satisfies condition \textbf{GE1} of \cite{yu:supercuspidal}*{\S 8},
\item $\phi$ is a linear character of $G_{x, r:r{+}}$\,,
and
\item
$\phi \circ \mexp_{x, r:r{+}}
= \Lambda \circ X^*$.
\end{itemize}
Note that, by Proposition \ref{prop:strong-mock-exp-holds},
the requirement that \bG satisfy Hypothesis
\ref{hyp:mock-exp} is superfluous.
Put $\bG' = C_\bG(X^*)\conn$.
We choose, arbitrarily, a tame maximal $F$-torus \bT with
$x \in \BB(\bT, F)$, and write $\mexp_x = \mexp_{\bT, x}$,
where $\mexp_{\bT, x}$ is as in Hypothesis
\ref{hyp:strong-mock-exp}.
(The resulting map will depend on our choice of \bT, but
this dependence will not affect our proof.)

\begin{lm}
\label{lem:phi-commute}
Suppose that
\begin{itemize}
\item $d \in \R_{> 0}$ with $d < r$,
\item $h \in G_{x, d}$\,,
and
\item $[h, \hat\phi]$ is trivial on $G_{x, r - d}$\,.
\end{itemize}
Then $h \in (G', G)_{x, (d, d{+})}$.
\end{lm}

Here, $[h, \hat\phi]$ is the character
$g \mapsto \hat\phi([h\inv, g])$ of $G_{x, r - d}$\,.

\begin{proof}
Fix $X \in \fg_{x, r - d}$\,, and
let $g$ be any element of $\mexp_{x, (r - d):(r - d){+}}(X)$.
By Hypothesis \ref{hyp:mock-exp-Ad},
$[h\inv, g]$ belongs to
$\mexp_{x, r:r{+}}\bigl((\Ad(h)\inv - 1)X\bigr)$, so
\begin{multline*}
1 = [h, \hat\phi](g)
= \hat\phi([h\inv, g]) \\
= \Lambda\bigl(X^*\bigl((\Ad(h)\inv - 1)X\bigr)\bigr)
= \Lambda\bigl((\Ad^*(h) - 1)X^*(X)\bigr).
\end{multline*}
Since $X \in \fg_{x, r - d}$ was arbitrary, we have
$(\Ad^*(h) - 1)X^* \in \fg^*_{x, (d - r)+}$\,.  By
Lemma 8.4 of \cite{yu:supercuspidal} (the proof of which
uses only condition \textbf{GE1}, not the full definition of
genericity),
$h \in (G', G)_{x, (d, d{+})}$.
\end{proof}

Since it is often easier to work with orbital integrals over
connected groups than over disconnected ones, we present a
basic result describing a $G$-orbital integral as a sum of
$G\conn$-orbital integrals.
Choose an invariant measure $d\dot h$ on $G/C_G(X^*)$.
Fix $g \in G$.
Then $d\dot h$ affords a natural choice of invariant measure on
$G/\lsup g C_G(X^*) = G/C_G(\lsup gX^*)$.
Since $G\conn/C_{G\conn}(\lsup gX^*)$ embeds naturally as an
open subset of $G/C_G(\lsup gX^*)$, it inherits this
invariant measure.  For convenience, we will write again
$d\dot h$ for the various measures occurring in this
way.
We will take all orbital integrals with respect to these
measures.
Notice that the induced measure on each
$G/C_G(\lsup gX^*)G\conn$ is just counting measure.

\begin{lm}
\label{lem:disc-orbital-int}
We have that
$\hat\mu^G_{X^*}
= \indx{C_G(X^*)}{C_{G\conn}(X^*)}\inv\sum_{g \in G/G\conn}
	\hat\mu^{G\conn}_{\lsup gX^*}$.
\end{lm}

The statement may be interpreted as an equality of
distributions, or of functions.  We will prove the equality
of distributions; that of functions follows immediately.

\begin{proof}
For $f \in C_c^\infty(\fg)$, we have that
\begin{multline*}
\hat\mu^G_{X^*}(f)
= \int_{G/C_G(X^*)} \hat f(\lsup hX^*)d\dot h \\
= \sum_{g \in G/C_G(X^*)G\conn} \int_{G\conn/C_{G\conn}(X^*)}
	\hat f(\lsup{g h}X^*)d\dot h \\
= \indx{C_G(X^*)}{C_{G\conn}(X^*)}\inv
	\sum_{g \in G/G\conn} \int_{G\conn/C_{G\conn}(X^*)}
		\hat f(\lsup{g h}X^*)d\dot h.
\end{multline*}
For fixed $g \in G$, the inner integral is
just
$\int_{G\conn/C_{G\conn}(\lsup gX^*)} \hat f(\lsup{h g}X^*)d\dot h
= \hat\mu^{G\conn}_{\lsup gX^*}(f)$.
The result follows.
\end{proof}

\begin{rem}
\label{rem:disc-orbital-int}
Our result as stated does not actually require that $X^*$
satisfy condition \textbf{GE1} of
\cite{yu:supercuspidal}*{\S8},
only that the relevant orbital integrals converge.  In
particular, the result holds for any semisimple $X^*$.
\end{rem}

Compare the next result to Lemma
6.3.5 of \cite{adler-debacker:mk-theory}.

\begin{lm}
\label{lem:orbital-cancel}
Suppose that
\begin{itemize}
\item $\bG'/Z(\bG')$ is $F$-anisotropic,
\item $Z$ is a cocompact subgroup of $Z(G\conn)$
that is normal in $G$,
\item \cpt is a compact open subgroup of $G\conn$,
and
\item $Y \in \bigcup_{y \in \BB(\bG, F)} \fg_{y, r}$ is regular semisimple.
\end{itemize}
If $Y \in \fg_{x, r}$ and $\gamma = \mexp_x(Y)$,
then
$$
\hat\mu_{X^*}^G(Y)
= \int_{G/Z} \int_\cpt
	[G_{x, r}](\lsup{g k}\gamma)\hat\phi(\lsup{g k}\gamma)
dk\,d\dot g,
$$
where $d\dot g$ is a Haar measure on $G/Z$ and
$dk$ is the Haar measure on \cpt, normalized so that
$\meas \cpt = 1$.
If $\lsup G Y \cap \fg_{x, r} = \emptyset$, then
$$
\hat\mu_{X^*}(Y) = 0.
$$
\end{lm}

When $Y \in \fg_{x, r}$, the notation of the lemma is meant to convey that
the integrand equals $\hat\phi(\lsup{g k}\gamma)$
if $\lsup{g k}\gamma \in G_{x, r}$\,,
and equals $0$ otherwise.

\begin{proof}
Since $G_{x, r} = G\conn_{x, r}$,
and since, for $g \in G$,
replacing $X^*$ by $\lsup gX^*$ has the effect of changing
$\phi$ to $\phi^g$,
we see by Lemma \ref{lem:disc-orbital-int} that it suffices
to prove this result in case $G$ is connected.

The map
$g
\mapsto
\min \set{\depth_x(\lsup{k_1 g k_2}Y)}
	{k_1 \in G_{x, 0+}, k_2 \in \cpt}$
on $G$ is locally constant.
For $t \in \R \cup \sset\infty$,
denote by $G(t)$ the preimage of $t$
under this map, and put
\begin{gather*}
I(t) = \int_{G(t)/Z} \int_\cpt \Lambda(X^*(\lsup{gk}Y))dk\,d\dot g \\
\intertext{and}
I'(t) = \int_{G(t)/Z} \int_\cpt
	[\fg_{x, r}](\lsup{gk}Y)\Lambda(X^*(\lsup{gk}Y))dk\,d\dot g.
\end{gather*}
We show that $I(t) = I'(t)$ for $t \in \R \cup \sset\infty$.

Fix $t \in \R \cup \sset\infty$.
If $t \ge r$, then the desired equality is obvious,
so we may suppose that $t < r$.
Denote by
$t = t_0 < t_1 < \cdots < t_m$ the distinct values of $\depth_x$ in $[t, r)$.
Given a compact open subgroup $\cpt'$ of $G$ such that
$k'G(t) = G(t)$ for all $k' \in \cpt'$
(for example, an open subgroup of $G_{x, 0{+}}$), we have
\begin{multline*}
I(t) = \indx{\cpt'}{\cpt''}
\int_{G(t)/Z} \int_{\cpt''} \int_\cpt
	\Lambda(X^*(\lsup{k'g k}Y))dk\,dk'\,d\dot g \\
= \int_{G(t)/Z}
	\sum_{h \in \cpt''\backslash\cpt'}
	\int_{\cpt''} \int_\cpt
	\Lambda(X^*(\lsup{k'h g k}Y))dk\,dk'\,d\dot g \\
= \int_{G(t)/Z} \int_{\cpt'} \int_\cpt
	\Lambda(X^*(\lsup{k'g k}Y))dk\,dk'\,d\dot g \\
= \int_{G(t)/Z} \int_\cpt \int_{\cpt'}
	\Lambda(X^*(\lsup{k'g k}Y))dk'\,dk\,d\dot g,
\end{multline*}
where
$\cpt'' = \cpt' \cap G_{x, (r - t){+}}$ and
$dk'$ is the Haar measure on $\cpt'$, normalized so
that $\meas(\cpt') = 1$.
An easy generalization allows us to handle several subgroups
$\cpt'_0, \dotsc, \cpt'_m$ as above, so we obtain
\begin{equation}
\tag{$**$}
I(t)
= \int_{G(t)/Z} \int_\cpt \int_{G_{x, r - t_m}} \cdots \int_{G_{x, r - t_0}}
	\Lambda(X^*(\lsup{h_m\cdots h_0 gk}Y))
dh_0\dotsb dh_m\,dk\,d\dot g,
\end{equation}
where, for $j = 0, \dotsc, m$,
$dh_j$ is the Haar measure on $G_{x, r - t_j}$
normalized so that $\meas(G_{x, r - t_j}) = 1$.
Now fix $g \in G(t)$ and $k \in \cpt$,
and put $\wtilde Y = \lsup{gk}Y$.
Then either $[\fg_{x, r}](\wtilde Y) = 1$,
or there is $0 \le j \le m$ such that $\depth_x(\wtilde Y) = t_j$
(in particular, $\wtilde Y \not\in \fg_{x, t_j{+}}$).
In the latter case,
fix $h_i \in G_{x, r - t_i}$ for $0 \le i \le m$.
By Remark \ref{rem:mock-exp-Ad},
we have that
$\lsup{h_j}\wtilde Y - \wtilde Y \in \fg_{x, r}$
and
(since
$h_{j - 1}\dotsb h_0
\in G_{x, r - t_{j - 1}} \subseteq G_{x, (r - t_j){+}}$)
also that
$\lsup{h_{j - 1}\dotsb h_0}\wtilde Y - \wtilde Y
\in \fg_{x, r{+}}$\,,
hence that
$\lsup{h_j}(
	\lsup{h_{j - 1}\dotsb h_0}\wtilde Y - \wtilde Y
) \in \fg_{x, r{+}}$\,.
Thus, by another application of Remark \ref{rem:mock-exp-Ad},
\begin{multline*}
\lsup{h_m\dotsb h_0}\wtilde Y
	- \lsup{h_m\dotsb h_{j + 1}}\wtilde Y \\
= \lsup{h_m\dotsb h_{j + 1}}\bigl(
	\lsup{h_j}(
		\lsup{h_{j - 1}\dotsb h_0}\wtilde Y - \wtilde Y
	)
	+ \lsup{h_j}\wtilde Y - \wtilde Y
\bigr) \\
\in \lsup{h_j}\wtilde Y - \wtilde Y + \fg_{x, r{+}}\,.
\end{multline*}
That is,
$$
\Lambda(X^*(\lsup{h_m\dotsb h_0}\wtilde Y))
= \Lambda(X^*(\lsup{h_m\dotsb h_{j + 1}}\wtilde Y))
	\Lambda(X^*(\lsup{h_j}\wtilde Y - \wtilde Y)),
$$
so the inner integrals in ($**$) are a constant multiple of
\begin{equation}
\tag{$*{*}*$}
\int_{G_{x, r - t_j}} \Lambda(X^*(\lsup{h_j}\wtilde Y - \wtilde Y))dh_j.
\end{equation}
Now, for $h_j, h_j' \in G_{x, r - t_j}$, we have
by Remark \ref{rem:mock-exp-Ad}
that
$\lsup{h_j'}\wtilde Y - \wtilde Y \in \fg_{x, r}$\,,
so
\begin{multline*}
\lsup{h_j h_j'}\wtilde Y - \wtilde Y
= \lsup{h_j}(\lsup{h_j'}\wtilde Y - \wtilde Y)
	+ (\lsup{h_j}\wtilde Y - \wtilde Y) \\
\in (\lsup{h_j'}\wtilde Y - \wtilde Y)
	+ (\lsup{h_j}\wtilde Y - \wtilde Y)
	+ \fg_{x, r{+}}\,.
\end{multline*}
That is,
$\varphi_{\wtilde Y}
: h_j \mapsto \Lambda(X^*(\lsup{h_j}\wtilde Y - \wtilde Y))$
is a homomorphism on $G_{x, r - t_j}$\,; so
($*{*}*$) equals $0$ unless $\varphi_{\wtilde Y}$ is
trivial there.

Suppose that it is trivial.
Choose $\tilde\gamma \in \mexp_{x, t_j:t_j{+}}(\wtilde Y)$.
By Hypothesis \ref{hyp:mock-exp-Ad},
$[h_j, \tilde\gamma]
\in \mexp_{x, r:r{+}}((\Ad(h_j) - 1)\wtilde Y)$,
so
$$
[\tilde\gamma\inv, \hat\phi](h_j)
= \hat\phi([h_j, \tilde\gamma])\inv
= \Lambda(X^*((\Ad(h_j) - 1)\wtilde Y)\inv = 1,
$$
for $h_j \in G_{x, r - t_j}$\,.
That is, $[\tilde\gamma\inv, \hat\phi]$ is trivial on
$G_{x, r - t_j}$\,; so, by Lemma \ref{lem:phi-commute},
$\tilde\gamma \in (G', G)_{x, (t_j, t_j{+})}$.
Then Hypothesis \ref{hyp:mock-exp-concave} gives
$\wtilde Y \in (\fg', \fg)_{x, (t_j, t_j{+})}$.
On the other hand, $\wtilde Y \in \fg_r \subseteq \fg_{t_j+}$\,.
Therefore, by Corollary 3.7.8 of \cite{adler-debacker:bt-lie},
we have that
$\wtilde Y \in \fg'_{t_j+} + \fg_{x, t_j+}$\,.
Since $\bG'/Z(\bG')$ is $F$-anisotropic, we have that
$\fg'_{t_j+} = \fg'_{x, t_j+} \subseteq \fg_{x, t_j{+}}$\,, so in fact
$\wtilde Y \in \fg_{x, t_j+}$\,, a contradiction.

We have shown that
\begin{multline*}
\int_{G_{x, r - t_m}} \cdots \int_{G_{x, r - t_0}}
    \Lambda(X^*(\lsup{h_m\cdots h_0}\wtilde Y))dh_0\cdots dh_m \\
= (\text{const})
  \int_{G_{x, r - t_j}} \Lambda(X^*(\lsup{h_j}\wtilde Y - \wtilde Y))dh_j
= 0
\end{multline*}
whenever $\wtilde Y \not\in \fg_{x, r}$\,.  Thus
$I(t) = I'(t)$.
By a theorem of Huntsinger
(see \cite{adler-debacker:mk-theory}*{Appendix A}),
$\hat\mu_{X^*}(Y) = \sum_{t \in \R \cup \sset\infty} I(t)$;
so
$\hat\mu_{X^*}(Y) = \sum_{t \in \R \cup \sset\infty} I'(t)$.
In particular, if
$\lsup G Y \cap \fg_{x, r} = \emptyset$, then
$\hat\mu_{X^*}(Y) = 0$.

Now suppose that $Y \in \fg_{x, r}$\,.
By Hypothesis \ref{hyp:strong-mock-exp-equi},
$\lsup{gk}\gamma \in G_{x, r}$ if and only if $\lsup{gk}Y \in \fg_{x, r}$\,,
in which case
$\lsup{gk}\gamma$ belongs to $\mexp_{x, r:r{+}}(\lsup{gk}Y)$.
Recall that
$\hat\phi \circ \mexp_{x, r:r{+}} = \Lambda \circ X^*$
on $\fg_{x, r:r{+}}$\,.
Thus
\begin{multline*}
\int_{G/Z} \int_\cpt
	[G_{x, r}](\lsup{gk}\gamma)\hat\phi(\lsup{gk}\gamma)dk\,d\dot g \\
=
\int_{G/Z} \int_\cpt
	[\fg_{x, r}](\lsup{gk}Y)\Lambda(X^*(\lsup{gk}Y))dk\,d\dot g \\
= \sum_{t \in \R} I'(t) = \hat\mu_{X^*}(Y).
\text{\qedhere}
\end{multline*}
\end{proof}



\newpage 


\setlength{\columnseprule}{.5pt}
\begin{multicols}{2}[\section*{Index of notation}]
\let\olditem\item
\def\item#1,#2{\olditem
	\makebox[.45\textwidth]{\ensuremath{#1} \hfill p.\ \pageref{index:#2}}}

\begin{trivlist}

  \item {\tR }, {tR}
  \item {\sgn_\F}, {sgnF}
  \item {\ord }, {ord}
  \item {\ff _F, \: \ff}, {ffF}
  \item {p}, {p}
  \item {F ^{\mathrm {un}}}, {Fun}
  \item {F ^{\mathrm {sep}}}, {Fsep}
  \item {\Lambda }, {Lambda}
  \item {\hat \mu _{X^*}^G}, {hat-mu}

\indexspace

  \item {\stab _H({{{\overline {x}}}})}, {stabHx}
  \item {\Phi ({{\mathbf G}},{{\mathbf T}}), \:
	\wtilde\Phi({{\mathbf G}},{{\mathbf T}})}, {PhiGT}
  \item {\Lie ({{\mathbf G}})_\alpha,\: \pmb{\mf g}_\alpha}, {LieGalpha}
  \item {\bU _\alpha , \quad \text{$\alpha $ a root}},
	{Ualpha}
  \item {\psi{+}}, {psi+}
  \item {\lsub E U_\psi,\: U_\psi ,\quad \text{$\psi$ an affine root}}, {EUpsi}
  \item {\lsub E {{\mathfrak {u}}}_\psi,\: \mathfrak u_\psi }, {Eupsi}
  \item {\depth _x}, {depthx}
  \item {F_{r:t},\: U_{\psi_1:\psi_2},\: G_{x, r:t}}, {filt-quot}
  \item {\mexp _{x, t:u}, \: \mexp_x}, {exp-xtu}
  \item {\lsub {{\mathbf T}}G_{x, f}}, {TGxf}
  \item {\lsub {{\mathbf T}}\Lie  (G)_{x, f}}, {TLieGxf}

  \item {\vG _{x, \vec{r}}}, {vGxr}
  \item {\Lie ({\vG})_{x, \vec{r}}}, {vLieGxr}
  \item {\Lie ({}_{{{\mathbf T}}}G_{x, f})}, {LieTGxf}
  \item {\Lie (\vG_{x, \vec{r}}}) , {LievGxr}

\indexspace

  \item {{C_{\bG }^{(r)}}(\gamma ), \: {C_{G}^{(r)}}(\gamma )}, {CbGrgamma}
  \item {{Z_{\bG }^{(r)}}(\gamma ), \: {Z_{G}^{(r)}}(\gamma )}, {ZbGrgamma}
  \item {\gamma _{< r} ,  \: \gamma _{\ge r}}, {gamma<r}

  \item {\BB _r(\gamma )}, {Brgamma}
  \item {\dc }, {dc}
  \item {\dc ^{(j)}}, {dcj}
  \item {\dc_{G'} , \: \dc_{G'} ^{(j)}}, {dcG'}
  \item {\mc T(\vbG , \vec{r})}, {TvGvr}

\indexspace

  \item {\Theta _\pi }, {Theta-pi}
  \item {\theta _\rho }, {theta-rho}
  \item {\dot \theta _\rho }, {dot-theta-rho}

\indexspace

  \item {\depth _x(\phi ), \quad \text{$\phi $ a character}}, {depth-x-char}
  \item {\bG ^i}, {bGi}

  \item {x }, {x}
  \item {r _i}, {r_i}
  \item {\phi _i}, {phi_i}
  \item {s _i}, {s_i}

  \item {K ^i}, {K-i}
  \item {J ^i , \: J ^i_+}, {J-i}

  \item {\rho _i'}, {rho_i'}
  \item {\pi _i}, {pi_i}
  \item {\tilde \phi _i}, {tilde-phi_i}

  \item {X _i^*}, {X_i*}
  \item {\hat \phi _i}, {hat-phi_i}
  \item {K _{\sigma _i}}, {K-sigma-i}
  \item {\sigma _i}, {sigma_i}
  \item {\tau _i}, {tau_i}

  \item G',{G'}
  \item K, K
  \item J, J
  \item J_+,{J+}
  \item \rho',{rho'}
  \item \sigma,{sigma}
  \item \tau,{tau}
  \item {r, \: s, \: \phi }, {r}
  \item {X^*,\: \hat\phi}, {X*}
  \item \tilde\phi, {tilde-phi}
  \item {K,\: K_\sigma,\: J,\: J_+}, K
  \item {\rho',\: \sigma,\: \tau,\: \pi}, {rho'}

  \item {\pi '_0}, {pi'_0}
  \item {\tilde \rho }, {tilde-rho}

\indexspace

  \item {\varepsilon (\phi ,\gamma )}, {eps_phi-gamma}

\indexspace

  \item {j (g), \: \jperp (g)}, {jg}
  \item {i (g),\: t(g),\: \iperp(g),\: \tperp(g)}, {ig}

\indexspace

  \item {\wtilde {\mathfrak {G}}(\phi , \gamma ), \: \mf G(\phi , \gamma )},
	{tG-phi-gamma}

\indexspace

  \item {\simm {i}}, {simm-i}

\indexspace

  \item {\mexp ^E_{x, t:u}}, {exp-E-xtu}
  \item {\mexp _{{{\mathbf T}}, x}}, {exp-Tx}

\end{trivlist}
\end{multicols}


\end{document}